\documentclass[12pt,a4paper]{book}
\usepackage[latin1]{inputenc}
\usepackage{amsmath}
\usepackage{amsfonts}
\usepackage{amssymb}
\usepackage{graphicx}
\usepackage[top=3cm, bottom=3cm, left=3cm, right=3cm]{geometry}
\usepackage {amsthm}
\usepackage{array}
\usepackage{enumerate}
\usepackage{enumitem}
\usepackage{multicol}
\usepackage{xcolor}
\usepackage{currfile}
\usepackage{hyperref}

\usepackage{fancyhdr}

\usepackage{tabto}

\newcommand{\N}{\mathbb{N}}
\newcommand{\Z}{\mathbb{Z}}
\newcommand{\Q}{\mathbb{Q}}
\newcommand{\R}{\mathbb{R}}
\newcommand{\C}{\mathbb{C}}

\newcommand{\D}{\mathbb{D}}
\newcommand{\tP}{\mathbb{P}}

\DeclareMathOperator{\li}{li}

\DeclareMathOperator{\tdiv}{div}

\newtheorem{thm}{Theorem}[chapter]
\newtheorem{lem}{Lemma}[chapter]
\newtheorem{conj}{Conjecture}[chapter]
\newtheorem{cor}{Corollary}[chapter]

\newtheorem{exa}{Example}[chapter]
\newtheorem{exe}{Exercise}[chapter]

\theoremstyle{definition}
\newtheorem{dfn}{Definition}[chapter]

\title{
\hrule
\hrule
\hrule
\vskip .10 in
\textbf{\LARGE Topics On The Correlations And Autocorrelations Of Arithmetic Functions} 
\vskip .10 in
\hrule
\hrule
\hrule
}

\author{\textbf{N. A. Carella}} 
\date{}

\begin{document}

\maketitle
\thispagestyle{empty}
\frontmatter



\tableofcontents	
\mainmatter

\chapter{Introduction} \label{c1}

Let $f,g:\mathbb{N} \longrightarrow \mathbb{C}$ be arithmetic functions, and let $\tau_{1}, \tau_{2},\ldots \tau_{k} \in \mathbb{N}$, and let $k \geq 1$ be fixed integers. The $k$-degree autocorrelation and correlation function of $f$ and $g$ are defined by 
\begin{equation}
R(\tau)=\sum_{n \leq x} f(n+\tau_{1})f(n+\tau_{2}) \cdots f(n+\tau_{k}).
\end{equation} 
and
\begin{equation}
C(\tau)=\sum_{n \leq x} f(n+\tau_{1})g(n+\tau_{1})g(n+\tau_{2})g(n+\tau_{1}) \cdots f(n+\tau_{k}) g(n+\tau_{k})
\end{equation} 
respectively. 
Some combinations of the parameters and functions are easy to evaluate but other are not. In particular, if $f(n)=\mu(n)$, the case $\tau_{1}=\tau_{2}= \cdots =\tau_{k}$, with $k \in 2\mathbb{Z}$, is usually not difficulty to estimate or calculate. But the case $\tau_{i}\neq \tau_{j}$ for some $i \neq j$ is usually a challenging problem. \\

Trivially, the autocorrelation function has the upper bound $|R(\tau)|\ll \left \| f \right \|^k x$, where
\begin{equation}
 \left \| f \right \|=\max_{n \in \N} \{ |f(n)| \}.
\end{equation}
And a priori, a random sequence $ f(n),f(n+1),f(n+3), \ldots $ is expected to have the upper bound $|R(\tau)|\ll  \left \| f \right \|^k x^{1/2}(\log x)^B$, where $B>0$ is a constant, see \cite{CS1999}, \cite{CS2000}, \cite[Theorem 2]{CS2002}, and \cite{AR2007}. Some extreme cases such that $|R(\tau)|\gg  \left \| f \right \|^k x$, are demonstrated in \cite{MS1998}. This note deals with the correlation functions for some arithmetic functions. The main results are Theorem \ref{thm1010.100} for the Mobius function, and Theorem \ref{thm1020.200} for the vonMangoldt function. \\

A new technique based on the log-correlation function is introduced in Chapter \ref{MF2332}. Applications to the vonMangoldt function are stated in two versions of this technique in Theorem \ref{thm9990A.950} and Theorem \ref{thm9990.950B}.

\section{Autocorrelations Of Mobius Functions} \label{S1010}
The study of the Mobius function and various forms of its autocorrelation is closely linked to the autocorrelations and distribution functions of binary sequences. The purpose is to improve the current theoretical results for the Mobius sequence $\{\mu(n): n\geq 1 \}$ and the Liouville sequence $\{\lambda(n): n\geq 1 \}$. Given a large number $x\geq 1$, these results imply that the short sequences $\{\mu(n),\mu(n+1), \ldots, \mu(n+r-1)\}$ and $\{\lambda(n),\lambda(n+1), \ldots, \lambda(n+r-1)\}$ of length $r=[x] \geq 1$ are well distributed and have 2-level correlation functions. The theory and recent results for the correlation functions of some binary sequences are investigated in \cite{CS2000}, et alii.\\

\begin{conj} \label{conj1010.100} Let $C>2$ be a constant, and let $k\geq 1$ be a small fixed integer. Then, for every large number $x>1$, 
\begin{equation}
\sum_{n \leq x} \mu(n+\tau_{1}) \mu(n+\tau_{2})\cdots\mu(n+\tau_{k})=O\left (\frac{x}{(\log x)^{C}} \right )
\end{equation}
for any fixed sequence of integers $\tau_{1}<\tau_{2}<\cdots<\tau_{k}$.
\end{conj}

The average order of the autocorrelations of Mobius functions have been proved. Specifically, 
\begin{equation}
\sum_{1 \leq \tau_{i} \leq T} \left |\sum_{x\leq n \leq 2x} \mu(n+\tau_{1}) \mu(n+\tau_{2})\cdots\mu(n+\tau_{k}) \right | \ll k \left (\frac{\log\log T}{\log T}+\frac{1}{\log^{1/3000} x}T^k x \right ),
\end{equation}
where $10 \leq T\leq x$, see \cite[p. 4]{EP1994}, and \cite{MR2015}.

\begin{conj} \label{conj1010.150} Let $C>2$ be a constant, and let $k\geq 1$ be a small fixed integer. Then, for every large number $x>1$,
	\begin{equation}
		\sum_{n \leq x} \lambda(n+\tau_{1}) \lambda(n+\tau_{2})\cdots\lambda(n+\tau_{k})=O\left (\frac{x}{(\log x)^{C}} \right )
	\end{equation}
	for any fixed sequence of integers $\tau_{1}<\tau_{2}<\cdots<\tau_{k}$.
\end{conj}
Other related works are the special case
\begin{equation}
\sum_{n \leq r}\mu(n) \mu(r-n)=O(x/(\log x)^B)
\end{equation} 
for almost every $r>1$, and $B>0$ constant, which is proved in \cite{DK2015}, and the functions fields versions in \cite{CD2015}, \cite{CR2014}, and \cite{MW2016}. Specifically, there is an explicit upper bound
\begin{equation}
 \left |\sum_{F \in M_{n}} \mu(F+f_{1}) \mu(F+f_{2})\cdots\mu(F+f_{k}) \right | \leq 2knq^{n-1/2}+3rn^2q^{n-1},
\end{equation}
where $f_i \in \mathbb{F}_{q}[x]$ are distinct fixed polynomials of degree $\deg(f_{i})<n$, and $M_n=\{F \in \mathbb{F}_{q}[x]:\deg(F)=n \}$ is the subset of polynomials of degree $\deg(f_{i})=n$, this appears in \cite{CR2014}.\\

The following case for $k=2$ is considered in this work. 
\begin{thm} \label{thm1010.100} Let $C>2$ be a constant, and let $k\geq 1$ be a small fixed integer. Then, for every large number $x>1$, 
	\begin{equation}
		\sum_{n \leq x} \mu(n) \mu(n+1)=O\left (\frac{x}{(\log x)^{C}} \right ).\nonumber
	\end{equation}

\end{thm}

The first proof of this result is assembled in Section \ref{S474R}, and the second proof of this result is assembled in Section \ref{S474B}.

\section{Autocorrelations Of vonMangoldt Functions} \label{S1020}
The correlation of the vonMangoldt function is a long standing problem in number theory. It is the foundation of many  related conjectures concerning the primes pairs such as $p$ and $p+2k$ as $p \to \infty$, where $k \geq 1$ is fixed, and the prime $m$-tuples $p, p+a_1, p+a_2, \ldots, p+a_m$, where $a_1,a_2, \ldots, a_m$ is an admissible finite sequence of relatively prime integers. The most recent progress in the theory of the correlation of the von Mangoldt function are the various results based on the series of papers authored by \cite{GP2005}. These authors employed an approximation $\Lambda_R(n)$ of the function $\Lambda(n)$ to obtain the estimate
\begin{equation}
 \sum_{n \leq x} \Lambda_R(n)\Lambda_R(n+k) =\mathfrak{G}(k)x +O \left (\frac{k}{\varphi(k)} \frac{x}{(\log 2R/k)^{C}}\right ) +O \left (R^2\right ),
\end{equation}
confer Theorem \ref{thm773.15} for the precise details. The density is defined by the singular series
\begin{equation} \label{eq773.52}
\mathfrak{G}(h)=2\prod_{p \mid h}\frac{p-1}{p-2} \prod_{n >2}\left( 1-\frac{1}{(p-1)^2} \right ).
\end{equation}
Another result in \cite[Theorem 1.3]{MR2017} has a claim for the average Hardy-Littlewood conjecture, it covers almost all shift of the autocorrelation, but no specific value. A synopsis of the statement is recorded below.

\begin{equation} \label{eq1020.150}
\sum_{n \leq x} \Lambda(n)\Lambda(n+h) =\mathfrak{G}(h)x +O \left (\frac{x}{(\log x)^{C}}\right ) 
\end{equation}
for all but $O \left (H/\log x^{C} \right )$ values $h$ with $|h-h_0| \leq H$, confer Theorem \ref{thm773.18} for the precise details. The function fields version of these von Mangoldt correlations results are proved in \cite{KR2015}.\\

A simpler and different approach to the proof of the autocorrelation of the vonMangoldt function, which is actually valid for any fixed small value, is considered here. 
\begin{thm} \label{thm1020.200} Let $x \geq 1$ be a large number, and let $C>2$ be a constant. Then, for any fixed integer $k \geq 1$,
\begin{equation} \label{eq1020.200}
\sum_{n \leq x} \Lambda(n)\Lambda(n+2k) =a_kx +O \left (\frac{x}{(\log x)^{C/2}}\right ), \nonumber
\end{equation}
where the density $a_k=\mathfrak{G}(2k)$.
\end{thm}

The proof of Theorem \ref{thm1020.200} is the same as Theorem \ref{thm474.100}, which appears in Section \ref{s773}. The new and interesting 3-level autocorrelation is described in Theorem \ref{thm773.8}.\\

\chapter{ Mobius Function}\label{MF2222}

\section{ Representations of Liouville and Mobius Functions} \label{S9700}
The symbols $\mathbb{N} = \{ 0, 1, 2, 3, \ldots \}$ and $\mathbb{Z} = \{ \ldots, -3, -2, -1, 0,  1, 2, 3, \ldots \}$ denote the subsets of integers. For $n \in \mathbb{N}$, the prime divisors counting functions are defined by
\begin{equation}
	\omega(n)=\sum_{p \mid n} 1 \qquad \text{ and } \qquad   \Omega(n)=\sum_{p^v \mid n} 1
\end{equation}
respectively. The former is not sensitive to the multiplicity of each prime $p \mid n$, but the latter does count the multiplicity of each prime $p \mid n$. \\

For $n \geq 1$, the quasiMobius function $\mu_{*}:\mathbb{N} \longrightarrow \{-1,1\}$ and the Liouville function $\lambda:\mathbb{N} \longrightarrow \{-1,1\}$ are defined, in terms of the prime divisors counting functions, by
\begin{equation}
	\mu_{*}(n)=(-1)^{\omega(n)}  \qquad \text{ and } \qquad  \lambda(n)=(-1)^{\Omega(n)}
\end{equation}
respectively. In addition, the Mobius function $\mu:\mathbb{N} \longrightarrow \{-1,0,1\}$ is defined by
\begin{equation}
	\mu(n) =
	\left \{
	\begin{array}{ll}
		(-1)^{\omega(n)}     &n=p_1 p_2 \cdots p_v\\
		0           &n \ne p_1 p_2 \cdots p_v,\\
	\end{array}
	\right .
\end{equation}
where the $p_i\geq 2$ are primes. The Mobius function is quasiperiodic. It has a period of 4, that is, $\mu(4)= \cdots =\mu(4m)=0$ for any integer $m \in \Z$. But its interperiods values are pseudorandom, that is, $\mu(n),\mu(n+4), \ldots, \mu(n+4m) \in \{-1,0,1\}$ is not periodic as $n \to \infty$. In contrast, the Liouville function is antiperiodic, since $\lambda(pn)=-\lambda(n)$ for any fixed prime $p \geq 2$ as $n \to \infty$. \\

The quasiMobius function and the Mobius function coincide on the subset of squarefree integers. From this observation arises a pair of fundamental identities.

\begin{lem} \label{lem297.51} For any integer $n \geq 1$, the Mobius function has the expansions
	\begin{equation}
		\mu(n)= (-1)^{\omega(n)} \mu^2(n)  \qquad \text{ and } \qquad \mu(n)= (-1)^{\Omega(n)} \mu^2(n).
	\end{equation}
\end{lem}

\begin{lem} \label{lem297.82} For any integer $n \geq 1$, the quasi Mobius function has the expansion
	\begin{equation}
		(-1)^{\omega(n)}= \sum_{t \mid n} \mu(t)d(t),
	\end{equation}
	where $d(n)=\sum_{d \mid n}1$ is the number of divisors function.
\end{lem}
\begin{proof}[\textbf{Proof}] It is sufficient to prove it at the prime powers $p^k$, where $k \geq 1$, see \cite[p.\ 473] {RD1996} for more details. \end{proof}

The characteristic function for squarefree integers is closely linked to the Mobius function. 
\begin{lem} \label{lem297.58} For any integer $n \geq 1$, the characteristic function for squarefree integers has the expansion
	\begin{equation}
		\mu(n)^2= \sum_{d^2 \mid n} \mu(d).
	\end{equation}
\end{lem}

\begin{lem} \label{lem297.93} For any integer $n \geq 1$, the Liouville function has the expansion
	\begin{equation}
		\lambda(n)= \sum_{d^2 \mid n} \mu(n/d^2).
	\end{equation}
\end{lem}

\begin{proof}[\textbf{Proof}]  Observe that $\lambda(n)$ is a completely multiplicative function, so it is sufficient to verify the claim for prime powers $p^v, v \geq 1$, refer to 
	\cite[p.\ 50]{AP1976}, and \cite[p.\ 471]{RD1996}, for similar information.   \end{proof}

\begin{lem} \label{lem2.4} For any integer $n \geq 1$, the Mobius function has the expansion
	\begin{equation}
		\mu(n)= \sum_{d^2 \mid n} \mu(d)\lambda(n/d^2).
	\end{equation}
\end{lem}

\begin{proof}[\textbf{Proof}]  Exercise, refer to \cite{AP1976}, and \cite{RD1996}, for similar information.   
\end{proof}

\section{Dyadic Representation}
The Mobius function has a Hyperbola method type identity. This identity is uselful in various summation methods.

\begin{lem} {\normalfont (Dyadic representation)} \label{lemRMF525.100} Let $y\geq1$ and $z\geq1$ be a pair of fixed parameters, and let $n \in \N$ be an integer. Then 
	\begin{equation}  \label{eqRMF525.100}
		\mu(n)=-\sum_{\substack{d \leq y, \\ de \mid n}} \sum_{e \leq z}\mu(d) \mu(e) +\sum_{\substack{d > y, \\ de \mid n}} \sum_{e >z}\mu(d) \mu(e) .
	\end{equation}
\end{lem} 

\begin{proof} Since $\mu(mn)=\mu(m)\mu(n)$ for relatively prime integers, that is $\gcd(m,n)=1$, and $\mu(mn)=0$ if $\gcd(m,n)>1$, it is sufficient to verify the identity for prime powers $p^k$, with $k \geq 1$. Take $n=p^k$ and $y=z=p$, then
	\begin{eqnarray}  \label{eqRMF525.110}
		\mu(p^k)&=&-\sum_{\substack{d \leq p, \\ de \mid p^k}} \sum_{e \leq p}\mu(d) \mu(e) +\sum_{\substack{d > p, \\ de \mid p^k}} \sum_{e >p}\mu(d) \mu(e) \nonumber\\
		&= &-\mu(1) \mu(1)-\mu(1)\mu(p)-\mu(p)\mu(1) -\mu(p)\mu(p) +\mu(p^2)\mu(p^2)+ \cdots \nonumber \\
		&= & 
		\begin{cases}
			-1& \mbox{if }  k =1,\\
			0& \mbox{if } k\geq 2.
		\end{cases}
	\end{eqnarray}
	As the divisors $d,e \geq p^2$ occur in the second double sum, and $\mu(p^k)=0$ if $k \geq 2$, it does not contribute to the total.
\end{proof}
Additional detail are given in Proposition 13.5, and applications Proposition 13.2 appear in \cite{IK2004}.

\section{Explicit Formula Representation}
The explicit representation which has a dependence on the zeros of the zeta function, is derived via the explicit version of the Mertens function in Theorem see Theorem \ref{thmSFMF222.500}. The exact formula for the Mertens function is given in \cite[p. 318]{TE1986}, \cite{BK1991}, et alii.
\begin{thm} {\normalfont (\cite{BK1991})} \label{thmMF222.500} Assume the zeros of the zeta function are simple. Then, there is a sequence $T_n$ $2^{n-1}T_0 \leq T_n \leq 2^nT_0$, where $T_0>0$ an absolute positive constant, and $n\geq1$, such that
	\begin{equation}\label{eqMF222.510}
		\sum_{n\leq x}\mu(n)=-2+\lim_{n\to \infty}	\sum_{| \Im m\rho|\leq T_n} \frac{x^{\rho}}{\rho\zeta^{\prime}(\rho) }+	\sum_{n\geq 1}\frac{x^{-2n}}{-2n\zeta^{\prime}(-2n) },
	\end{equation}
	where $\rho\in\C$ varies over the critical zeros of the zeta function, and $x\notin \N$. 
\end{thm} 

\begin{lem} {\normalfont (Explicit representation)} \label{lemRMF525.200} Let $x\geq1$ be a large number. If $n \in (x,x+1)$ is an integer, then
	\begin{enumerate}  [font=\normalfont, label=(\roman*)]
		\item $\displaystyle \mu(n)= \frac{1}{i2\pi } \int_{c- i\infty}^{c+i\infty} \frac{1}{\zeta(s)}  \frac{(x+1)^{s}-x^{s}}{s} ds,$\\
		
		\item $\displaystyle \mu(n)=  \sum_{\rho} \frac{(x+1)^\rho -x^\rho}{\rho\zeta^{\prime}(\rho) } + O(1),$		
	\end{enumerate}
	where $c>1$  is a constant, and the index $\rho\in \C $ ranges through the nontrivial zeros of the zeta function $\zeta(s)$, as $x \to \infty$.
\end{lem} 
\begin{proof}(ii) The Perron formula provides the asymptotic 
	\begin{eqnarray}\label{eqRMF525.210}
		\sum_{n\leq x}\mu(n) &=& \frac{1}{i2\pi } \int_{c- i\infty}^{c+i\infty} \frac{1}{\zeta(s)}  \frac{x^{s}}{s} ds\\
		&=&\sum_{|\rho|\leq T} \frac{x^\rho}{\rho\zeta^{\prime}(\rho) } + O\left(\frac{x(\log Tx)^2}{T} \right) \nonumber,
	\end{eqnarray}
	where $c>1$ is a constant, and $T>0$, see \cite[Section 5.1]{MV2007}, et alii. Hence, the difference
	\begin{equation}\label{eqRMF525.220}
		\mu(n+1)=\sum_{n\leq x+1}\mu(n)-\sum_{n\leq x}\mu(n)
	\end{equation}	
	whenever $n+1\in(x,x+1)$.
\end{proof}

\section{Generalized Definition}
\begin{dfn} \label{dfn297.25}{\normalfont An integer $n \in \N$ is said to be $s$-power free if for each prime $p \mid n$, the maximal prime power divisor is $p^{s-1} \mid \mid n$. Equivalently, the $p$-\textit{adic} valuation $v_p(n)=s-1$ for any $s\geq 2$.
}
\end{dfn}
The $2$-free integers are usually called squarefree integers.
\begin{dfn}  \label{dfn297.30}{\normalfont The characteristic function for $s$-power free integers is defined by
	\begin{equation}
		\mu_s(n) =
		\left \{
		\begin{array}{ll}
			1     &\text{if } p^s \nmid n \text{ for any prime } p \mid n,\\
			0           &\text{if } p^s \mid n \text{ for any prime } p \mid n.\\
		\end{array}
		\right .
	\end{equation}
}
\end{dfn}

The characteristic function for $s$-power free integers is closely linked to the Mobius function. 
\begin{lem} \label{lem297.58B} For any integer $n \geq 1$, the characteristic function for $s$-power free integers has the expansion
\begin{equation}
	\mu_s(n)= \sum_{d^s \mid n} \mu(d).
\end{equation} 
\end{lem}
The case $s=2$ for squarefree integers is usually denoted by $\mu^2(n)= \sum_{d^2 \mid n} \mu(d)$, see Lemma \ref{lem297.58}. Some early works on this topic appear in \cite{CL1932} and \cite{ML1947}. 

\begin{dfn}  \label{dfn297.35}{\normalfont A pair of integers $a$ and $q$ are relatively prime if and only if $\gcd(a,q)=1$. The characteristic function for $s$-power free, and relatively prime integers is defined by
	\begin{equation}
		\sum_{\substack{d^{s-1} \mid a \\ d^{s-1} \mid q}} \mu_s(d) =
		\left \{
		\begin{array}{ll}
			1     &\text{if and only if } \gcd(a,q)=1,\\
			0           &\text{if and only if } \gcd(a,q)\ne 1.\\
		\end{array}
		\right .
	\end{equation}
}
\end{dfn}
This indicator function is widely used to remove the relatively prime dependence in many applications.

\section{Inversion Formulas}
\begin{lem} {\normalfont Mobius inversion formula}\label{lem297.88} Let $f,g: \N \longrightarrow \N$ be arithmetic functions, and $n \geq 1$ an integer. Then
	\begin{enumerate} [font=\normalfont, label=(\roman*)]
		\item  $\displaystyle f(n)=\sum_{d\mid n}g(d)    \qquad \text{ and }  \qquad g(n)=\sum_{d\mid n}\mu(d) f(n/d)$\\is an additive inverse pair.
		\item  $\displaystyle f(n)=\prod_{d\mid n}g(d)     \qquad \text{ and }  \qquad g(n)=\prod_{d\mid n} f(n/d)^{\mu(d)}$ \\
		is a multiplicative inverse pair.
	\end{enumerate}
\end{lem}
\begin{lem} \label{lem297.98} Let \(n\geq 1\) be an integer, and let $\delta>0$ be a small real number. Then, 
\begin{enumerate} [font=\normalfont, label=(\roman*)]
	\item  
	$\displaystyle 
	\sum_{d\mid n}\mu(n) =
	\left \{
	\begin{array}{ll}
		1     &\text{if } n=1,\\
		0           &\text{if } n\ne 1.\\
	\end{array}
	\right .
	$
	\item 
	$\displaystyle
	\sum_{d\mid n}\mu^2(n) =2^{\omega(n)}.  $
	\item 
	$\displaystyle
	\sum_{d\mid q}|\mu(d)|=O\left (q^{\delta}\right ). 
	$
	
\end{enumerate}
\end{lem}

\section{ Problems} \label{EMF666A}

\begin{exe}\label{exeEMF666.100} {\normalfont Verify the Mellin transform formula for $M(x)$. If $\zeta(s)$ is the zeta function, and $\Gamma(s)$ is the gamma function, then the Mellin transform of the ratio $\Gamma(s)/\zeta(s)$ yields the Mertens function
		$$M(x) = \frac{1}{i2\pi } \int_{c- i\infty}^{c+i\infty} \frac{1}{\zeta(s)}  \frac{\Gamma(s)}{x^{s}} ds,$$
		where $c>1$  is a constant. 
	}
\end{exe}

\begin{exe}\label{exeEMF666.105} {\normalfont Use the Perron formula to verify the approximate explicit formula for $M(x)$. If the complex number $\rho\in \C $ ranges through the nontrivial zeros of the zeta function $\zeta(s)$, then
		$$\sum_{n \leq x} \mu(n) = \sum_{\rho} \frac{x^\rho}{\rho\zeta(\rho) } + O(1),$$
		as $x \to \infty$.
	}
\end{exe}

\chapter{ Summatory Mobius Function}\label{MF2222B}

A variety of results for the average orders of the Mobius function are stated in this chapter.  \\

\section{Standard Results for the Mobius Function}\label{SMF222}
There are many sharp bounds of the summatory function of the Mobius function, say, $O(xe^{-\sqrt{\log x}})$, and the conditional estimate $O(x^{1/2+\varepsilon})$ presupposes that the nontrivial zeros of the zeta function $ \zeta(\rho)=0$ in the critical strip $\{0<\Re e(s)<1 \}$ are of the form $\rho=1/2+it, t \in \mathbb{R}$. However, the simpler notation will be used whenever it is convinient.

\begin{thm} \label{thmMF222.050} If $\mu: \N\longrightarrow \{-1,0,1\}$ is the Mobius function, then, for any large number $x>1$, the following statements are true.
	\begin{enumerate} [font=\normalfont, label=(\roman*)]
		\item $\displaystyle \sum_{n \leq x} \mu(n)=O \left (x^{-c\sqrt{\log x}}\right )$, where $c>0$ is an absolute constant, unconditionally,
		\item $\displaystyle \sum_{n\leq x}\mu(n)=O\left( x^{1/2+\varepsilon} \right ), $ where $\varepsilon>0$ is an arbitrarily small number, conditional on the RH.
	\end{enumerate}
\end{thm}
\begin{proof}[\textbf{Proof}]  See \cite[p.\ 6]{DL2012}, \cite[p.\ 182]{MV2007}, \cite[p.\ 347]{HW2008}, et alii.   
\end{proof}

There are sharper bounds, say, $O(xe^{-\sqrt{\log x}})$, but the simpler notation will be used here. And the conditional estimate $O(x^{1/2} \log x)$ presupposes that the nontrivial 
zeros of the zeta function $ \zeta(\rho)=0$ in the critical strip $\{0<\Re e(s)<1 \}$ are of the form $\rho=1/2+it, t \in \mathbb{R}$. Moreover, the explicit upper bounds are 
developed in \cite{BO2015}.\\

The standard proof for the summatory Mobius function over an arithmetic progression is linked to the Siegel-Walfisz Theorem for primes in arithmetic progressions, the upper bounds are proved or discussed in 
\cite[p.\ 424]{IK2004}, \cite[p.\ 385]{MV2007}, et alii. 

\begin{thm} \label{thmMF222.090} Let $x\geq1$ be a large number, and let $q\ll(\log x)^B$, where $B\geq0$ is an arbitrary constant. If $1\leq a<q$ are relatively prime integers, then, 
	\begin{equation}
		\sum_{\substack{n \leq x\\n\equiv a \bmod q}} \mu(n)=O \left (\frac{x}{\log^{C}x}\right )\nonumber, 
	\end{equation}
	where $C=C(B)>0$ is a constant. 
\end{thm}
\begin{proof}[\textbf{Proof}]A sketch of the proof appears in \cite[p.\ 385]{MV2007}.
\end{proof}


The Mobius function over an arithmetic progression is linked to the Siegel-Walfisz Theorem for primes in arithmetic progressions, it has the upper bound given below, see 
\cite[p.\ 424]{IK2004}, \cite[p.\ 385]{MV2007}. \\

\begin{thm} \label{thm3.2} Let $a,q$ be integers such that  $gcd(a,q)=1$. If $C>0$ is a constant, and $\mu$ is the Mobius function, then, for any large number $x>1$,\\
	\begin{enumerate} [font=\normalfont, label=(\roman*)]
		\item $ \displaystyle \sum_{\substack{n \leq x\\ n\equiv a \bmod q}} \mu(n)=O \left (\frac{x}{\log^{C}x}\right ) .$    
		\item $ \displaystyle   \sum_{\substack{n \leq x\\ n\equiv a \bmod q}}\frac{\mu(n)}{n}=O \left (\frac{1}{\log^{C}x}\right ). $ 
	\end{enumerate}
\end{thm}
A few results for the average orders over short intervals are proved in the literature. For the new developments, confer \cite[Theorem 1.1]{MT2019}. 
\begin{thm} \label{thmMF222.080} Let $C>0$ be a constant, and let $\theta>7/12$ be a small real number. If $x>1$ is a large number, and $H\geq x^{\theta}$, then
	\begin{equation}
		\sum_{x\leq n \leq x+H} \mu(n)=O \left (\frac{H}{\log^{C}x}\right )\nonumber. 
	\end{equation}
\end{thm}

\section{Explicit Formula for the Mertens Function}\label{SFMF222}
The Perron formula applied to the generating series of the Mobius function produces a representation which has a dependence on the zeros of the zeta function. This is derived from an explicit version of the Mertens function. The exact formula for the Mertens function is given in \cite[p. 318]{TE1986}, \cite{BK1991}, et alii.
\begin{thm} {\normalfont (\cite{BK1991})} \label{thmSFMF222.500} Assume the zeros of the zeta function are simple. Then, there is a sequence $T_n$ $2^{n-1}T_0 \leq T_n \leq 2^nT_0$, where $T_0>0$ an absolute positive constant, and $n\geq1$, such that
\begin{equation}\label{eqSFMF222.510}
\sum_{n\leq x}\mu(n)=-2+\lim_{n\to \infty}	\sum_{| \Im m\rho|\leq T_n} \frac{x^{\rho}}{\rho\zeta^{\prime}(\rho) }+	\sum_{n\geq 1}\frac{x^{-2n}}{-2n\zeta^{\prime}(-2n) },
\end{equation}
where $\rho\in\C$ varies over the critical zeros of the zeta function, and $x\notin \N$. 
\end{thm} 

\section{ Randomness and Correlation Results}\label{AMF222}
A few correlation results for the Mobius function are stated in this section.  \\

\begin{conj}\label{conjMF222.190}
	\normalfont{ (Mobius Randomness Law) } For a bounded function $f: \N \longrightarrow \C$, the partial sum 
	\begin{equation}
		\sum_{n \leq x} \mu(n) f(n)=o(x)
	\end{equation}
	is small relative to a large number $x \geq 1$.
\end{conj}

\begin{exa} \label{exaMF222.190}{\normalfont This example shows that the \textit{Mobius randomness law} works for some bounded functions. For an irrational number $\alpha >0$,
		$$\sum_{n \leq x} \mu(n) e^{i \alpha \pi n}=o(x).$$}
\end{exa}

\begin{exa}  \label{exaMF222.195}{\normalfont This example shows that the \textit{Mobius randomness law} fails for unbounded functions. For a large number $x \geq 1$,
		$$\sum_{n \leq x} \mu(n) \Lambda(n)=-\sum_{p \leq x} \log p=-x+O(x/\log x),$$
		where the last sum ranges over the primes $p \leq x$.}
\end{exa}

\section{Mean Value And Equidistribution}\label{MMF525}
An equidistribution result for arithmetic functions over arithmetic progressions of level of distribution $\theta<1/2$  is recorded here. This result is applicable to a wide range of complex valued arithmetic functions $f: \N \longrightarrow \C$.

\begin{thm} \label{thmMMF525.500} {\normalfont (\cite[Theorem 1]{WD1973})} Let $x \geq 1$ be a large number. Let $f: \N \longrightarrow \C$ be a multiplicative number theoretical function such that $f(n)\ne \chi(n)$, where $\chi$ is a Dirichlet character, and satisfies the followings conditions.
	\begin{enumerate} [font=\normalfont, label=(\roman*)]
		\item $ \displaystyle \left |
		f(p^v) \right | \leq c_1v^{c_2}$ for all prime powers $p^v$, $v \geq 1$, and a pair of constants $c_1, c_2>0$.    
		\item $ \displaystyle   \sum_{p \leq x} \left | f(p)+\tau \right | =O \left (\frac{x}{\log^{c_3}x}\right ) $ for $\tau \in \R$, and $c_3>0$. 
	\end{enumerate}
	Then, for any constant $C>0$, there exists a constant $B=B(C, c_3, f)>0$, for which
	\begin{equation}\label{eqMMF525.500}
		\sum_{q\leq x^{1/2}/\log^B x} \max_{a \bmod q}\max_{z \leq x}  \bigg |\sum_{\substack{n \leq z \\
				n \equiv a \bmod q}} f(n) - \frac{a(\tau)}{\varphi(q)}\sum_{\substack{n \leq z \\
				n \equiv a \bmod q}
		} f(n)\bigg |\ll \frac{x}{(\log x)^{C}},
	\end{equation} where $a(\tau)=0$ if $\tau \leq 0$ otherwise $a(\tau)=1$ for $\tau >0$.
\end{thm}

\begin{cor} \label{corMMF525.550}  Let $a\geq 1$ be a fixed parameter, and let $x \geq 1$ be a large number. If $C>0$ is a constant, then
	\begin{equation}
		\sum_{q\leq x^{1/2}/\log^B x} \max_{a \bmod q}\max_{z \leq x}  \bigg |\sum_{\substack{n \leq z \\
				n \equiv a \bmod q}} \mu(n) \bigg |\ll \frac{x}{(\log x)^{C}},
	\end{equation}where the constant $B>0$ depends on $C$.
\end{cor}
\begin{proof} Let the parameter $\tau=0$, and let $f(n)=\mu(n)$. A routine calculation shows that 
	\begin{equation}\label{eqMMF525.560}
		f(nm)=\mu(mn) =\mu(m)\mu(n)=f(m)f(n)  
	\end{equation}
	is multiplicative for all integers $m,n \in \{n \equiv a \bmod q : n \in N \}$ such that $\gcd(m,n)=1$. Moreover, the conditions 
	\begin{enumerate} [font=\normalfont, label=(\roman*)]
		\item $ \displaystyle  \left |
		f(p^v) \right |=\left |
		\mu(p^v) \right |\leq 1$ for all prime powers $p^v$, $v \geq 1$, and any $a\geq 1$;    
		\item $ \displaystyle   \sum_{p \leq x} \left | f(p)+\tau \right | =\sum_{p \leq x} \left | \mu(p) \right | =O \left (\frac{x}{\log x}\right ) $ for all large $x \geq 1$, and any $a\geq 1$; 
	\end{enumerate}
	are valid. Thus, an application of Theorem \ref{thmMMF525.500} completes the verification.
\end{proof}
This corollary was previously proved in \cite{SW1971}. 

\begin{cor} \label{corMMF525.580}  Let $x \geq 1$ be a large number, and let $f(n)=O(1)$ be a bounded multiplicative function. If $C>0$ is a constant, and $f(n) \ne (-1)^{\Omega(n)}$, and $f(n)\ne \chi(n)$, where $\chi$ is a Dirichlet character, then
	\begin{equation}
		\sum_{q\leq x^{1/2}/\log^B x} \max_{a \bmod q}\max_{z \leq x}  \bigg |\sum_{\substack{n \leq z \\
				n \equiv a \bmod q}} (-1)^{\Omega(n)}f(n) \bigg |\ll \frac{x}{(\log x)^{C}},
	\end{equation}where the constant $B>0$ depends on $C$.
\end{cor}
\begin{proof} Let the parameter $\tau=0$. Since both $f(n)$ and $(-1)^{\Omega(n)}$ are multiplicative, the product
	\begin{equation}\label{eqMMF525.580}
		g(mn)=(-1)^{\Omega(mn)}f(mn) =(-1)^{\Omega(m)+\Omega(n)}f(m)f(n) =g(m)g(n)  
	\end{equation}
	is also multiplicative for any $m,n \in \N$, $\gcd(m,n)=1$. Moreover, the conditions 
	\begin{enumerate} [font=\normalfont, label=(\roman*)]
		\item $ \displaystyle  \left |
		g(p^v) \right |=\left |
		(-1)^{\Omega(p^v)}g(p^v) \right |\ll 1$ for all prime powers $p^v$, and $v \geq 1$;   
		\item $ \displaystyle   \sum_{p \leq x} \left | g(p)+\tau \right | =\sum_{p \leq x} \left | (-1)^{\Omega(p)}f(p) \right | =O \left (\frac{x}{\log x}\right ) $ for all large $x \geq 1$; 
	\end{enumerate}
	are valid since $f(n) \ne (-1)^{\Omega(n)}$. Thus, assuming both $f(n) \ne (-1)^{\Omega(n)}$ and $f(n)\ne \chi(n)$, where $\chi$ is a Dirichlet character, an application of Theorem \ref{thmMMF525.500} completes the verification.
\end{proof}
\subsection{ Bombieri-Vinogradov Type Results} \label{AMF525}

There are other approaches through the generalized Bombieri-Vinogradov type theorem, which provides the summation
\begin{equation}\label{eqAMF525.200}
	\sum_{\substack{n \leq x\\ n\equiv a \bmod q}} \mu(n)=\frac{1}{\varphi(q)}\sum_{\substack{n \leq x\\ \gcd(n,q)=1}} \mu(n)+ O \left (\frac{\sqrt{q}x}{(\log x)^{C}}\right ),
\end{equation} 
confer \cite[p.\ 40]{EP1994} for more details.  \\

\begin{thm} \label{thmMMF525.300} If $x\geq 1$ is a large number, then, for any constant $C>0$, 
	\begin{equation}\label{eqMMF525.300}
		\sum_{q\leq Q}\max_{1\leq a <q }\bigg |\sum_{\substack{n\leq x\\n\equiv a\bmod q}}\mu(n)\bigg |\ll \frac{x}{(\log x)^{C}} ,\nonumber
	\end{equation}
	where $Q=x^{1/2}(\log x)^{-B}$, and both the constant $B=B(A)$ and the implied constant depend on $C$. 
\end{thm}
\begin{proof} The proof appears in Chapter 17, \cite{IK2004}.
\end{proof}

\section{ Moments of the Mobius Function} \label{MF222}
The moments of the Mobius function are very simple to evaluate, in fact, there just two asymptotic formulas for all the moments.

\begin{thm} \label{thmMF222.750} Let $\mu(n)\in \{-1,0,1\}1$ be the Mobius function. Then, for any sufficiently large number $x>1$, and $m$th moment is given by 
	\begin{enumerate} [font=\normalfont, label=(\roman*)]
		\item  $ \displaystyle  \sum_{n\leq x} \mu^{2k+1}(n)=O \left (x^{-c\sqrt{\log x}}\right ), $ 
		\item $ \displaystyle  \sum_{p\leq x} \mu^{2k}(n)=(6/\pi^2)x+O \left (x^{1/2}\right ), $ 
	\end{enumerate}
	where $k\geq0$ is an integer, and $c>0$ is an absolute constant.  
\end{thm}
\begin{proof}[\textbf{Proof}] (i) Since $\mu^{2k+1}(n)=\mu(n)$ for $k\geq0$, this follows from the unconditional part of Theorem \ref{thmMF222.050}. (ii) Since $\mu^{2k}(n)=\mu^2(n)$ for $k\geq1$, this follows from the Theorem for squarefree integers. 
\end{proof}

The sequence of moments 

\begin{equation} \label{eq132.00}
	M_k(x)=\sum_{n \leq x} \mu^k(n)=
	\begin{cases} \displaystyle c_0 x+o(x^{1/2}) & k=\text{even},\\
		\displaystyle O \left ( xe^{-c\sqrt{\log x}} \right ) &k=\text{odd},
	\end{cases}
\end{equation}
where $c_0=6/\pi^2$, $c>0$ is an absolute constant, and the quadratic identity
\begin{equation} \label{eq132.02}
	\sum_{n \leq x} \left (\mu(n)+\mu(n+t) \right )^2- \sum_{n \leq x}\mu(n)^2-\sum_{n \leq x} \mu(n+t)^2=2\sum_{n \leq x} \mu(n)\mu(n+t) 
\end{equation}
suggests that the autocorrelation of the Mobius function
\begin{equation} \label{eq132.10}
	\sum_{n \leq x} \mu(n) \mu(n+t)=
	\begin{cases} \displaystyle c_0 x+o(x^{1/2}) & t=0,\\
		\displaystyle O \left ( xe^{-c\sqrt{\log x}} \right ) &t\ne 0,		\\
\displaystyle \Omega \left ( x^{1/2} \right ) &t\in \Z,
	\end{cases}
\end{equation}
where $c_0=6/\pi^2$. The big omega result in the last item has not been proved yet. Currently, the above two-value autocorrelation function is unconditional. In comparison, the conditional version is approximately
\begin{equation} \label{eq132.12}
	\sum_{n \leq x} \mu(n) \mu(n+t)=
	\begin{cases} \displaystyle c_0 x+O(x^{1/4+\varepsilon}) & t=0,\\
		\displaystyle O \left ( x^{1/2+\varepsilon} \right ) &k\ne 0,	\\
\displaystyle \Omega \left ( x^{1/2} \right ) &t\in \Z,
	\end{cases}
\end{equation}
where $\varepsilon>0$ is an arbitrarily small number. 
\section{ Variances of the Mobius Function} \label{VMF666}

\begin{thm} \label{thmVMF666.300} Let $\mu(n)\in\{-1,0,1\}$ be the Mobius function. Then, for any sufficiently large number $x>1$, the variance satisfies the asymptotic
	\begin{equation}
		\sum_{n\leq x}\left | \mu(n)-0 \right |^2=\frac{6}{\pi^2}x+O \left (x^{1/2}\right )\nonumber. 
	\end{equation}
\end{thm}

\section{ Problems} \label{EMF666}

\subsection{Constants and Error Terms of Mobius Autocorrelation $L$ Series}
\begin{exe} \label{exeEMF666.110}{\normalfont Let $\mu(n)\in \{-1,0,1\}$ be the Mobius function. Estimate the constant $c\in \R$, and the error of the following sums.
		\begin{enumerate}
			\item $\displaystyle \sum_{n\leq x} \frac{\mu(n)\mu(n+1)}{n^2} =c+E(x),$
			\item $\displaystyle \sum_{n> x} \frac{\mu(n)\mu(n+1)}{n^2} =E(x).$
			\item $\displaystyle E(x)\overset{?}{=}O\left (xe^{-\sqrt{\log x}}\right ).$ Comment on whether or not this is this an open problem.
		\end{enumerate}
	}
\end{exe}

\subsection{Explicit Estimates Problems}
\begin{exe} {\normalfont Find an analytic method to show that
$$
\sum_{n \leq x} \frac{\mu(n)}{n}<1 \quad \text{          and          } \quad \sum_{n \leq x} \frac{\lambda(n)}{n}<2. 
$$
These are simple explicit estimates; other sharper explicit estimates of the forms $\sum_{n \leq x} \frac{\mu(n)}{n}<1/\log x$ are proved in \cite{BO2015}.
}
\end{exe}

\subsection{Representations Problems}
\begin{exe} {\normalfont Show that the Mobius function representation is not unique; there are other representations:
$$\mu(n)= (-1)^{\omega(n)} \mu^2(n)=\mu(n)= (-1)^{\Omega(n)} \mu^2(n).$$
}
\end{exe}
\begin{exe} {\normalfont Find an expansion for the Liouville function $$(-1)^{\Omega(n)}\overset{?}{=} \sum_{q \mid n} \lambda(q)d(q) \qquad \text{ analogous to } \qquad (-1)^{\omega(n)}= \sum_{q \mid n}\mu(q)d(q).$$
}
\end{exe}
\begin{exe} {\normalfont Use a result for infinitely many pairs of primes bounded by a constant to show that the Liouville  function (or Mobius function) changes sign infinitely often.
}
\end{exe}


\chapter{Exponential Sums and Multiplicative Coefficients  } \label{S1155}
Exponential sums $\sum_{n \leq x} f(n) e^{i2 \pi \alpha n},$ where $0<\alpha<1$ is a real number, with the multiplicative coefficient $f(n)=\mu(n)$ and $f(n)=\lambda(n)$ are treated here. The earlier results on these exponential sums appears to be $\sum_{n \leq x} \mu(n) e^{i2 \pi \alpha n} =O(x(\log x)^{-c}), c>0,$ in \cite{DH1937}. Recent contributions appear in \cite{DD1974}, \cite{MV1977}, \cite{HS1987}, \cite{BR1999}, \cite{BG2003}, \cite{MS2011}, \cite{MM2015}, et alii.\\

\section{Survey of Results}
One of the earliest result for twisted sums is stated below, and the recent version  over short intervals appears in \cite[Theorem 1.5]{MT2019}.
\begin{thm} \label{thm3970M.300} {\normalfont (\cite{DH1937})} If $\alpha$ is a real number, and $D>0$ is an arbitrary constant, then
	\begin{equation}\label{eq3970M.310} 
		\sup_{\alpha\in\R}\sum_{n \leq x} \mu(n)e^{i 2 \pi  \alpha n}<\frac{c_1x}{(\log x)^{D}} \nonumber,
	\end{equation}
	where $c_1=c_1(D)>0$ is a constant depending on $D$, as the number $x \to \infty$.
\end{thm}

Let $\mathcal{P}\subset \tP=\{2,3,5,7,\ldots\}$ be a subset of primes, and let $\mu_{\mathcal{P}}:\N\longrightarrow \{-1,0,1\}$ be the restriction of the Mobius function to smooth integer $n\in \N(\mathcal{P})=\{n\geq1: p\mid n \rightarrow p\in \mathcal{P}\}$.

\begin{thm} \label{thm3970M.320} {\normalfont (\cite{BG1993})} If $\alpha$ is a real number, then
	\begin{equation}\label{eq3970M.320} 
		\sup_{\alpha\in\R,\; \mathcal{P}}\sum_{n \leq x} \mu_{\mathcal{P}}(n)e^{i 2 \pi  \alpha n}=\frac{c_2x}{\sqrt{\log x}} \left(1+O\left(\frac{\log\log x}{\sqrt{\log x}}\right)\right)\nonumber,
	\end{equation}
	where $c_2>0$ is a constant, as the number $x \to \infty$.
\end{thm}

\section{Sharper Estimate}
The proof for a better unconditional estimate relies on the zerofree region $\Re e(s)>1-c_0/\log t$, where $c_0>0$ is a constant, and $t \in \mathbb{R}$.\\

\begin{thm} \label{thm11.1} {\normalfont (\cite{HS1987}) } Let $\alpha \in \mathbb{R}$ be a real number such that $0<\alpha<1$, and let $x \geq 1$ be a large number. Then, 
	\begin{equation}
		\sum_{n \leq x} \mu(n) e^{i2 \pi \alpha n} =O \left (xe^{-c\sqrt{\log x}} \right ),
	\end{equation} 
	where $c>0$ is an absolute constant.
\end{thm}

This result is a corollary of a more general result for arbitrary functions of certain forms that satisfy the orthogonality relation $\sum_{n \leq x} \mu(n)f(n)=o(x)$, see Conjecture \ref{conjMF222.190}. The proof provided in \cite{HS1987} is a lot simpler than the proofs given in
\cite{BG2003}, and \cite{MS2011}, which are based on the circle methods and the Vaugham identity respectively. 
\\

\section{Conditional Upper Bound}
An upper bound conditional on the generalized Riemann hypothesis proved in \cite{BH1991} claims that $
\sum_{n \leq x} \mu(n) e^{i2 \pi s\alpha n} \ll x^{3/4+\varepsilon},$	
where $\varepsilon>0$ is an arbitrarily small number. The analysis is based on the zerofree region of $L$-functions.\\

An improved upper bound derived from the optimal zerofree region $\{s \in \mathbb{C}: \Re e(s)>1/2 \}$ of the zeta function is presented here. \\

\begin{thm} \label{thm11.2} Suppose that $ \zeta(\rho)=0 \Longleftrightarrow \rho=1/2+it, t \in \mathbb{R}$. Let $\alpha \in \mathbb{R}$ be a real number such that $0<\alpha<1$, 
	nd let $x \geq 1$ be a large number. Then, 
	\begin{equation}
		\sum_{n \leq x} \mu(n) e^{i2 \pi \alpha n} =O \left (x^{1/2}\log^{2}x \right).
	\end{equation}
\end{thm}
\begin{proof} Assume $\alpha =m/q\in\Q$ be a rational number. Then, $f(n)=e^{i2 \pi \alpha n}$ is a periodic function. Consequently,
	\begin{eqnarray}
		\sum_{n \leq x} \mu(n) e^{i2 \pi \alpha n} &=&\sum_{a \bmod q,} \sum_{n \equiv a \bmod q} \mu(n) e^{i2 \pi \alpha n} \nonumber \\
		&=&\sum_{a \bmod q,}  e^{i2 \pi \alpha a}\sum_{n \equiv a \bmod q} \mu(n) \\
		&=&O \left (x^{1/2}\log^{2}x \right) \nonumber.
	\end{eqnarray}
The irrational case $\alpha \in \mathbb{R}$, based on the rational approximation  $|\alpha-m/q|\ll 1/q^2$), is quite similar.
\end{proof} 

The same estimate can be computed by other means, such as the Perron formula.

\section{Problems}
\begin{exe} {\normalfont Let $\alpha \in \mathbb{Q}$ be a rational number. Show that the subset of integers $A_\alpha=\{ n \in \mathbb{N}:\mu(n) e^{i2 \pi \alpha n}= \mu^2(n)\}$ is infinite but has zero density in the set of integers $\mathbb{N}$. Hint: Consider squarefree integers $ \alpha n \in \mathbb{Z}.$ }
\end{exe}

\begin{exe} {\normalfont Let $\alpha \notin \mathbb{Q}$ be an irrational number. Show that $\mu(n) e^{i2 \pi \alpha n} \ne \mu^2(n)$ for all integers $n \geq 1$.}
\end{exe} 

\begin{exe} {\normalfont Let $\alpha \notin \mathbb{Q}$ be an irrational number. Compute an estimate for $$\sum_{n \leq x} \mu(n) e^{i2 \pi \alpha n^2}$$ for large numbers $x \geq 1$.}
\end{exe}

\begin{exe} {\normalfont Let $\alpha \notin \mathbb{Q}$ be an irrational number. Compute an estimate for $$\sum_{n \leq x} \mu(n) e^{i2 \pi \alpha n^3}$$ for large numbers $x \geq 1$.}
\end{exe}

\begin{exe} {\normalfont Let $x\geq 1$ be a large number, and let $s(n)=\sum_{0 \leq i \leq k} n_i$ be the sum of digits function, where $n=\sum_{0 
			\leq i \leq k} n_i\cdot 2^i.$ Compute an estimate for the finite sums
		$$
		\sum_{n \leq x} \frac{\mu(n)(-1)^{s(n)}}{n}   \text{    and     }    \sum_{n \leq x} \mu(n)(-1)^{s(n)}.
		$$}\end{exe}

\begin{exe} {\normalfont Let $x\geq 1$ be a large number. Compute an estimate for the finite sum
		$$
		\sum_{n \leq x} \mu^2(n)e^{i2 \pi \alpha n}.
		$$ }\end{exe}

\chapter{Sets Of $s$-Powerfree Integers} \label{c9339}
Various results on the subsets of $s$-squarefree integers are stated and or proved in this chapter. 
\section{Summatory Functions For Squarefree Integers} \label{s9339}
The subset of $2$-power free integers are usually called squarefree  integers, and denoted by 
\begin{equation}
	\mathcal{Q}_2=\{n\in \mathbb{Z}:\mu^2(n)\ne 0\}
\end{equation} 
and the complementary subset of non squarefree integers is denoted by 
\begin{equation}
	\overline{\mathcal{Q}_2}=\{n\in \mathbb{Z}:\mu^2(n)= 0\}.
\end{equation} 
The number of squarefree integers have the following asymptotic formulas.
\begin{lem} \label{lem9339.107} Let $\mu: \mathbb{Z} \longrightarrow \{-1,0,1\}$ be the Mobius function. Then, for any sufficiently large number $x\geq1$, 
	\begin{equation} 
		\sum_{n \leq x} \mu^2(n) =\frac{6}{\pi^2}x+O \left (x^{1/2} \right ). \nonumber
	\end{equation} 
\end{lem}
\begin{proof} Use Lemma \ref{lem297.58} or confer to the literature.\end{proof}

The constant coincides with the density of squarefree integers. Its approximate numerical value is
\begin{equation}\label{eq9339.72}
	\frac{6}{\pi^2}=\prod_{q\geq 2}\left ( 1-\frac{1}{q^2}\right )=0.607988295164627617135754\ldots,
\end{equation}
where $q\geq2$ ranges over the primes. The remainder term
\begin{equation} 
	R(x)=\sum_{n \leq x} \mu^2(n) -\frac{6}{\pi^2}x
\end{equation} 
is a topic of current research, its optimum value is expected to satisfies the upper bound $R(x)=O(x^{1/4+\varepsilon})$ for any small number $\varepsilon>0$. Currently, $R(x)=O\left (x^{1/2}e^{-\sqrt{\log x}}\right )$ is the best unconditional remainder term.

\begin{lem} \label{lem9339.8} Let $\mu(n)$ be the Mobius function. Then, for any sufficiently large number $x\geq1$, 
	\begin{equation} 
		\sum_{n \leq x} \mu^2(n) =\frac{6}{\pi^2}x+\Omega \left (x^{1/4} \right ). \nonumber
\end{equation} \end{lem}
\begin{proof} The generating series for squarefree integers is $\zeta(s)/\zeta(2s)=\sum_{n \geq 1}\mu^2(n)n^{-s}$ at $s=2$. The Perron intergral yields
	\begin{equation}\label{eq9339.66}
		\sum_{n \leq x} \mu^2(n) =\frac{1}{i2 \pi}\int_{c-\infty}^{c+\infty}\frac{\zeta(s)}{\zeta(s)}\frac{x^s}{s}ds=\frac{1}{\zeta(2)}x +\sum_{\zeta(\rho)=0}c_{\rho}x^{\rho/2},
	\end{equation}
	where $c\ne0$ is a constant. The coefficients $c_{\rho}$ are indexed by the zeros $\rho \in \C$ of the zeta function $\zeta(s)$. Since the zeta function has a zero $\rho_0=1/2+i14.134725\ldots $, the claim follows.
\end{proof}
\begin{thm}\label{thm9339.10}  Let $x\geq 1$ be a large number, let $a$ and $q$ be a pair of integers, $1 \leq a <q =O(\log^c x )$, with $c\geq 0$ constant, and let $\mu: \mathbb{Z} \longrightarrow \{-1,0,1\}$ be the Mobius function. Then, 
	\begin{equation}\label{eq9339.74}
		\sum_{\substack{n \leq x \\
				n \equiv a \bmod q}}\mu(n)^2=\frac{6}{\pi^2}\prod_{p\mid q}\left ( 1-\frac{1}{p^2}\right )^{-1}\frac{x}{q}+O\left (\frac{x^{2/3}}{q} +q^{1/2+\varepsilon} \right ),\nonumber
	\end{equation}
	where $\varepsilon>0$ is an arbitrary small number.
\end{thm}

\begin{proof} Consult \cite{HC1975}, \cite{WR1980}, and the literature.\end{proof}
The range of moduli $q \leq x^{2/3}$ is discussed and improved to $q \leq x^{1-\varepsilon}$ in \cite{NR2014}. The $q$-dependence in the constant 
\begin{equation}\label{eq9339.76}
	\frac{1}{q}\sum_{\substack{n \geq 1 \\
			\gcd(n,q)=1}} \frac{\mu(n)}{n^2}=\frac{1}{q}\prod_{p \nmid q}\left ( 1-\frac{1}{p^2}\right )=\frac{6}{\pi^2}\frac{1}{q}\prod_{p\mid q}\left ( 1-\frac{1}{p^2}\right )^{-1}
\end{equation}
propagates the dependence in the asymptotic formula for consecutive $s$-power free integers. For example, the  probability or density of two consecutive squarefree integers is not $\left (6/\pi^2 \right )^2$, but a more complicated expression similar to \eqref{eq9339.76}. The equidistribution of $s$-power free integers in arithmetic progressions is affirmed by the result below. This also indicates a level of distribution of $2/3$ over any arithmetic progression $\{n=qm+a: m \geq 1\}$. 

\begin{thm}\label{thm3.110}  Let $x\geq 1$ be a large number, let $a$ and $q$ be a pair of integers, $1 \leq a <q =O(\log^c x )$, with $c\geq 0$ constant, and let $\mu: \mathbb{Z} \longrightarrow \{-1,0,1\}$ be the Mobius function. Then, 
	\begin{equation}\label{eq9339.174}
		\sum_{q \leq x^{2/3}\log^{-c-1}x} \max_{ a \bmod q} \left | \sum_{\substack{n \leq x \\
				n \equiv a \bmod q}}\mu(n)^2-\frac{\varphi(q)}{d\varphi(q/d)}\prod_{p\nmid q}\left ( 1-\frac{1}{p^2}\right )\frac{x}{q} \right | \ll \frac{x}{\log ^c x} ,
	\end{equation}
	where $d=\gcd(a,q)$ and $c>0$ is an arbitrary constant.
\end{thm}

\begin{proof} Consult \cite{OR1971} and the literature.\end{proof}

\begin{lem} \label{lem9339.117} Let $x\geq1$ be a large number, and let $\mu: \mathbb{Z} \longrightarrow \{-1,0,1\}$ be the Mobius function. If $q=O(\log^c x)$ with $c\leq 0$ constant, then,  
	\begin{equation} 
		\sum_{\substack{n \leq x\\ \gcd(n,q)=1}} \mu^2(n) =\frac{6}{\pi^2}\prod_{p\nmid q}\left ( 1+\frac{1}{p}\right )^{-1}x+O \left (x^{1/2} \right ). \nonumber
	\end{equation} 
\end{lem}
\begin{proof} The proof is lengthier and more difficult than Lemma \ref{lem9339.107}, see \cite[Lemma 2]{DK2005}.\end{proof}

\section{Correlation Functions For Squarefree Integers} \label{s8009}
A sequence of squarefree integers
\begin{equation} \label{eq8009.030}
	n +a_0, \quad n+a_1,\quad n+a_2,\quad \ldots,\quad n+a_k, 
\end{equation}
imposes certain restriction on the $(k+1)$-tuple $( a_0, a_1, \ldots, a_k)$. A stronger restriction is required for sequence of prime $(k+1)$-tuples , see \cite{BT2013}, and the literature for extensive details.

\begin{dfn}  \label{dfn8009.35}{\normalfont A $k$-tuple $(a_0,a_1, \ldots, a_k)$ is called \textit{admissible} if the numbers $a_0,a_1, \ldots, a_k$ is not a complete residues system modulo $p$ for any prime $p\leq k$.
	}
\end{dfn}

\begin{lem}\label{lem8009.41}  Let $x\geq 1$ be a large number, and let $\mu: \mathbb{Z} \longrightarrow \{-1,0,1\}$ be the Mobius function. Then, 
	\begin{equation}\label{eq8009.74}
		\sum_{n \leq x}\mu(n)^2 \mu(n+1)^2=\prod_{p\geq 2}\left ( 1-\frac{2}{p^2}\right )x+O\left (x^{2/3} \right ).\nonumber
	\end{equation}
\end{lem}

\begin{proof} The earliest proof seems to be that in \cite{CL1932}, and \cite{ML1947}. Recent proofs appear in \cite{MI2017}, and the literature.
\end{proof}

The constant coincides with the density of 2-consecutive squarefree integers. Its approximate numerical value is
\begin{equation}\
	\prod_{q\geq 2}\left ( 1-\frac{2}{q^2}\right )=0.322699054242535576161483\ldots,
\end{equation}
where $q\geq2$ ranges over the primes.
\begin{lem}\label{lem8009.44}  Let $x\geq 1$ be a large number, and let $\mu: \mathbb{Z} \longrightarrow \{-1,0,1\}$ be the Mobius function. Then, 
	\begin{equation}\label{eq8009.74}
		\sum_{n \leq x}\mu(n)^2 \mu(n+1)^2\mu(n+2)^2=\prod_{p\geq 2}\left ( 1-\frac{3}{p^2}\right )x+O\left (x^{2/3} \right ).
	\end{equation}
\end{lem}
The earliest result in this direction appears to be
\begin{equation}\label{eq8009.74}
	\sum_{n \leq x}\mu(n)^2 \mu(n+t)^2=cx+O\left (x^{2/3} \right ),
\end{equation}
where $c>0$ is the constant \eqref{eq9339.72}, is studied in \cite{ML1947}. Except for minor adjustments, the generalization to sequences of $(k+1)$-tuples of squarefree integers has the same structure.

\begin{thm}\label{thm8009.10}  Let $ a\geq 1$ and $s\geq 2$ be small integers. Let $x\geq 1$ be a large number, and let $\mu_s: \mathbb{Z} \longrightarrow \{-1,0,1\}$ be the $s$-power free characteristic function. Then, 
	\begin{equation}\label{eq8009.74}
		\sum_{n \leq x}\mu(n+a_0)^2\mu(n+a_1)^2 \cdots \mu(n+a_k)^2=\prod_{p\geq 2}\left ( 1-\frac{\rho(s)}{p^2}\right )x+O\left (x^{2/3+\varepsilon} \right ),\nonumber
	\end{equation}
	where $q\geq 1$ is a constant, and 
	\begin{equation}\label{eq8009.79}
		\rho(s)=\#\{m\leq p^2: qm+a_i\equiv 0 \bmod p^2 \text{ for } i=0,1,2, ..., k\},
	\end{equation}
	and $\varepsilon>0$ is an arbitrary small number depending on $k$ and $q$.
\end{thm}
\begin{proof}Consult \cite{ML1947}, \cite[Theorem 1.2]{MI2017}, \cite{TK1985}, and the literature.\end{proof}

The literature does not seem to offer any results for squarefree twin integers $n$ and $n+a$, which are relatively prime to $q=q(a)$. A plausible result might have the form given below. 
\begin{conj}\label{conj8009.105}  Let $x\geq 1$ be a large number, and let $\mu: \mathbb{Z} \longrightarrow \{-1,0,1\}$ be the Mobius function. If $a\geq 1$ is a fixed integer, and $q=O(\log^c x)$ with $c\geq 0$ constant, then, 
	\begin{equation}\label{eq8009.190}
		\sum_{\substack{n \leq x\\ \gcd(n,q)=1\\\gcd(n+a,q)=1}}\mu(n)^2 \mu(n+a)^2=c_2(q,a)\prod_{p\nmid q}\left ( 1+\frac{1}{p}\right )^{-2}\prod_{p\geq 2}\left ( 1-\frac{2}{p^2}\right )x+O\left (x^{1-\delta} \right ),\nonumber
	\end{equation}
	where dependence correction factor $c_2(q,a)\geq 0$, and $\delta >0$ is a small number.
\end{conj}

The dependence correction factor $c_2(q,a)\geq 0$, and the parameter $q=q(a)$ depends on $a \geq 1$. For instance, for $a=2b+1$ odd, the value $q=q(a)$ must be odd, and $c_2(q,a)>0$, otherwise $c_2(q,a)= 0$ for even $q$.\\

The related case of even-odd power has a short proof, this is provided here.

\begin{lem}\label{lem8009.350}  Let $x\geq 1$ be a large number, and let $\mu: \mathbb{Z} \longrightarrow \{-1,0,1\}$ be the Mobius function. If $a\geq 1$ is a fixed integer,  then, 
	\begin{equation}\label{eq8009.350}
		\sum_{n \leq x}\mu(n)^2 \mu(n+a)=O\left (\frac{x}{(\log x)^c} \right ),\nonumber
	\end{equation}
	where $c> 0$ is an arbitrary constant.
\end{lem}
\begin{proof} Using Lemma \ref{lem297.58}, and switching the order of summation yield
\begin{eqnarray}\label{eq8009.360}
\sum_{n \leq x}\mu(n)^2 \mu(n+a)&=&\sum_{n \leq x} \mu(n+a)\sum_{d^2\mid n}\mu(d)\\
&=&\sum_{d^2 \leq x}\mu(d) \sum_{\substack{n\leq x\\d^2\mid n}}\mu(n+a)\nonumber\\
&=&\sum_{d^2 \leq x^{2\varepsilon}}\mu(d) \sum_{\substack{n\leq x\\d^2\mid n}}\mu(n+a)+\sum_{x^{2\varepsilon}< d^2 \leq x}\mu(d) \sum_{\substack{n\leq x\\d^2\mid n}}\mu(n+a)\nonumber,
\end{eqnarray}
where $\varepsilon\in (0,1/2)$. Applying Corollary \ref{corMMF525.550} to the first subsum in the partition yields
\begin{eqnarray}\label{eq8009.370}
\sum_{d^2 \leq x^{2\varepsilon}}\mu(d) \sum_{\substack{n\leq x\\d^2\mid n}}\mu(n+a)&\leq&\sum_{q \leq x^{\varepsilon}} \left |\mu(d) \sum_{\substack{n\leq x\\m\equiv b \bmod q}}\mu(m)\right |\\
	&=&O\left( \frac{x}{(\log x)^{c}}\right) \nonumber,
\end{eqnarray}
An estimate of the second subsum in the partition yields
\begin{eqnarray}\label{eq8009.380}
\sum_{x^{2\varepsilon}< d^2 \leq x}\mu(d) \sum_{\substack{n\leq x\\d^2\mid n}}\mu(n+a)&\leq&\sum_{x^{2\varepsilon}< d^2 \leq x} \sum_{\substack{n\leq x\\d^2\mid n}}1\\
	&\ll&x\sum_{d^2 \leq x}\frac{1}{d^2}\nonumber\\
	&\ll&x^{1-\varepsilon}\nonumber.
\end{eqnarray}
Summing \eqref{eq8009.370} and \eqref{eq8009.380} completes the verification.
\end{proof}

\section{Summatory Functions For $s$-Power Free Integers} \label{s9539}
The subset of $k$-power free integers is usually denoted by 
\begin{equation}
	\mathcal{Q}_s=\{n\in \mathbb{Z}:\mu_s(n)\ne 0\}
\end{equation} 
and the complementary subset of non $s$-free integers is denoted by 
\begin{equation}
	\overline{\mathcal{Q}_s}=\{n\in \mathbb{Z}:\mu_s(n)= 0\}.
\end{equation}

The number $s$-power free integers have the following asymptotic.
\begin{lem} \label{lem9539.118} Given an integer $s \geq 2$, let $\mu_s(n)$ be the $s$th-Mobius function. Then, for any sufficiently large number $x\geq1$, 
	\begin{equation} 
		\sum_{n \leq x} \mu_s(n) =\frac{1}{\zeta(s)}x+O \left (x^{1/s} \right ).  
\end{equation} \end{lem}

\begin{proof}  The basic $s$th-Mobius function $\mu_s$ is explained in Definition \ref{dfn297.30}. This result is attributed to Gegenbauer, 1885. Recent proofs are provided in \cite{IA2003} and the literature.\end{proof}

\begin{lem} \label{lem9339.108}  Given an integer $s \geq 2$, let $\mu_s(n)$ be the $s$th-Mobius function. Then, for any sufficiently large number $x\geq1$, 
	\begin{equation} 
		\sum_{n \leq x} \mu_s(n) =\frac{1}{\zeta(2s)}x+\Omega \left (x^{1/2s} \right ). \nonumber
\end{equation} \end{lem}
\begin{proof} Same as the proof of Lemma \ref{lem9339.8}, mutatis mutandus.
\end{proof}

\begin{conj} \label{conj9339.208}  Given a pair of integers $s \geq 2$, and $q\geq 2$, let $\mu_s(n)$ be the $s$th-Mobius function. Then, for any sufficiently large number $x\geq1$, 
	\begin{equation} 
		\sum_{\substack{n \leq x\\ \gcd(n,q)=1}} \mu_s(n) =\frac{1}{\zeta(2s)}\prod_{p\nmid q}\left ( 1+\frac{1}{p}\right )^{-1}x+O \left (x^{1/2s} \right ). \nonumber
\end{equation} \end{conj}

\section{Correlation Functions For $s$-Power Free Integers} \label{s9549}
\begin{thm}\label{thm9549.310}  Let $s \geq 2)$ be an integer. Let $x\geq 1$ be a large number, and let $\mu_s: \mathbb{Z} \longrightarrow \{-1,0,1\}$ be the characteristic function of $s$-power free integers. Then, 
	\begin{equation}\label{eq9539.374}
		\sum_{n \leq x}\mu_s(n)\mu_s(n+a) =\prod_{p\geq 2}\left ( 1-\frac{\rho(p,a)}{p^s}\right )x+O\left (x^{\alpha(s)+\varepsilon} \right ),\nonumber
	\end{equation}
	where 
	\begin{equation} \label{eq9549.376}
		\rho(p) =
		\left \{
		\begin{array}{ll}
			2     &\text{ if } p^s \nmid a,\\
			1           &\text{ if } p^s \mid a,\\
		\end{array}
		\right .
	\end{equation}
	and 
	\begin{equation} \label{eq9549.378}
		\alpha(p,a) =\frac{14}{7s+8}
	\end{equation}
	and $\varepsilon>0$ is an arbitrary small number.
\end{thm}
\begin{proof} Different proofs are given in \cite{RT2012}, \cite[Theorem 1.2]{BJ2013}, which have slightly different remainder terms.\end{proof}

The main problems in this area are the determination of the best remainder terms for various summatory functions. For instance, the remainder term
\begin{equation} 
	R_s(x)=\sum_{n \leq x} \mu_s(n) -\frac{1}{\zeta(2s)}x
\end{equation} 
in Theorem \ref{thm9549.310} is expected to satisfies the upper bound $R_s(x)=O(x^{1/2s+\varepsilon})$ for any small number $\varepsilon>0$. A survey of the literature on $s$-power free integers and arithmetic functions is presented in \cite{PF2005}. Currently, $R_s(x)=O\left (x^{1/2s}e^{-\sqrt{\log x}}\right )$ is the best unconditional remainder term.\\

The literature does not seem to offer any results for $s$-power free twin integers $n$ and $n+a$, with $a\geq 1$. A plausible result might have the form given below.

\begin{conj}\label{conj9549.100}  Given a pair of integers $a \geq 1$ and $s\geq 2$. Let $x\geq 1$ be a large number, and let $\mu: \mathbb{Z} \longrightarrow \{-1,0,1\}$ be the Mobius function. If $a \geq 1$, and $q=O(\log^c x)$ with $c\geq 0$ constant, then, 
	\begin{equation}\label{eq8009.190}
		\sum_{\substack{n \leq x\\ \gcd(n,q)=1\\\gcd(n+a,q)=1}}\mu_s(n) \mu_s(n+1)=c_s(q,a)\prod_{p\nmid q}\left ( 1+\frac{1}{p}\right )^{-s}\prod_{p\geq 2}\left ( 1-\frac{2}{p^s}\right )x+O\left (x^{1/2s-\delta} \right ),\nonumber
	\end{equation}
	where $c_s(q,a)\geq 0$ is a constant, and $\delta >0$ is a small number.
\end{conj}

The constant $c_s(q,a)\geq 0$ and the parameter $q=q(a)$ depend on $a \geq 1$. For instance, for $a=2b+1$ odd, the value $q=q(a)$ must be odd, and $c_s(q,a)>0$, otherwise $c_s(q,a)= 0$ for even $q$.

\section{Probabilities For Consecutive Squarefree Integers}\label{s1180}
The events of 2 consecutive squarefree integers $X_0$ and $X_1$ are dependent random variables. Similar, the events of 3 consecutive squarefree integers $X_0$, $X_1$, and $X_2$ are dependent random variables.  \\

The probability $P(\mu(X_0)=\pm1, \mu(X_1)=\pm1)$ for 2 consecutive squarefree integers is asymptotic to the constant attached to the main term in Lemma \ref{lem8009.41}. Specifically,
\begin{equation} \label{eq1180.040}
	\prod_{q \geq 2}\left (1-\frac{2}{q^2}\right )=\left (\frac{6}{\pi^2}\right )^2 \prod_{q \geq 2}\left (1+\frac{1}{q^2(q^2-2)}\right )^{-1}=0.322699054242535\ldots.
\end{equation}
The reduction from independent events is measured by the dependence correction factor
\begin{equation} \label{eq1180.048}
	c_2(2)=\prod_{q \geq 2}\left (1+\frac{1}{q^2(q^2-2)}\right )^{-1}=0.872985953449313618771745\ldots.
\end{equation}

The probability $P(\mu(X_0)=\pm1, \mu(X_1)=\pm1, \mu(X_2)=\pm1)$ for 3 consecutive squarefree integers is asymptotic to the constant attached to the main term in Lemma \ref{lem8009.44}. Specifically,
\begin{equation} \label{eq1180.140}
	\prod_{q \geq 2}\left (1-\frac{3}{q^2}\right )=\left (\frac{6}{\pi^2}\right )^3\prod_{q \geq 2}\left (1+\frac{3q^2-1}{q^4(q^2-3)}\right )^{-1}=0.125524878896821\ldots.
\end{equation}
The reduction from independent events is measured by the dependence correction factor
\begin{equation} \label{eq1180.148}
	c_2(3)=\prod_{q \geq 2}\left (1+\frac{3q^2-1}{q^4(q^2-3)}\right )^{-1}=0.558526979127689105533330\ldots.
\end{equation}
Accordingly, consecutive squarefree integers are highly correlated.

\section{Densities For Squarefree Integers}
The subset of squarefree integers is usually denoted by 
\begin{equation}
	\mathcal{Q}=\{n\in \mathbb{Z}:n=p_1p_2 \cdots p_t, \text{with }p_k \text{ prime }\}
\end{equation} 
and the subset of nonsquarefree integers is denoted by 
\begin{equation}
	\overline{\mathcal{Q}}=\{n\in \mathbb{Z}:n\ne p_1p_2 \cdots p_t, \text{with }p_k \text{ prime }\}.
\end{equation} 
\begin{lem} \label{lem3.7} Let $\mu(n)$ be the Mobius function. Then, for any sufficiently large number $x>1$, 
	\begin{equation} 
		\sum_{n \leq x} \mu(n)^2 =\frac{6}{\pi^2}x+O \left (x^{1/2} \right ). 
	\end{equation} 
\end{lem}

\begin{lem} \label{lem3.8} Let $\mu(n)$ be the Mobius function. Then, for any sufficiently large number $x>1$, 
	\begin{equation} 
		\sum_{n \leq x} \mu(n)^2 =\Omega \left (x^{1/2} \right ). 
\end{equation} \end{lem}

\begin{lem} \label{lem3.9} Let $d(n)=\sum_{d \mid n}1$ be the divisors function, and let $\mu(n)$ be the Mobius function. Then, for any sufficiently large number $x>1$, 
	\begin{equation} 
		\sum_{n \leq x} \mu(n)^2d(n) =\frac{6}{\pi^2}(\log x+1-\gamma)x+O \left (x^{1/2} \right ). 
\end{equation} \end{lem}

\begin{lem}\label{lem3.10}  Fir a pair of integers $1 \leq a <q$. Let $x\geq 1$ be a large number, and let $\mu: \mathbb{Z} \longrightarrow \{-1,0,1\}$ be the Mobius function. Then, for some constant $c>0$,
	\begin{equation}\label{eq800.74}
		\sum_{\substack{n \leq x \\
				n \equiv a \bmod q}}\mu(n)^2=\frac{c}{q}x+O\left (x^{1/2} \right ).
	\end{equation}
\end{lem}

\begin{proof} Consult the literature.\end{proof}

\section{Subsets of Squarefree Integers of Zero Densities} 
A technique for estimating finite sums over subsets of integers $\mathcal{A} \subset \mathbb{N}$ of zero densities in the set of integers $\mathbb{N}$ is sketched here. Write the counting function as $A(x)=\# \{n \leq x:n \in \mathcal{A}\}$. The case for squarefree integers occurs frequently in number theory. In this case, let  $ \mathcal{A} \subset \mathcal{Q}=\{n\in \mathbb{N}:\mu(n)\ne0\} $, and the measure $A(x)=\# \{n \leq x:n \in \mathcal{A}\}=o(x)$. \\

\begin{lem} \label{lem3.10} Let $C> 1$ be a constant, let $d(n)=\sum_{d|n}1$ be the divisors function. If $A(x)=O(x/\log^C x)$, then, for any sufficiently large number $x>1$,
	\begin{equation} 
		\sum_{n \leq x, n\in \mathcal{A}} \mu(n)^2d(n) =O \left (\frac{x}{(\log x)^{C}} \right ). 
	\end{equation}
\end{lem}
\begin{proof}  For squarefree integers the divisor function reduces to $d(n)=\sum_{d|n}1=2^{\omega(n)}$, but this is not required here. Rewrite the finite sum as
	\begin{equation}
		\sum_{n \leq x, n\in \mathcal{A}} \mu(n)^2d(n)=\sum_{n\leq x, n\in \mathcal{A}} \mu(n)^2 \sum_{d \mid n} 1
		=\sum_{d\leq x,  } \sum_{n\leq x/d, n\in \mathcal{A}} \mu(n)^2.  
	\end{equation}
	
	Next, applying the measure $A(x)=\# \{n \leq x:n \in \mathcal{A}\}=O(x/\log^C x)$ to the inner finite sum yields: 
	\begin{eqnarray}
		\sum_{d\leq x,} \sum_{n\leq x/d, n\in \mathcal{A}} \mu(n)^2 
		&=&O \left (\frac{x}{\log^{C}x} \sum_{d\leq x} \frac{1}{d}\right ) \\
		&=&O \left (\frac{x}{\log^{C-1}x}\right )\nonumber,  
	\end{eqnarray}
	with $C>1$.  
\end{proof}

\begin{lem} \label{lem3.11} Let $C>1$ be a constant, and let $d(n)=\sum_{d|n}1$ be the divisor function. If $A(x)=O(x/\log^C x)$, then, for any sufficiently large number $x>1$,
	\begin{equation}
		\sum_{n \leq x, n\in \mathcal{A}} \frac{\mu(n)^2d(n)}{n} =O \left (\frac{1}{(\log x)^{C}} \right ).
	\end{equation} 
\end{lem}
\begin{proof} Let $R(x)= \sum_{n \leq x, n\in \mathcal{A}} \mu(n)^2d(n)$. A summation by parts leads to the integral representation 
	\begin{equation}
		\sum_{n\leq x, n\in \mathcal{A}} \frac{\mu(n)^2d(n)}{n}= \int_{1}^{x} \frac{1}{t} d R(t).  
	\end{equation}
	For information on the Abel summation formula, see \cite[p.\ 4]{CR2006}, \cite[p.\ 4]{MV2007}, \cite[p.\ 4]{TG2015}. Evaluate the integral: 
	\begin{equation}
		\int_{1}^{x} \frac{1}{t} d R(t)=\frac{1}{x} \cdot  O \left (\frac{x}{\log^{C}x} \right )+ \int_{1}^{x}\frac{1}{t^2} R(t)dt=O \left (\frac{1}{\log^{C}x} \right ) ,  
	\end{equation}
	where the constant is $C-1>0$. \end{proof}

\section{Problems}
\begin{exe} {\normalfont Let $x\geq 1$ be a large number. Show that the Prime Number Theorem implies that
		$$
		\sum_{n \leq x, \mu(n)=1} \mu(n)=\sum_{n \leq x} \frac{\mu^2(n)+\mu(n)}{2}=\frac{3}{\pi^2} x+O\left (xe^{-c\sqrt{\log x}} \right ).
		$$}
\end{exe}

\begin{exe} {\normalfont Let $x\geq 1$ be a large number. Show that the Prime Number Theorem implies that
		$$
		\sum_{n \leq x, \mu(n)=-1} \mu(n)=\sum_{n \leq x} \frac{\mu^2(n)-\mu(n)}{2}=\frac{3}{\pi^2} x+O\left (xe^{-c\sqrt{\log x}} \right ).
		$$
	}
\end{exe}

\begin{exe} {\normalfont Let $x\geq 1$ be a large number, and let $\{s_n:n \in \mathbb{N}\} \subset \mathbb{N}$ be a subsequence  of integers, including random sequences. Compute an asymptotic formula or estimate for
		$$
		\sum_{n \leq x} \mu(s_n) .
		$$}
\end{exe}

\begin{exe} {\normalfont Let $x\geq 1$ be a large number. Show that the main term in the finite sum
		$$
		\#\{n \leq x:n=m^2 \}=\sum_{n \leq x,} \sum_{d \mid n}\lambda(d)=[x^{1/2}]+E(x),
		$$
		is $M(x)=[x^{1/2}]$. But the error term $E(x)=O(x^{1/2})$ has the same order of magnitude. Hence, it can change signs infinitely often.}
\end{exe}

\begin{exe} {\normalfont Let $x\geq 1$ be a large number, and let $[x]=x- \{x\}$ be the largest integer function. Consider the finite sum
		$$
		\sum_{n \leq x,} \sum_{d \mid n}\lambda(d)=\sum_{d \leq x} \lambda(d)\sum_{\substack{n \leq x\\d \mid n}}1=x\sum_{d \leq x} \frac{\lambda(d)}{d}-\sum_{d \leq x} \lambda(d)
		\left \{\frac{x}{d} \right \}.
		$$
		Prove or disprove that the main term and the error term are
		$$
		M(x)=-\sum_{d \leq x} \lambda(d)  \qquad \text{ and } \qquad E(x) = x\sum_{d \leq x} \frac{\lambda(d)}{d}
		$$
		
		respectively, where $M(x)=[x^{1/2}]$, and $E(x)=O(x^{1/2})$.}
\end{exe}

\begin{exe} {\normalfont Let $x\geq 1$ be a large number. Show that the main term in the finite sum
		$$
		x^{1/2} \leq \sum_{n \leq x} \sum_{d \mid n}\lambda(d) <2x^{1/2}.
		$$
	}
\end{exe}

\begin{exe} {\normalfont Let $x\geq 1$ be a large number. Show that the Prime Number Theorem implies that
		$$
		\sum_{n \geq 1, \,n=even} \frac{\mu(n)}{n}=\frac{-1}{3}\sum_{n \geq 1} \frac{\mu(n)}{n}		\quad \text{and} \quad 
		\sum_{n \geq 1, \,n=odd} \frac{\mu(n)}{n}=\frac{4}{3}\sum_{n \geq 1} \frac{\mu(n)}{n},$$}
\end{exe}


\chapter{Prime Numbers Theorems} \label{s8200}
A survey of the prominent results in prime numbers theory is recorded in this Chapter.

\section{Primes Indicator Functions}\label{S9900}
The prime numbers have a few indicator functions, two of them are displayed in the next two Lemmas. However, every result in number theory is derived or proved via the simpler weighted indicator function specified in Definition \ref{dfn9900.020}.
\begin{lem}\label{lem9900.100} Let $\omega(n)=\#\{p\mid n\}$ be the prime divisors counting function, and let $\mu(n)\in \{-1,0,1\}$ be the Mobius function. If $n\geq 1$ is an integer, then the prime numbers indicator function is given by 
	\begin{equation}\label{eq9900.100}
		\varkappa(n)=\sum_{d\mid n}\mu(d)\omega(n/d)=
		\begin{cases}
			1& \text{ if } n \text{ is prime,}\\
			0& \text{ if } n \text{ is not prime.}\nonumber 
		\end{cases}
	\end{equation}
\end{lem}
\begin{proof}[\textbf{Proof}] Consider the generating series
	\begin{equation}\label{eq9900.110}
		\sum_{n\geq 1}\frac{\omega(n)}{n^s}=\sum_{n\geq 1}\frac{1}{n^s}\sum_{n\geq 1}\frac{\varkappa(n)}{n^s},
	\end{equation}
	and solve for $\varkappa(n)$.
\end{proof}
\begin{lem}\label{lem9900.120} Let $\Omega(n)$ be the prime divisors, including multiplicities, counting function, and let $\mu(n)\in \{-1,0,1\}$ be the Mobius function. If $n\geq 1$ is an integer, then the prime powers indicator function is given by 
	\begin{equation}\label{eq9900.120}
		\varkappa_0(n)=\sum_{d\mid n}\mu(d)\Omega(n/d)=
		\begin{cases}
			1& \text{ if } n \text{ is a prime power,}\\
			0& \text{ if } n \text{ is not prime power.}\nonumber 
		\end{cases}
	\end{equation}
\end{lem}

\begin{dfn}\label{dfn9900.020}{\normalfont The vonMangoldt function is defined by the weighted prime powers indicator function
		$$\Lambda(n)=
		\begin{cases}
			\log p & \text{if } n=p^m,\\ 
			0 & \text{if } n\ne p^m.\\ 
		\end{cases}
		$$
	}
\end{dfn}
The symbol $p^m\geq 2$, with $ m \in \N$, denotes a prime power.\\

\section{Results for the vonMangoldt Function}

\begin{lem} \label{lem773.1} For any integer $n \in \mathbb{N}$, the function $\Lambda(n)$ has the inverse Mobius pair 
	$$
	\Lambda(n)=-\sum_{d \mid n} \mu(d) \log d   \qquad \qquad \mbox{ and } \qquad \qquad  \log(n)=\sum_{d \mid n} \Lambda(d) . 
	$$
\end{lem} 
\begin{proof} Use the inversion formula, \cite[Theorem 2.9]{AP1976}.  
\end{proof}
The approximation of the vonMangoldt function
\begin{equation}  \label{eq600.42}
	\Lambda_R(n)=-\sum_{\substack{d \mid n \\ d \leq R}} \mu(d) \log d  
\end{equation}
is employed in Theorem \ref{thm773.15} to derive an approximation of the autocorrelation.

\section{Primes Counting Functions And Prime Numbers Theorems} \label{S8200}
The weighted primes counting functions, psi $\psi(x)$ and theta $\theta(x)$, are defined by
\begin{equation}\label{eq8200.001}
	\theta(x)=\sum_{p \leq x} \log p 
\end{equation}
and
\begin{equation}\label{eq8200.003}
	\psi(x)=\sum_{n \leq x} \Lambda(n)                 
\end{equation}
respectively. The standard prime counting function is denoted by
\begin{equation}
	\pi(x)=\#\{p \leq x\}=\sum_{p \leq x} 1=\sum_{n\leq x} \varkappa(n).                 
\end{equation}

\begin{thm} \label{thm8200.100} Uniformly for $x \geq 2$ the psi and theta functions have the followings asymptotic formulae.
	\begin{enumerate} [font=\normalfont, label=(\roman*)]
		\item  $\displaystyle \theta(x)=x +O\left (xe^{-c_0 \sqrt{\log x}}\right ), $ \tabto{8cm}  unconditionally.
		\item 
		$\displaystyle
		\theta(x)=x +\Omega_{\pm} \left ( x^{1/2}\log \log \log x \right ),$\tabto{8cm}  unconditional oscillation.
		\item $\displaystyle
		\theta(x)=x +O\left (x^{1/2} \log^2 x \right ).  
		$\tabto{8cm} conditional on the RH.
	\end{enumerate}
\end{thm}

\begin{proof} (ii) The oscillations form of the theta function is proved in \cite[p.\ 479]{MV2007}, 
\end{proof}

The same asymptotics hold for the function $\psi(x)$. Explicit estimates for both of these functions are given in \cite{CP1985}, \cite{SL1976}, \cite[Theorem 5.2]{DP2010}, and related literature. 

\begin{conj} \label{conj8200.1} Assuming the RH and the LI conjecture, the suprema are
	\begin{equation}
		\lim \inf_{x \to \infty} \frac{\psi(x)-x}{\sqrt{x} (\log \log x)^2}=\frac{-1}{\pi} \qquad \text{and} \qquad \lim \sup_{x \to \infty} \frac{\psi(x)-x}{\sqrt{x} (\log \log x)^2}=\frac{1}{\pi}.
	\end{equation}
\end{conj}

More details on the Linear Independence conjecture appear in \cite{IA1942}, \cite[Theorem 6.4]{EE1985}, and recent literature. The LI conjecture asserts that the imaginary parts of the nontrivial zeros $\rho_n=1/2+i \gamma_n$ of the zeta function $\zeta(s)$ are linearly independent over the set $\{-1,0, 1\}$. In short, the equations  
\begin{equation} \label{eq8200.30}
	\sum_{1 \leq n \leq M} r_n \gamma_n=0,
\end{equation}
where $r_n \in \{-1,0, 1\}$, have no nontrivial solutions. \\

This function is usually expressed in term of the logarithm integral $\li(x)=\int_2^x (\log t)^{-1} dt$.
\begin{thm} \label{thm8200.200} Let $x \geq 1$ be a large number. Then
	\begin{enumerate} [font=\normalfont, label=(\roman*)]
		\item  $\displaystyle \pi(x)=\li(x) +O\left (xe^{-c_0 \sqrt{\log x}}\right ) , $ \tabto{8cm}  unconditionally.
		\item 
		$\displaystyle
		\pi(x)=\li(x) +\Omega_{\pm} \left (\frac{ x^{1/2}\log \log \log x}{\log x} \right ),$\tabto{8cm}  unconditional oscillation.
		\item $\displaystyle
		\pi(x)=\li(x) +O\left (x^{1/2} \log x \right ),  
		$\tabto{8cm} conditional on the RH.
	\end{enumerate}
\end{thm}

\begin{proof} (i) The unconditional part of the prime counting formula arises from the delaVallee Poussin form $\pi(x)=\li(x)+O\left (xe^{-c_0 \sqrt{\log x}}\right )$ of the prime number  theorem, see \cite[p.\ 179]{MV2007}. Recent information on the constant $c_0>0$ and the sharper estimate 
	\begin{equation} \label{eq8200.40}
		\pi(x)=\li(x)+O\left ( x e^{-c_0 \log x^{3/5}(\log \log x)^{-2/5}}\right ) 
	\end{equation}
	appears in \cite{FK2002}. 
	(ii) The unconditional oscillations part arises from the Littlewood form 
	\begin{equation} \label{eq8200.42}
		\pi(x)=\li(x)+\Omega_{\pm} \left (x^{1/2}\log \log \log  x /\log x \right )
	\end{equation} 
	of the prime number theorem, 
	consult \cite[p.\ 51]{IA2003}, \cite[p.\ 479]{MV2007}, et cetera. \\
	
	(iii)  The conditional part arises from the Riemann form $\pi(x)=\li(x)+O\left (x^{1/2}\log^2 x \right )$ of the prime number theorem. 
\end{proof}

New explict estimates for the number of primes in arithmetic progressions are computed in \cite{BR2018}.

\begin{dfn} \label{dfn2000.2} Let $f: \N \longrightarrow \C$ be an arithmetic function, and let $x \geq 1$ be a real number. A level of distribution is a value $Q<x$ such that
	\begin{equation}\label{eq2000.23?}
		\sum_{q<Q}\max_{\gcd(a,q)} \left | \sum_{\substack{n \leq x\\ n \equiv a \bmod q}} f(n) -  \frac{ 1}{\varphi(q)}\sum_{\substack{n \leq x\\
				\gcd(n,q)=1}} f(n)\right |\ll \left (\frac{ x}{\log^C x}\right ) ,
	\end{equation}
	where $C>0$ is a constant.
\end{dfn}

This compares the average orders over arithmetic progressions to the weigted average orders of the function. Recent develoments in this topic are discussed in \cite{GR2017}. The prime counting function $\pi(x,a,q)$ has the best known results for the level of distributions.\\

\begin{tabular}{l|l}
	Result& Level  Of Distribution\\
	\hline
	Bombieri-Vinogradov Theorem    &  $Q=x^{1/2}/\log^B x$\\
	Friedlander- Granville Theorem, \cite{FG1989}    & $Q<x/\log^B x$\\
	Elliot-Halberstam Conjecture    & $Q=x^{1-\varepsilon}$\\
\end{tabular}

\begin{thm} {\normalfont (Bombieri-Vinogradov theorem)}\label{thm2000.222} Let $C>0$ be a constant. Then, there exists a constant $D>1$ for which
	\begin{equation}\label{eq2000.23?}
		\sum_{q<x^{1/2}/\log^D x}\max_{\gcd(a,q)} \left | \pi(x, a,q) -  \frac{ 1}{\varphi(q)}\li(x)\right |\ll \left (\frac{ x}{\log^C x}\right ) ,
	\end{equation}
	as $x \to \infty$.
\end{thm}

The prime numbers functions $\psi(x)$ and $\pi(x)$ are linked by the identities
\begin{equation}\label{lem9900.140}
	\psi(x)=\pi(x)\log x -\int_2^x\frac{\pi(t)}{t}dt,
\end{equation}
and 
\begin{equation}\label{lem9900.145}
	\pi(x)=\frac{\psi(x)}{\log x} +\int_2^x\frac{\psi(t)}{t(\log t)^2}dt,
\end{equation}
for all large numbers $x\geq1$.

\begin{thm} \label{thm773.03} Let $x \geq 1$ be a large number, and let $C>0$ be a constant. Then, 
	\begin{enumerate} [font=\normalfont, label=(\roman*)]
		\item If $a<q$ are relatively prime integers, and $q=O(\log^C x)$, then\\
		$$  \psi(x,q,a)=\frac{x}{\varphi(q)}+O \left (x e^{-c_0 \sqrt{\log x}}\right ),$$  \\
		where $c_0>0$ is an absolute constant.
		\item If $a<q$ are relatively prime integers, then \\
		$$  \psi(x,q,a)=\frac{x}{\varphi(q)}+O \left (\frac{x}{\log^{C}x}\right ).$$ 
	\end{enumerate}
\end{thm}

\begin{proof} The proof of part (i) appears in \cite[Corollary 11.19]{MV2007}, and the proof of part (ii) appears in \cite[Theorem 8.8]{EE1985}.  
\end{proof}

The same asymptotics hold for the function $\psi(x)$. Explicit estimates for both of these functions are given in \cite{CP1985}, \cite{SL1976}, \cite[Theorem 5.2]{DP2010}, and related literature. 

\begin{thm}  \label{thm773.07} Let $x \geq 1$ be a large number, and $C>0$ be a constant, then
	\begin{equation} \label{eq182.52}
		\sum_{n \leq x}\Lambda(n)^2=x \log x +O\left( x \right ).
	\end{equation}
\end{thm}
\begin{proof} Convert the summatory function into an integral and evaluate it:
	\begin{equation}
		\sum_{n \leq x}\Lambda(n)^2=\int_1^x(\log t)^2 d \pi(t) +O(x),
	\end{equation}
	where $\pi(x)=x/\log x+O(x/\log^2 x) $, and the error term $O(x)$ accounts for the omitted prime powers $p^m$, $m \geq 2$.
\end{proof}
Different proofs for the related versions such as $\sum_{n \leq x}(x-n)\Lambda(n)$ of this result appears in \cite[Theorem 2.8]{EE1985}, and \cite[p.\ 329]{NW2000}.

\section{Primes and Almost Primes Indicator Functions} \label{s8100}
For $k\geq 1$, the Selberg function $\Lambda_k:\N\longrightarrow \R$ is defined by

\begin{equation} \label{eq8100.002}
	\Lambda_k(n)=
	\begin{cases}
		(-1)^k\sum_{d \mid n} \mu(d) \log^k(n/ d) & \text{if } \omega(n)\leq k,\\ 
		0 & \text{if } \omega(n)> k.\\ 
	\end{cases}
\end{equation}
The subsets of integers $\N_k=\{n\in \N: \omega(n)\leq k\}$ is the support of this function. The selberg function generalizes the earlier special case for $k=1$, known as the vonMangoldt function. The vonMangoldt function, a weighted prime power indicator function, is defined by
\begin{equation} \label{eq8100.52}
	\Lambda(n)=
	\begin{cases}
		\log p & \text{if } n=p^m,\\ 
		0 & \text{if } n\ne p^m.\\ 
	\end{cases}
\end{equation}
The above notation $p^m\geq 2$, with $ m \in \N$, denotes a prime power.\\

\begin{lem} \label{lem8100.001} Given a fixed $k\geq 1$,and any integer $n \in \mathbb{N}$, the function $\Lambda_k(n)$ has the inverse Mobius pair 
	$$
	\Lambda_k(n)=(-1)^k\sum_{d \mid n} \mu(d) \log^k(n/ d)   \qquad \qquad \text{ and } \qquad \qquad  \log^k(n)=\sum_{d \mid n} \Lambda_k(d) . 
	$$
\end{lem} 
\begin{proof} Use the inversion formula, \cite[Theorem 2.9]{AP1976}.  
\end{proof}

\begin{lem} \label{lem8100.005} Given a fixed $k\geq 1$, and any integer $n \in \mathbb{N}$, the function $\Lambda_k(n)$ has the inverse Mobius pair 
	\begin{enumerate}[font=\normalfont, label=(\roman*)]
		\item  $\displaystyle \Lambda_k(n)\geq 0, $ \tabto{8cm}  nonenegativity.\\
		\item $\displaystyle \Lambda_k(n)= 0,$\tabto{8cm}  if and only if $\omega(n)>k$.\\
		\item $\displaystyle \Lambda_k(n)\leq \log^k x ,  $\tabto{8cm} polynomial growth.
	\end{enumerate}
\end{lem}

\section{Partial Sums And Generating Functions} \label{s8150}
For a fixed $k\geq 1$, the partial sum of Selberg functions is defined by
\begin{equation}\label{eq8150.003}
	\psi_k(x)=\sum_{n \leq x} \Lambda_k(n)                 
\end{equation}
and the corresponding generting function is defined by
\begin{equation}\label{eq8150.007}
	\sum_{n\geq 1}\frac{\Lambda_k(n)}{n^s}=(-1)^k \zeta^{(k)}(s)\cdot \frac{1}{\zeta(s)}.                 
\end{equation}

\begin{thm}\label{thm8150.001}If $k\geq 1$ is a fixed integer, and $x\geq 1$ is a large number, then
	\begin{equation}\label{eq8150.003}
		\sum_{n \leq x} \Lambda_k(n) =xS_k(\log x)+O(x),                
	\end{equation}
	where $S_k(z)=kz^{k-1}+\cdots+a_1z+a_0\in \R[z]$ is a real polynomial. 
\end{thm}

There is a vast literature devoted to the case $k=1$, known as the prime number theorem, and some literature devoted to the case $k=2$, 
known as the elementary proof of the prime number theorem.

\section{Sums Over The Primes} \label{s2000} 
The most basic finite sum over the prime numbers is the prime harmonic sum $\sum_{p \leq x}1/p$. The refined estimate of this finite sum, stated below, is a synthesis of various results due to various authors. 

\begin{lem} \label{lem2000.1}
	Let \(x\geq 2\) be a large number, then
	\begin{enumerate} [font=\normalfont, label=(\roman*)]
		\item  $\displaystyle \sum_{p \leq x} \frac{1}{p}
		=\log \log x+ B_1 +O\left (e^{-c_0 \sqrt{ \log x}}\right ) , $ \tabto{8cm}  unconditionally.
		\item 
		$\displaystyle
		\sum_{p \leq x} \frac{1}{p}
		=\log \log x+ B_1 +\Omega_{\pm} \left (\frac{\log \log \log x}{x^{1/2}\log x} \right ),$\tabto{8cm}  unconditional oscillation.
		\item $\displaystyle
		\sum_{p \leq x} \frac{1}{p}
		=\log \log x+ B_1 +O\left (\frac{\log x}{ x^{1/2}} \right ) ,  
		$\tabto{8cm} conditional on the RH.
	\end{enumerate}
	where $B_1 = 0.2614972128 \ldots, $ is Mertens constant, and $c_0>0$ is an absolute constant. 
\end{lem}

\begin{proof} Replace the logarithm integral $\li(x)=\int_2^x (\log t)^{-1}dt$, and the appropriate prime counting measure $\pi(x)$ in Theorem \ref{thm8200.200} into the Stieltjes integral representation
	\begin{equation}
		\sum_{p \leq x}  \frac{1}{p}=\int_2^{x}\frac{1}{t}d\pi(t) 
	\end{equation}
	and evaluate it.\\
	
	(i) The unconditional part of the prime counting formula arises from the delaVallee Poussin form $\pi(x)=\li(x)+O\left (xe^{-c_0 \sqrt{\log x}}\right )$ of the prime number theorem, see \cite[p.\ 179]{MV2007}.\\
	
	(ii) The unconditional oscillations part arises 
	from the Littlewood form $\pi(x)=\li(x)+\Omega_{\pm} \left (x^{1/2}\log \log \log  x /\log x \right )$ of the prime number theorem, consult \cite[p.\ 51]{IA2003}, \cite[p.\ 479]{MV2007}, et cetera. \\
	
	(iii)  The conditional part arises from the Riemann form $\pi(x)=\li(x)+O\left (x^{1/2}\log^2 x \right )$ of the prime number theorem. 
\end{proof}

The asymptotic order $\sum_{p \leq x}1/p \sim \log \log x$ is due to Euler, confer \cite[Chapter 15]{EL88}. The earliest version including error term $\sum_{p \leq x}1/p=\log \log x+ B_1 +O(1/ \log x)$ is due to Mertens, see \cite{VM2005}. The qualitative form of the oscillations of the differences 
\begin{equation} \label{eq2000.8}
	\sum_{p^k \leq x} \frac{1}{p^k}-\left ( \log \log x + \gamma \right ) \quad \text{ and} \quad \sum_{p^k \leq x} \frac{\log p}{p^k}-\left ( \log x + \gamma \right )
\end{equation} 
seems to be due to Phragmen, confer \cite[p.\ 182]{NW2000}. \\

The Euler constant and Mertens constant occur very frequently in analysis. The former is defined by, (twenty four digits accuracies), 
\begin{equation}
	\gamma=\lim_{x \to \infty} \left ( \sum_{n \leq x}\frac{1}{n}-\log x \right ) =0.577215664901532860606512 \ldots,
\end{equation}
and the later 
is defined by 
\begin{equation}
	B_1=\lim_{x \to \infty} \left ( \sum_{p \leq x}\frac{1}{p-1}-\log \log x \right ) =0.261497212847642783755426\ldots.
\end{equation}
Other definitions of these constants are available in the literature, confer \cite{FS2003}.

\begin{lem} \label{lem2000.2} The constants $\gamma$ and $B_1$ satisfy the linear relation
	\begin{equation}
		B_1=\gamma-\sum_{p \geq 2} \sum_{n \geq 2} \frac{1}{np^{n}}.                 
	\end{equation}
	
\end{lem}

\begin{proof} This relation stems from the power series expansion 
	\begin{equation}
		B_1-\gamma=\sum_{p \geq 2} \left (\log \left (1-\frac{1}{p} \right ) + \frac{1}{p} \right )
	\end{equation}
	via the power series for $\log(1+z)$ with $|z|<1$. This leads to this identity, see \cite[p.\ 466]{HW2008}, \cite[p.\ 182]{MV2007}.\end{proof}

\begin{lem} \label{lem2000.3} Let \(x\geq 2\) be a large number, then
	\begin{equation}
		\sum_{p \leq x} \sum_{n \geq 2} \frac{1}{np^{n}} =\gamma-B_1-\frac{1}{x \log x}+ O \left (\frac{1}{x\log^2 x}\right ). 
	\end{equation}
\end{lem}

\begin{proof} Rearrange the power series expansion as
	\begin{equation}
		\sum_{p \leq x} \sum_{n \geq 2} \frac{1}{np^{n}} =\gamma-B_1- \sum_{p > x} \sum_{n \geq 2} \frac{1}{np^{n}} =\gamma-B_1-\frac{1}{x \log x}+ O \left (\frac{1}{x\log^2 x}\right ). 
	\end{equation}
	The estimate for the last two terms on the right follows from Lemma \ref{lem2000.4} computed below.
\end{proof} 

\begin{lem} \label{lem2000.4}  Let $x \geq 1$ be a large number, then
	\begin{equation}
		\sum_{p >x}\sum_{n \geq 2} 
		\frac{1}{np^{n}} = \frac{1}{x \log x}+ O \left (\frac{1}{x\log^2 x}\right ).
	\end{equation}
\end{lem}
\begin{proof} Split the infinite sum into two subsums:
	\begin{eqnarray}
		\sum_{p > x} \sum_{n \geq 2} \frac{1}{np^n}&=&\sum_{p > x}  \frac{1}{2p^2}+\sum_{p > x} \sum_{n \geq 3} \frac{1}{np^n} \nonumber \\
		&=&\sum_{p \geq x}  \frac{1}{2p^2}+O \left (\frac{1}{x^2 \log x}\right) \nonumber.
	\end{eqnarray}
	Employ the prime counting measure $\pi(t)=\#\{p \leq t\}$ to evaluate the first subsum using the integral 
	\begin{eqnarray}
		\sum_{p \geq x}  \frac{1}{p^2}&=&\int_x^{\infty}\frac{1}{t^2}d\pi(t) \nonumber \\
		&=&-\frac{\pi(x)}{x^2} +2\int_x^{\infty}\frac{\pi(t)}{t^3}dt \nonumber \\
		&=&\frac{1}{x \log x}+O \left (\frac{1}{x \log^2 x}\right) \nonumber .
	\end{eqnarray} 
\end{proof}

A generalized Mertens theorem to products of rational primes was recently proved by a few authors, see \cite{PD2016} and \cite{TG2019}.

\begin{thm} \label{thm2000.11}
	Let $x\geq2$ be a large number, and let $k\geq 1$. Then,
	\begin{equation}\label{eq8080.23?}
		\sum_{p_1p_2\cdots p_k \leq x}\frac{1}{p_1p_2\cdots p_k} =P_k(\log \log x) +O\left (\frac{(\log \log x)^{k-1}}{\log x}\right ) ,
	\end{equation}
	where $P_k(z)\in \R[z]$ is a polynomial of degree $\deg P_k=k$.
\end{thm}
The first two polynomials are these:
\begin{enumerate}
	\item $P_1(z)=z+B,$ 
	\item $P_2(z)=(z+B)^2-\pi^2/6$, 
\end{enumerate}
where $B=B_1=0.26\ldots$ is Mertens constant. This results is very useful in the calculations of the moments of the prime counting function $\omega(n)$, the second moment is required in the proof of variance of the arithmetic function $\omega(n)$.

\section{Products Over The Primes} \label{s2003} 
The asymptotics for a variety of interesting products are simple applications of the results for prime harmonic sums in the previous section.
\begin{lem} \label{lem2003.1}
	Let \(x\geq 2\) be a large number, then
	\begin{enumerate} [font=\normalfont, label=(\roman*)]
		\item  $\displaystyle \prod_{p \leq x}\left( 1- \frac{1}{p} \right) ^{-1}
		=e^{\gamma} \log x+ O\left (e^{-c_0 \sqrt{ \log x}}\right ) , $ \tabto{9cm}  unconditionally.
		\item 
		$\displaystyle
		\prod_{p \leq x}\left( 1- \frac{1}{p} \right) ^{-1}
		=e^{\gamma} \log x+\Omega \left (\frac{\log \log \log x}{x^{1/2}} \right ),$\tabto{9cm} unconditional oscillation.
		\item $\displaystyle
		\prod_{p \leq x}\left( 1- \frac{1}{p} \right) ^{-1}
		=e^{\gamma} \log x+ O\left (\frac{\log x}{ x^{1/2}} \right ) ,  
		$\tabto{9cm} conditional on the RH.
	\end{enumerate}
	where $\gamma$ is Euler constant, and $c_0>0$ is an absolute constant. 
\end{lem}
The results for products over arithmetic progression are proved in \cite{LZ2007}, et alii.\\

\begin{lem} \label{lem2003.701}
	Let \(x\geq 2\) be a large number, then
	\begin{enumerate} [font=\normalfont, label=(\roman*)]
		\item  $\displaystyle \prod_{p \leq x}\left( 1+ \frac{1}{p} \right) 
		=\frac{6e^{\gamma}}{\pi^2} \log x+ O\left (e^{-c_0 \sqrt{ \log x}}\right ) , $ \tabto{9cm}  unconditionally.
		\item 
		$\displaystyle
		\prod_{p \leq x}\left( 1+ \frac{1}{p} \right) 
		=\frac{6e^{\gamma}}{\pi^2} \log x+\Omega \left (\frac{\log \log \log x}{x^{1/2}\log x} \right ),$\tabto{9cm} unconditional oscillation.
		\item $\displaystyle
		\prod_{p \leq x}\left( 1+ \frac{1}{p} \right) 
		=\frac{6e^{\gamma}}{\pi^2} \log x+ O\left (\frac{\log x}{ x^{1/2}} \right ) ,  
		$\tabto{9cm} conditional on the RH.
	\end{enumerate}
	where $\gamma$ is Euler constant, and $c_0>0$ is an absolute constant. 
\end{lem}
\begin{proof} (i) For a large real number $x \in \R$, rewrite the product as
	\begin{eqnarray}\label{eq2003.701}
		\prod_{p \leq x}\left( 1+ \frac{1}{p} \right)&=& \prod_{p \leq x}\left( 1+ \frac{1}{p} \right)\left( 1- \frac{1}{p} \right) \left( 1- \frac{1}{p} \right)^{-1} \\
		&=& \prod_{p \leq x}\left( 1- \frac{1}{p^2} \right)\prod_{p \leq x}\left( 1- \frac{1}{p} \right)^{-1}\nonumber.
	\end{eqnarray}
	Replacing the completed product
	\begin{equation}\label{eq2003.705}
		\prod_{p \leq x}\left( 1- \frac{1}{p^2} \right)=\prod_{p \geq 2}\left( 1- \frac{1}{p^2} \right)+O\left ( \frac{1}{x^c}\right )=\frac{6}{\pi^2}+O\left ( \frac{1}{x^c}\right ) ,
	\end{equation}
	where $c\geq 1$ is a constant, in the first product on the right side of \eqref{eq2003.701}, and applying Lemma \ref{lem2003.1} yield
	\begin{eqnarray}\label{eq2003.707}
		\prod_{p \leq x}\left( 1- \frac{1}{p^2} \right)\prod_{p \leq x}\left( 1- \frac{1}{p} \right)^{-1}&=& \left ( \frac{6}{\pi^2}+O\left ( \frac{1}{x^c}\right )\right ) \prod_{p \leq x}\left( 1- \frac{1}{p} \right)^{-1}\\
		&=& \left ( \frac{6}{\pi^2}+O\left ( \frac{1}{x^c}\right )\right ) \left( e^{\gamma} \log x+ O\left (e^{-c_0 \sqrt{ \log x}}\right )   \right) \nonumber \\
		&=& \frac{6e^{\gamma}}{\pi^2} \log x+ O\left (e^{-c_0 \sqrt{ \log x}}\right )   \nonumber.
	\end{eqnarray}
	The verification of statements (ii) and (iii) are similar, mutatis mutandus. 
	
\end{proof}
The nonquantitative unconditional oscillations of the error of the product of primes is implied by the work of Phragmen, refer to equation (\ref{eq2000.8}), and \cite[p.\ 182]{NW2000}. Since then, various authors have developed quantitative versions, see \cite{RS1962}, \cite{DP2009}, \cite{LY2014}, \cite{LJ2015}, et alii. The specific quantitative form
\begin{equation} \label{eq3.10}
	\prod_{p \leq x}\left( 1- \frac{1}{p} \right) ^{-1}
	=e^{\gamma} \log x+ \Omega_{\pm}\left (\frac{f(x)}{ x^{1/2}} \right ), 
\end{equation}
where $f(x)$ is a slowly increasing function, was proved in \cite{DP2009}. \\

\begin{thm} \label{thm2003.2} { \normalfont (Martens, 1874)} The following asymptotic formulas hold:
	\begin{enumerate} [font=\normalfont, label=(\roman*)]
		\item  $\displaystyle \lim_{x \to \infty} \frac{1}{\log x} \prod_{p \leq x} \left (1-\frac{1}{p}\right )^{-1}=e^{\gamma}, $ \tabto{7cm}  the product over the integers.
		\item 
		$\displaystyle
		\lim_{x \to \infty} \frac{1}{\log x} \prod_{p \leq x} \left (1+\frac{1}{p}\right )^{-1}=\frac{6e^{\gamma}}{\pi^2}$\tabto{7cm} the product over the squarefree integers.
	\end{enumerate}
\end{thm}

\chapter{Mean Values of Arithmetic Functions} \label{c3300}

Various properties and important results on the mean values of arithmetic functions are discussed and surveyed in this chapter. The mean value results due to Wintner, Wirsing, Halasz, and a new result in comparative multiplicative functions due to Indlekofer, Katai and Wagner are the main topic. There is a well documented literature on this topic, see \cite[Section 3]{GR2008}, and \cite[Propositions 1 to 4]{IK2001} for introductions, and \cite{GS1999} for advanced materials.
\section{Some Definitions} \label{c3399}
Let $f : \N \to \C$ be a complex valued arithmetic function on the set of nonnegative integers. 
\begin{dfn} \label{dfn955.02} An arithmetic function $f:\N \longrightarrow \C$ is multiplicative if 
\begin{equation} \label{eq955.03}
f(mn)=f(m)f(n), \qquad \gcd(m,n),
\end{equation}
and completely multiplicative if \eqref{eq955.03} holds for all pairs of integers.
\end{dfn}
\begin{dfn} \label{dfn3399.02} The mean value of an arithmetic function is defined by
\begin{equation} \label{eq3399.003}
M(f)= \lim_{x \to \infty}\,\frac{1}{x}\sum_{n \leq x}f(n).
\end{equation}
\end{dfn}
The mean value of an arithmetic function is sort of a \textit{weighted density} of the subset of integers $supp(f)=\{n\in \N: f(n)\ne 0\} \subset \N$, which is the \textit{support} of the function $f$, see \cite[p.\ 46]{SS1994} for a discussion of the mean value. 
\begin{dfn} \label{dfn3399.12}
The \textit{natural density} of a subset of integers $\mathcal{A} \subset \N$ is defined by
\begin{equation}
\delta(\mathcal{A})= \lim_{x \to \infty}\,\frac{\{n \leq x: n \in \mathcal{A}\}}{x}.
\end{equation}
\end{dfn}

\section{Some Results For Arithmetic Functions}
This section investigates the two cases of convergent series, and divergent series
\begin{equation}
\sum_{n \geq 1}\frac{f(n)}{n} < \infty  \qquad \mbox{ and } \qquad  \sum_{n \geq 1}\frac{f(n)}{n} =\pm \infty,
\end{equation}
and the corresponding mean values
\begin{equation}
M(f)= \lim_{x \to \infty} \frac{1}{x}\sum_{n \leq x}f(n) =0  \qquad \mbox{ and } \qquad M(f)= \lim_{x \to \infty} \frac{1}{x}\sum_{n \leq x}f(n) \ne 0 \nonumber
\end{equation}
respectively. First, a result on the case of convergent series is considered here.

\begin{lem} \label{lem3399.01} Let $f : \N \longrightarrow \C$ be an arithmetic function. If the series  $\sum_{n \geq 1}f(n)n^{-1}$ converges, then its mean value  
\begin{equation}
M(f)= \lim_{x \to \infty} \frac{1}{x}\sum_{n \leq x}f(n) =0  
\end{equation}
vanishes.     
\end{lem}
                            
\begin{proof} Consider the pair of finite sums $\sum_{n \leq x}f(n) $, and $\sum_{n \leq x}f(n)n^{-1} $. By hypothesis, $\sum_{n \leq x}f(n)n^{-1}=c+o(1) $, where $c\ne 0$ is constant, for large $ x \geq  1$. Therefore
\begin{eqnarray}
\sum_{n \leq x}f(n) &=& \sum_{n \geq 1}n \cdot \frac{f(n)}{n}  \\
&=& \int_1^x t \cdot dR(t) \nonumber \\
&=&xR(x)-R(1)- \int_1^x R(t) dt  \nonumber \\
&=& o(x) \nonumber,
\end{eqnarray}
where $R(t)=\sum_{n \leq t}f(n)n^{-1} $. By the definition of the mean value of a function in (\ref{eq3399.003}), this confirms that $f(n)$ has mean value zero.         
\end{proof}   

This is standard material in the literature, see \cite[ p.\ 4]{PS2012}. The second case is covered by a few results on divergent series, which are considered next. \\

Let  $f(n)=\sum_{d \mid n}g(d)$, and let the series $c=\sum_{n \leq x}g(n)n^{-1} $ be absolutely convergent. Under these conditions, the mean value of the function $f(n)$ can be determined indirectly from the properties of the function $g(n)$.

\begin{thm} \label{thm3399.03}    {\normalfont (Wintner)}  Consider the arithmetic functions $f,g : \N \longrightarrow \C$, and assume that the associated generating series are zeta multiple
\begin{equation} \label{eq3399.043}
\sum_{n \geq 1}\frac{f(n)}{n^s} =\zeta(s) \sum_{n \geq 1}\frac{g(n)}{n^s}.
\end{equation}
Then, the followings hold.
\begin{enumerate} [font=\normalfont, label=(\roman*)]
\item If the series $\sum_{n \geq 1}f(n)n^{-s}$ is defined for $\Re e(s) > 1$, and the series $c=\sum_{n \leq x}g(n)n^{-1} $ is absolutely convergent, then the mean value, and the partial sum are given by
$$
M(f)= \lim_{x \to \infty} \frac{1}{x}\sum_{n \geq 1}\frac{g(n)}{n}  \quad \mbox{ and } \quad \sum_{n \leq x}f(n)=cx+o(x).
$$
\item If the series $\sum_{n \geq 1}f(n)n^{-s}$ is defined for $\Re e(s) > 1/2$, and the series $c=\sum_{n \leq x}g(n)n^{-1} $ is absolutely convergent, then, for any $\varepsilon > 0$, the mean value, and the partial sum are given by
$$
M(f)= \lim_{x \to \infty} \frac{1}{x}\sum_{n \geq 1}\frac{g(n)}{n}  \quad \mbox{ and } \quad \sum_{n \leq x}f(n)=cx+O\left (x^{1/2+\varepsilon}\right ).
$$
\end{enumerate}
\end{thm}

\begin{proof} (i) The partial sum of the series (\ref{eq3399.043}) is rearranged as 
\begin{eqnarray}
\sum_{n \leq x}f(n) &=& \sum_{n \leq x} \sum_{d \mid n} g(d)  \\
&=& \sum_{d \leq x} g(d)\sum_{n \leq x/d} 1  \nonumber \\
&=& \sum_{d \leq x} g(d) \left ( \frac{x}{d}- \left \{ \frac{x}{d} \right \} \right )  \nonumber \\
&=&x \sum_{d \leq x} \frac{g(d)}{d}+O \left (\sum_{d \leq x}|g(d)| \right ) \nonumber,
\end{eqnarray}
wheere $\{x\}=x-[x]$ is the fractional part function. The first line arises from the convolution of the power series $\zeta(s)=\sum_{n \geq 1}n^{-s}$, and $\sum_{n \geq 1}g(n)n^{-s}$. This is followed by reversing the order of summation. Moreover, the first finite sum is
\begin{equation} \label{eq3399.048}
\sum_{n \leq x}\frac{g(n)}{n} = \sum_{n \geq 1}\frac{g(n)}{n}+o(1) =c +o(1).
\end{equation}
because $\sum_{n \geq 1}g(n)n^{-1}$ is absolutely convergent. A use a dyadic method to split the second finite sum as
\begin{eqnarray}
\sum_{d \leq x}|g(d)|  &=& \sum_{d \leq x^{1/2}}\frac{|g(d)|}{d} \cdot d+ \sum_{ x^{1/2} \leq d \leq x }\frac{|g(d)|}{d} \cdot d \\
&\leq & x^{1/2}\sum_{d \leq x^{1/2}}\frac{|g(d)|}{d}+ x \sum_{ x^{1/2} \leq d \leq x }\frac{|g(d)|}{d}  \nonumber \\
&=& O\left ( x^{1/2} \right ) +o(x) \nonumber\\
&=& o(x) \nonumber.
\end{eqnarray}
Again, this follows from the absolute convergence $\sum_{n \geq 1}|g(n)|n^{-1}<\infty$.   
\end{proof}

Similar proofs appear in \cite[p.\ 138]{PA1988}, \cite[p. 83]{DL2012}, and \cite[p.\ 72]{HA2005}. Another derivation of the Wintner Theorem from the Wiener-Ikehara Theorem is also given in  \cite[p. \ 139]{PA1988}.

\begin{thm} \label{thm3399.07}    {\normalfont (Axer)}  Let $f,g : \N \longrightarrow \C$, be arithmetic functions and assume that the associated generating series are zeta multiple
\begin{equation} \label{eq3399.143}
\sum_{n \geq 1}\frac{f(n)}{n^s} =\zeta(s) \sum_{n \geq 1}\frac{g(n)}{n^s}.
\end{equation}
If the series $\sum_{n \geq 1}f(n)n^{-s}$ is defined for $\Re e(s) > 1$, and $\sum_{n \leq x}|g(n)|=O(x) $ is convergent, then the mean value, and the partial sum are given by
\begin{equation} \label{eq3399.145}
M(f)= \sum_{n \geq 1}\frac{g(n)}{n}  \quad \mbox{ and } \quad \sum_{n \leq x}f(n)=cx+o(x).
\end{equation}
\end{thm}

The goal of the next result is to strengthen Wintner Theorem by removing the absolutely convergence condition.

\begin{thm} \label{thm3399.09}    Let $f,g : \N \longrightarrow \C$, be arithmetic functions and assume that the associated generating series are zeta multiple
\begin{equation} \label{eq3399.243}
\sum_{n \geq 1}\frac{f(n)}{n^s} =\zeta(s) \sum_{n \geq 1}\frac{g(n)}{n^s}.
\end{equation}
If the series $\sum_{n \geq 1}f(n)n^{-s}$ is defined for $\Re e(s) > 1/2$, and the series $c=\sum_{n \leq x}g(n)n^{-1} $ is convergent, then,
\begin{equation} \label{eq3399.088}
M(f)= \sum_{n \geq 1}\frac{g(n)}{n}  \quad \mbox{ and } \quad \sum_{n \leq x}f(n)=cx+o(x).
\end{equation}
\end{thm}

\begin{proof} (i) The partial sum of the series (\ref{eq3399.243}) is rearranged as 
\begin{eqnarray}
\sum_{n \leq x}n \cdot f(n) &=& \sum_{n \leq x}n \sum_{d \mid n} g(d)  \\
&=& \sum_{d \leq x} g(d)\sum_{n \leq x/d} n  \nonumber \\
&=& \sum_{d \leq x} g(d) \left ( \frac{x^2}{2d}+o \left ( \frac{x^2}{2d} \right ) \right )  \nonumber \\
&=&\frac{x^2}{2}  \sum_{d \leq x} \frac{g(d)}{d}+o \left (\frac{x^2}{2} \left |\sum_{d \leq x}\frac{g(d)}{d} \right | \right ) \nonumber,
\end{eqnarray}
wheere $\{x\}=x-[x]$ is the fractional part function. The first line arises from the convolution of the power series $\zeta(s)=\sum_{n \geq 1}n^{-s}$, and $\sum_{n \geq 1}g(n)n^{-s}$. This is followed by reversing the order of summation. Moreover, the first finite sum is
\begin{equation} \label{eq3399.148}
\sum_{n \leq x}\frac{g(n)}{n} = \sum_{n \geq 1}\frac{g(n)}{n}+o(1) =c +o(1).
\end{equation}
because $\sum_{n \geq 1}g(n)n^{-1}$ is absolutely convergent. A use a dyadic method to split the second finite sum as
\begin{eqnarray}
\sum_{d \leq x}|g(d)|  &=& \sum_{d \leq x^{1/2}}\frac{|g(d)|}{d} \cdot d+ \sum_{ x^{1/2} \leq d \leq x }\frac{|g(d)|}{d} \cdot d \\
&\leq & x^{1/2}\sum_{d \leq x^{1/2}}\frac{|g(d)|}{d}+ x \sum_{ x^{1/2} \leq d \leq x }\frac{|g(d)|}{d}  \nonumber \\
&=& O\left ( x^{1/2} \right ) +o(x) \nonumber\\
&=& o(x) \nonumber.
\end{eqnarray}
Again, this follows from the absolute convergence $\sum_{n \geq 1}|g(n)|n^{-1}<\infty$. Lastly, but not least, the original partial sum  is recovered by partial summation.         
\end{proof}

The comparative multiplicative functions result due to Indlekofer, Katai, and Wagner is presented here. A related and simpler result has the following claim. 
\begin{thm} {\normalfont (Hall)} \label{thm482.01} Suppose that \(f:\mathbb{N}\longrightarrow \mathbb{C}\) is a multiplicative function with the following properties.
\begin{enumerate} [font=\normalfont, label=(\roman*)]
\item $\left |f(n) \right |\leq 1$ for all integers $n\in \mathbb{N}$. 
\item $f(n) \in \D=\{|z|\leq 1: z \in \C\}$, the unit disk on the complex plane. 
\end{enumerate}
Then
\begin{equation}\label{eq444.250}
\frac{1}{x} \left |	\sum _{n\leq x} f(n)\right | \ll e^{-cD(f,x)},
	\end{equation}
where $c>0$ is a small constant, and	
\begin{equation}  \label{eq444.255}
 D(f,x)=\sum_{p\leq x}  \frac{1-\Re e(f(p))}{p}.
\end{equation}
\end{thm}
\begin{proof} A recent proof appears in \cite[p.\ 3]{GS1999}.
\end{proof}

\begin{thm} {\normalfont (Halasz)} \label{thm482.66} Suppose that \(f:\mathbb{N}\longrightarrow \mathbb{C}\) is a multiplicative function with the following properties.
\begin{enumerate} [font=\normalfont, label=(\roman*)]
\item $\left |f(n) \right |\leq 1$ for all integers $n\in \mathbb{N}$. 
\item $f(n) \in \D=\{|z|\leq 1: z \in \C\}$, the unit disk on the complex plane. 
\end{enumerate}
Then, as $x \to \infty$,
\begin{equation}\label{eq444.260}
\frac{1}{x} \left |	\sum _{n\leq x} f(n)\right | \ll e^{-cD(f,x)},
	\end{equation}
where 
\begin{equation}  \label{eq444.235}
 M(T,x)=\min_{|y|\leq 2T }\sum_{p\leq x}  \frac{1-\Re e(f(p)p^{iy})}{p}.
\end{equation}
\end{thm}

\section{Wirsing Formula} \label{s444}
This formula provides decompositions of some summatory multiplicative functions as products over the primes supports of the functions. This technique works well with certain multiplicative functions, which have supports on subsets of primes numbers of nonzero densities. \\

\begin{thm} {\normalfont (\cite[p. 71]{WE1961})} \label{thm444.01}Suppose that \(f:\mathbb{N}\longrightarrow \mathbb{C}\) is a multiplicative function with the following properties.
\begin{enumerate} [font=\normalfont, label=(\roman*)]
\item $f(n) \geq 0$ for all integers $n\in \mathbb{N}$. 
\item $f\left(p^k\right)\leq c^k$ for all integers $k\in \mathbb{N}$, and $c<2$ constant. 
\item There is a constant $\tau >0$ such 
\begin{equation}\label{eq444.201}
\sum _{p\leq x} f(p)=(\tau +o(1)) \frac{x}{\log  x}
\end{equation} 
as $x \longrightarrow  \infty$.  
\end{enumerate}

Then
\begin{equation}\label{eq444.230}
	\sum _{n\leq x} f(n)=\left(\frac{1}{e^{\gamma \tau }\Gamma (\tau )}+o(1)\right)\frac{x}{\log  x}\prod _{p\leq x} \left(1+\frac{f(p)}{p}+\frac{f\left(p^2\right)}{p^2}+\cdots
	\right) .
	\end{equation}
	
\end{thm}

The gamma function appearing in the above formula is defined by \(\Gamma (s)=\int _0^{\infty }t^{s-1}e^{-s t}d t, s\in \mathbb{C}\). The intricate
proof of Wirsing formula appears in \cite{WE1961}. It is also assembled in various papers, such as \cite{HA1987}, \cite[p. 195]{PA1988}, and discussed in \cite[p. 70]{MV2007}, \cite[p. 308]{TG2015}. Various applications are provided in \cite{MP2011}, \cite{PS2013}, \cite{WK1975}, et alii. \\

\section{Result In Comparative Multiplicative Functions} \label{s955}
The mean value results due to Wirsing and Halasz, see \cite[Section 3]{GR2008}, and \cite[Propositions1 to 4]{IK2001} for an introduction to the basic theorems were expanded to a comparative multiplicative functions result in \cite{IK2001}.

\begin{thm} \label{thm955.88} {\normalfont (\cite[Theorem]{IK2001})}  Let $g$ be a multiplicative function satisfying $g(n) \geq 0$, and let $f$ be a complex-valued function, which satisfies $
\left | f(n)\right | \leq g(n)$ for every positive integer $n \in \N$. Let 
\begin{equation}  \label{eq955.08}
 \sum_{p\geq 2}  \frac{g(p)-\Re e(f(p)p^{-ia}}{p}.
\end{equation}
\begin{enumerate} [font=\normalfont, label=(\roman*)]
\item If there exists a real number $a_0 \in \R$ such that the series \eqref{eq955.08} converges for $a=a_0$, then
\begin{eqnarray}  \label{eq955.18}
 \sum_{n\leq x} f(n)&=&\frac{x^{ia_0}}{1+ia_0}    \prod_{p \leq x} \left ( 1+\sum_{m\geq 1} \frac{f(p^m)}{p^{m(1+ia_0)}} \right ) \\
 && \hskip 1 in \times \left ( 1+\sum_{m\geq 1} \frac{g(p^m)}{p^{m}} \right )^{-1}  \sum_{n\leq x}g(n)+o\left (  \sum_{n\leq x}g(n)\right )\nonumber 
\end{eqnarray}
as $x \longrightarrow \infty$.  
\item If the series \eqref{eq955.08} diverges for all $a \in \R$, then
\begin{equation}  \label{eq955.18}
 \sum_{n\leq x} f(n)=o\left (  \sum_{n\leq x}g(n)\right ) 
\end{equation}
as $x \longrightarrow \infty$.  
 \end{enumerate}
\end{thm}

\section{Extension To Arithmetic Progressions} \label{s4902}
The mean value of theorem arithmetic functions over arithmetic progressions $\{ qn + a : n \geq 1 \}$ facilitates another simple proof of Dirichlet Theorem.

\begin{thm} \label{thm4902.31}    {\normalfont }      Let $f(n)=\sum_{d \mid n} g(d)$, and let the series $\sum_{n \geq 1}g(n)n^{-1}\ne0$ be absolutely convergent. Then
\begin{equation} \label{eq257.03}
M(f)= \lim_{x \to \infty}\,\frac{1}{x}\sum_{\substack{n \leq x \\
n \equiv a \bmod q}}f(n)= \frac{1}{q} \sum_{d \mid q} \frac{c_d(a)}{d} \sum_{ n \geq 1} \frac{g(dn)}{n}.
\end{equation}
where $c_k(n)=\sum_{\gcd(x,k)=1}e^{i2 \pi nx/k}$.
\end{thm}

For the parameter $1 \leq a < q$, and $\gcd(a, q) = 1$, the mean value reduces to
\begin{equation} \label{eq257.03}
M(f)= \lim_{x \to \infty}\,\frac{1}{x}\sum_{\substack{n \leq x \\
n \equiv a \bmod q}}f(n)= \frac{1}{q} \sum_{d \mid q} \frac{\mu(d)}{d} \sum_{ n \geq 1} \frac{g(dn)}{n}.
\end{equation}
The proof is given in \cite[p.\ 143]{PA1988}, seems to have no limitations on the range of values of $q \geq 1$. Thus, it probably leads to an improvement on the Siegel-Walfisz Theorem, which states that 
\begin{equation} \label{eq257.22}
\pi(x,q,a)=\frac{1}{\varphi(q) }\frac{x}{ \log x}+O \left ( e^{-(\log x)^{\beta}} \right )
\end{equation}
 where $q = O\left ( \log^B x \right )$, with $B > 0$, and $0 < \beta <1$ constants.

\chapter{Autocorrelations of Mobius Functions} \label{c5}
The multiplicative function $\mu:\mathbb{N} \longrightarrow \{-1,0,1\}$ has sufficiently many sign changes to force meaningful cancellation on the twisted summatory functions $\sum_{n \leq x} \mu(n)f(n)$ for some functions $f: \mathbb{N} \longrightarrow \mathbb{C}$ such that $f(n) \ne \lambda(n), \mu(n), \mu_{*}(n)$. This randomness phenomenon is discussed in Conjecture \ref{conjMF222.190} in Section \ref{AMF222}. \\

\section{Current Result on Autocorrelation} \label{S474}
The current estimates of the logarithmic average orders of the autocorrelations of the Mobius function $\mu$ and the Liouville function $\lambda$ have the asymptotic formulae
\begin{equation} \label{eq479.037}
\frac{1}{\log x}  \sum_{n \leq x} \frac{\mu(n) \mu(n+a)}{n} =O \left (\frac{1}{(\log \log \log x)^c } \right ),
\end{equation} 
and 
\begin{equation} \label{eq479.137}
\frac{1}{\log x}  \sum_{n \leq x} \frac{\lambda(n) \lambda(n+a)}{n} =O \left (\frac{1}{(\log \log \log x)^c } \right ),
\end{equation} 
respectively, where $a \ne0$ is a fixed parameter, and $c>0$ is a constant, as $x \to \infty$, confer \cite{MR2015}, \cite{TT2015}, \cite[Chapter 5]{TJ2018}. Another attempt to prove the arithmetic average order of \eqref{eq479.037} was made in \cite{MA2016}, and an improved error term $O((\sqrt{ \log \log x})^{-2})$ was just proposed in \cite[Corollary 2]{HH2022}. Observe that a stronger error term, approximately $O((\log x)^{-2})$, is required to compute the arithmetic average orders directly from \eqref{eq479.037}, and \eqref{eq479.137}, see \cite[Exercise 2.12]{HA2013} for explanation. Here, the following result is considered. \\

\section{First Proof}\label{S474R}
The first proof of Theorem \ref{thm1010.100} is based on the explicit representation of the Mobius function in Lemma \ref{lemRMF525.200}. This technique offers a simpler analysis and sharper error term.

\begin{thm} \label{thm474.500R} Let $c>0$ be a constant, and let $\mu:\mathbb{N} \longrightarrow \{-1,0,1\}$ be the Mobius function. Then, for any sufficiently large number $x>1$, and a nonzero fixed integer $a \ne0$,
	\begin{equation} \label{eq474.500R}
		\sum_{n \leq x} \mu(n) \mu(n+a) =O \left (xe^{-c\sqrt{\log x}} \right )\nonumber. 
	\end{equation}	
\end{thm}

\begin{proof}[\textbf{Proof}] Without loss in generality, assume $a=1$. By Lemma \ref{lemRMF525.200}, the autocorrelation function has the equivalent form 
	\begin{equation}   \label{eq474.510R}
		\sum_{n \leq x} \mu(n) \mu(n+1) 
		=\sum_{n \leq x} \mu(n)\left( \sum_{\rho} \frac{(x+1)^\rho -x^\rho}{\rho\zeta^{\prime}(\rho) } + O(1)\right) , 
	\end{equation}
	where the index $\rho\in \C $ ranges through the nontrivial zeros of the zeta function $\zeta(s)$, as $x \to \infty$.	Since the new equivalent form in \eqref{eq474.510R} is a product of two independent finite sums, taking absolute values, and applying the upper bound of the Mertens function $\sum_{n \leq x} \mu(n)$, see Theorem \ref{thmMF222.050}, and Lemma \ref{lem474.300R}, yield the following. 
	\begin{eqnarray} \label{eq474.520R}
		\left |\sum_{n \leq x} \mu(n)\right |\left| \sum_{\rho} \frac{(x+1)^\rho -x^\rho}{\rho\zeta^{\prime}(\rho) } + O(1)\right| 
		&\ll&xe^{-c\sqrt{\log x}}\times O(1) \\
		&=&O\left( xe^{-c\sqrt{\log x}}\right)  \nonumber,  
	\end{eqnarray} 
	where $c>0$ is an absolute constant. 
\end{proof}

\begin{lem} \label{lem474.300R} If $x\geq1$ is a large number, then, 
	\begin{equation} \label{eq474.300R}
		\sum_{\rho} \frac{(x+1)^\rho -x^\rho}{\rho\zeta^{\prime}(\rho) } + O(1)=O(1),\nonumber 
	\end{equation}
	unconditionally, and independent of the zero of the zeta function.
\end{lem}
\begin{proof}[\textbf{Proof}] Using a binomial series expansion to approximate the difference 
	\begin{eqnarray}   \label{eq474.310R}
		(x+1)^\rho -x^\rho 
		&=&x^\rho\left(1+\frac{1}{x} \right)^\rho-x^\rho\\ 
		&=&\left( \frac{1}{2x} +O\left(\frac{1}{2\rho x^2}  \right)\right)\frac{x^\rho}{\rho}  \nonumber. 
	\end{eqnarray}
	
	Now, consider the zerofree region $\rho=1-c/\log t+it$, where $t\leq T$, and let $ T= e^{\sqrt{\log x}}$. and replace \eqref{eq474.310R}, to obtain  
	\begin{eqnarray}   \label{eq474.320R}
		\sum_{\rho} \frac{(x+1)^\rho -x^\rho}{\rho\zeta^{\prime}(\rho) } 
		&=&\left( \frac{1}{2x} +O\left(\frac{1}{2x^2}  \right)\right)\sum_{|\rho|\leq T} \frac{x^\rho}{\rho^2\zeta^{\prime}(\rho) } \\
		&=&\left( \frac{1}{2x} +O\left(\frac{1}{2x^2}  \right)\right)x^{1-c/\sqrt{\log x}}\sum_{|\rho|\leq T} \frac{x^{it}}{\rho^2\zeta^{\prime}(\rho) }\nonumber. 
	\end{eqnarray}
	Since the series 
	\begin{equation}\label{eq474.330R}
		\sum_{\rho} \frac{x^{it}}{\rho^2\zeta^{\prime}(\rho) }\ll \sum_{t\geq 1} \frac{1}{(1/2+it)^2|\zeta^{\prime}(\rho) |}<\infty	
	\end{equation}
	where the index $\rho\in \C $ ranges through the nontrivial zeros of the zeta function,	converges, taking absolute value returns
	\begin{eqnarray}   \label{eq474.340R}
		\left |	\sum_{\rho} \frac{(x+1)^\rho -x^\rho}{\rho\zeta^{\prime}(\rho) } \right |
		&=&\left |\left( \frac{1}{2x} +O\left(\frac{1}{2x^2}  \right)\right)x^{1-c/\sqrt{\log x}}\sum_{|\rho|\leq T} \frac{x^{it}}{\rho^2\zeta^{\prime}(\rho) }\right |\\
		\nonumber
		&\ll & e^{-c\sqrt{\log x}}\left |\sum_{|\rho|\leq T} \frac{x^{it}}{\rho^2\zeta^{\prime}(\rho) }\right |\\
		\nonumber\\
		&\ll & e^{-c\sqrt{\log x}}\nonumber,
	\end{eqnarray}
	where $c>0$ is an absolute constant. The claim follows from these information.
\end{proof}

\section{Conditional Upper Bounds}
The conditional upper bounds are derived from the optimal zerofree region $\{s \in \mathbb{C}: \Re e(s)>1/2 \}$ of the zeta function. \\

\begin{thm} \label{thm5.4} Suppose that $ \zeta(\rho)=0 \Longleftrightarrow \rho=1/2+it, t \in \mathbb{R}$. Let $\mu:\mathbb{N} \longrightarrow \{-1,0,1\}$ be the Mobius function. Then, for any sufficiently large number $x>1$, 
	\begin{equation}
		\sum_{n \leq x} \mu(n) \mu(n+1) =O \left (x^{1/2} \log^2 x \right )\nonumber.
\end{equation} \end{thm}
\begin{proof}[\textbf{Proof}] The proof is similar to the proof of Theorem \ref{thm474.500R}, but use the conditional result in Theorem \ref{thmMF222.050}.
\end{proof}

\begin{thm} \label{thm5.3} Suppose that $ \zeta(\rho)=0 \Longleftrightarrow \rho=1/2+it, t \in \mathbb{R}$. Let $\mu:\mathbb{N} \longrightarrow \{-1,0,1\}$ be the Mobius function. Then, for any sufficiently large number $x>1$, 
	\begin{equation}
		\sum_{n \leq x} \frac{\mu(n) \mu(n+1)}{n} =O \left (\frac{\log^2 x}{x^{1/2}} \right )\nonumber.
\end{equation} \end{thm}

\section{Second Proof}\label{S474B}
The second proof of Theorem \ref{thm474.100} is based on various identities and standard results in analytic number theory.

\begin{thm} \label{thm474.100} Let $c>0$ be a constant, and let $\mu:\mathbb{N} \longrightarrow \{-1,0,1\}$ be the Mobius function. Then, for any sufficiently large number $x>1$, and a nonzero fixed integer $a \in \Z$,
\begin{equation} \label{eq474.100}
\sum_{n \leq x} \mu(n) \mu(n+a) =O \left (\frac{x}{\log^{c}x} \right )\nonumber. 
\end{equation}	
\end{thm}

\begin{proof}[\textbf{Proof}] ({\bfseries Theorem \ref{thm1010.100}}) By Lemma \ref{lem297.51}, the autocorrelation function has the equivalent form 
\begin{equation}   \label{eq474.41}
\sum_{n \leq x} \mu(n) \mu(n+a) =\sum_{n \leq x} (-1)^{\omega(n)} \mu(n)^2 \mu(n+a). 
\end{equation}		
Applying Lemma \ref{lem297.82} and reversing the order of summation yield: 
\begin{eqnarray} \label{eq474.43}
\sum_{n \leq x} (-1)^{\omega(n)} \mu^2(n)\mu(n+a) 
&=&\sum_{n \leq x} \mu^2(n)\mu(n+a) \sum_{q\mid n} \mu(q)d(q) \\
&=&\sum_{q\leq x}\mu(q)d(q) \sum_{\substack{n \leq x\\ q\mid n}} \mu^2(n)\mu(n+a) \nonumber\\
&=&\sum_{q\leq x}\mu(q)d(q) \sum_{m\leq x/q} \mu^2(m)\mu(mq+a) \nonumber,  
\end{eqnarray} 
where $\gcd(m,q)=1$. Applying Lemma \ref{lem297.58} and reversing the order of summation, yield: 
\begin{eqnarray} \label{eq474.43}
S(x)&=&\sum_{q\leq x}\mu(q)d(q) \sum_{m\leq x/q} \mu^2(m)\mu(mq+a)\\
&=& \sum_{q\leq x}\mu(q)d(q) \sum_{m\leq x/q} \mu(mq+a)\sum_{d^2\mid m} \mu(d)\nonumber\\
&=&\sum_{q\leq x}\mu(q)d(q) \sum_{d^2 \leq x/q} \mu(d)\sum_{\substack{m \leq x/q\\ mq+a\equiv a \bmod d^2}} \mu(mq+a) \nonumber.  
\end{eqnarray}

Let $x_0=(\log x)^B $, where $ B>0$ is a constant, and consider the dyadic partition
\begin{eqnarray}  \label{eq474.47}
S(x)&=& \sum_{q\leq x}\mu(q)d(q) \sum_{d^2 \leq x/q} \mu(d)\sum_{\substack{m \leq x/q\\ mq+a\equiv a \bmod d^2}} \mu(mq+a)  \nonumber\\
&=&\sum_{q\leq x}\mu(q)d(q) \sum_{d^2 \leq x_0} \mu(d)\sum_{\substack{m \leq x/q\\ mq+a\equiv a \bmod d^2}} \mu(mq+a) \nonumber\\
&&\hskip 1 in+\sum_{q\leq x}\mu(q)d(q) \sum_{x_0<d^2 \leq x/q} \mu(d)\sum_{\substack{m \leq x/q\\ mq+a\equiv a \bmod d^2}} \mu(mq+a) \nonumber\\
&=&S_0(x)\;+\; S_1(x) \nonumber.   
\end{eqnarray} 
The upper bounds of the triple finite sums $S_0(x)$ and $S_1(x)$ are computed in Lemma \ref{lem474.55} and Lemma \ref{lem474.57}. Summing these estimates in (\ref{eq474.88}) returns
\begin{eqnarray}  \label{eq474.249}
\sum_{n \leq x} \mu(n) \mu(n+a)
&=&S_0(x)\;+\; S_1(x) \\
&=&O\left( \frac{x}{ (\log x)^{C-2} }\right )+O\left( \frac{x}{ (\log x)^{C-2+B/2} }\right )\nonumber\\
&=&O\left( \frac{x}{ (\log x)^{C-2} }\right )\nonumber ,   
\end{eqnarray}
where $C>2$ yields a nontrivial bound.
\end{proof}

\begin{lem} \label{lem474.55} For any fixed integer $a\ne0$, and an arbitrary constant $C>2$,
\begin{equation} \label{eq474.88}
S_0(x)=\sum_{q\leq x}\mu(q)d(q) \sum_{d^2 \leq x_0} \mu(d)\sum_{\substack{m \leq x/q\\ mq+a\equiv a \bmod d^2}} \mu(mq+a)\ll  \frac{x}{ (\log x)^{C-2}} \nonumber
\end{equation}
as the number $x \to \infty$.
\end{lem}
\begin{proof}[\textbf{Proof}] Let $x_0=(\log x)^B$, where $ B>0$ is a constant. Take absolute value and write the finite sum as
\begin{eqnarray}
\left |S_0 (x)\right |&=&\left |\sum_{q\leq x}\mu(q)d(q) \sum_{d^2 \leq x_0} \mu(d)\sum_{\substack{m \leq x/q\\ mq+a\equiv a \bmod d^2}} \mu(mq+a)\right |\\
&\leq&  \sum_{q\leq x}d(q) \sum_{r \leq x_0} \left |\sum_{\substack{m \leq x/q\\ mq+a \equiv a \bmod r}} \mu(mq+a) \right |\nonumber ,
\end{eqnarray}
where $r=d^2\leq (\log x)^B=x_0$. An application of Corollary \ref{corMMF525.550} to the inner double finite sum produces the upper bound
\begin{eqnarray}  \label{eq474.85}
\sum_{q\leq x}d(q) \sum_{r \leq x_0} \left |\sum_{\substack{m \leq x/q\\ mq+a \equiv a \bmod r}} \mu(mq+a) \right |
&\ll&  \sum_{q\leq x}d(q) \cdot \left (\frac{x}{q (\log x)^{C_1} } \right)\\
&\ll&  \frac{x}{ (\log x)^{C_1} } \sum_{ q\leq x}\frac{d(q)}{q}\nonumber\\
&\ll&\left (\frac{x}{(\log x)^{C_1}}\right)\cdot  \left ((\log x)^2\right) \nonumber, 
\end{eqnarray}
where the last inner finite sum $\sum_{q\leq x}d(q)/q\ll (\log x)^2$. Now, select a constant $C\geq C_1=C_1(B)>2$, which depends on $B>0$, to produce a nontrivial upper bound. 
\end{proof}

\begin{lem} \label{lem474.57} For any fixed integer $a\ne0$, and an arbitrary constant $C>2$,
\begin{equation} \label{eq474.89}
S_1(x) =\sum_{q\leq x}\mu(q)d(q) \sum_{x_0<d^2 \leq x/q} \mu(d)\sum_{\substack{m \leq x/q\\ mq+a\equiv a \bmod d^2}} \mu(mq+a) \ll \frac{x}{(\log x)^{C-2}} \nonumber
\end{equation}
as the number $x \to \infty$.
\end{lem}
\begin{proof}[\textbf{Proof}]  Let $x_0=(\log x)^B$, where $ B>0$ is a constant. The absolute maximal is
\begin{eqnarray}  \label{eq474.92}
\left | S_{1} (x)\right | &\leq &\sum_{q \leq x}d(q) \sum_{x_0<d^2 \leq x/q,}  \sum_{\substack{m\leq x/q\\  mq+a\equiv a \bmod d^2}} 1 \\
&\ll&  \sum_{q\leq x}d(q) \sum_{x_0^{1/2}<d\leq  x^{1/2}}\frac{x}{q d^2} \nonumber\\
&\ll&  x \left (\frac{1}{ x_0^{1/2}}-\frac{1}{ x^{1/2}}\right )  \sum_{ q\leq x}\frac{d(q)}{q}\nonumber\\
&\ll&x \left (\frac{1}{ x_0^{1/2}}- \frac{1}{ x^{1/2}}\right )  (\log^2x) \nonumber\\
&\ll&\frac{x}{(\log x)^{B/2}}\cdot  (\log x)^2 \nonumber, 
\end{eqnarray}
where the sum $\sum_{q\leq x}d(q)/q\ll (\log x)^2$. Now, select a constant $C= B/2-2>0$ to realize a nontrivial upper bound.
\end{proof}

A different approach using exact formula is sketched in problem 5.7, in the Problems subsection. This similar to the proof of Lemma 2.17 in \cite[p. 66]{MV2007}.\\

\section{Correlation For Squarefree Integers of Degree Two}

\begin{lem}\label{lem800.41}  Fix an integer $t$. Let $x\geq 1$ be a large number, and let $\mu: \mathbb{Z} \longrightarrow \{-1,0,1\}$ be the Mobius function. Then, for some constant $c>0$,
\begin{equation}\label{eq800.74}
\sum_{n \leq x}\mu(n)^2 \mu(n+t)^2=cx+O\left (x^{1/2} \right )\nonumber.
\end{equation}
\end{lem}

\begin{proof} Start with the identity $\mu^2(n)=\sum_{d^2 \mid n}\mu(d)$, and substitute it:
\begin{eqnarray}\label{eq800.19}
\sum_{n \leq x}\mu(n)^2 \mu(n+t)^2&=&\sum_{n \leq x} \mu(n+t)^2\sum_{d^2 \mid n}\mu(d) \\
&=&\sum_{d \leq x^{1/2}}\mu(d) \sum_{n \leq x, d^2 \mid n} \mu(n+t)^2 \nonumber .
\end{eqnarray}
Applying Lemma \ref{lem3.10}, returns
\begin{eqnarray}\label{eq800.24}
\sum_{d \leq x^{1/2}}\mu(d) \sum_{n \leq x, d^2 \mid n} \mu(n+t)^2&=&\sum_{d \leq x^{1/2}}\mu(d)  \left ( \frac{c_0}{d^2}\frac{x}{d^2}   
+O\left (\frac{x^{1/2}}{d}  \right ) \right ) \\
&=&c_0x\sum_{d \leq x^{1/2}}\frac{\mu(d)}{d^4}+ O\left (x^{1/2} \sum_{d \leq x^{1/2}}\frac{1}{d} \right ) \nonumber \\
&=&c_0x+ O\left (x^{1/2} \log x \right ) \nonumber,
\end{eqnarray}
where the constant is
\begin{equation}\label{eq800.74}
c=c_0\sum_{n \geq 1}\frac{\mu(d)}{d^4}.
\end{equation}

\end{proof}

The earliest result in this direction appears to be
\begin{equation}\label{eq800.74}
\sum_{n \leq x}\mu(n)^2 \mu(n+t)^2=cx+O\left (x^{2/3} \right ).
\end{equation}
in \cite{ML1948}.

\begin{lem}\label{lem800.82} Fix an integer $t\ne 0$. Let $x\geq 1$ be a large number, and let $\mu: \mathbb{Z} \longrightarrow \{-1,0,1\}$ be the Mobius function. Then
\begin{enumerate} [font=\normalfont, label=(\roman*)]
\item  For a constant $c_1=(12-c)/\pi^2>0$, the error term satisfies
$$ \label{eq800.66}
\left |\sum_{n \leq x}\left (\mu(n)+\mu(n+t)\right )^4-c_1x -6\sum_{n \leq x}\mu(n)\mu(n+t) \right | =O\left (x^{1/2} \right ) .
$$
\end{enumerate}
\end{lem}

\begin{proof} Expanding the polynomial yields
\begin{eqnarray}
&& \sum_{n \leq x}\left (\mu(n)+\mu(n+k)\right )^4 \\ &=&\sum_{n \leq x}\mu(n)^4+4\sum_{n \leq x}\mu(n)^3\mu(n+k)+6\sum_{n \leq x}\mu(n)^2\mu(n+k)^2\nonumber 
\\
&&+4\sum_{n \leq x}\mu(n)\mu(n+k)^3+\sum_{n \leq x}\mu(n+k)^4 \nonumber\\
&=&\sum_{n \leq x}\mu(n)^2+4\sum_{n \leq x}\mu(n)\mu(n+k)+6\sum_{n \leq x}\mu(n)^2\mu(n+k)^2\nonumber \\
&&+4\sum_{n \leq x}\mu(n)\mu(n+k) +\sum_{n \leq x}\mu(n+k)^2 \nonumber.
\end{eqnarray}
Replace  $\sum_{n \leq x}\mu(n)^2=6\pi^{-2}x+O(x^{1/2})$, and $\sum_{n \leq x}\mu(n)^2\mu(n+t)^2=cx+O(x^{1/2})$, see Lemma \ref{lem800.41}, and rearrange them: 
\begin{eqnarray}
&& \sum_{n \leq x}\left (\mu(n)+\mu(n+t)\right )^4 -8\sum_{n \leq x}\mu(n)\mu(n+t)  \nonumber \\
&=& \sum_{n \leq x}\mu(n)^2+\sum_{n \leq x}\mu(n+t)^2 +6\sum_{n \leq x}\mu(n)^2\mu(n+t)^2\nonumber \\
&=&  \frac{6}{\pi^2}x+O(x^{1/2})+\frac{6}{\pi^2}x+O(x^{1/2}) +\frac{c}{\pi^2}x+O(x^{1/2}) \nonumber.
\end{eqnarray}
Set $c_1=12/\pi^2-c>0$. Rearranging the terms and taking absolute value completes the verification.
\end{proof}

\section{Problems}

\begin{exe} {\normalfont Show that $\mu(q)\mu(qm+1)= \mu(q)^2 \mu(m)$ for all $m,q\geq 1$, implies correlation. For example, it has a much larger upper bound
$$
\left | \sum_{q\leq x} \frac{\mu(q)d(q)}{q} \sum_{m \leq x/q} \frac{\mu(qm+1)}{m} \right |
\leq (\log x)\left |\sum_{q\leq x} \frac{\mu(q)^2 d(q)}{q} \right | \ll x^{\varepsilon},
$$
for some $\varepsilon>0$, which implies correlation. }
\end{exe}

\begin{exe} {\normalfont The evaluation of the right side of the series
$$
\sum_{q\leq x} \frac{\mu(q)^2d(q)}{q} \sum_{m \leq x/q, \gcd(m,q)=1} \frac{\mu(m)}{m} =\sum_{n\leq x} \frac{\mu(n)^2}{n}.
$$
is well known, that is, the evaluation of the left side reduces to $\sum_{n\leq x} \mu(n)^2) /n=6 \pi^{-2} \log x +O(x^{-1/2})$. Use a different technique to find the equivalent evaluation on the left side. Hint: try the inverse Dirichlet series
$$
\frac{1}{L(s,\chi_{0})}= \frac{1}{\zeta(s)} \prod_{p |q} \left (1-\frac{1}{p^s} \right )^{-1},
$$
where $\chi_{0}=1$ is the principal character  mod $q$, see \cite[p.\ 334]{MV2007}.}
\end{exe}

\begin{exe} {\normalfont Verify that $\mu(q) \mu(qm+1)=-1$ for $\gg x/q\log x$ integers $m,q \in [1,x]$. This proves that $\mu(q)\mu(qm+1) \ne \mu(q)^2 \mu(m)$ for all $m,q\geq 1$.}
\end{exe}

\begin{exe} {\normalfont Let $f(x) \in \mathbb{Z}[x]$ be an irreducible polynomial of degree $\deg(f)=2$. Estimate $\sum_{n\leq x} \mu(f(n))$, consult \cite{BM2009} for related works.}
\end{exe}

\begin{exe} {\normalfont Let $\mu(q)\mu(qm+1)\ne \mu(q)^2 \mu(m)$ for all $m,q\geq 1$. Use an exact formula for the inner sum such as
$$
\sum_{m \leq x/q} \frac{\mu^2(qm)\mu(qm+1)}{m}=\frac{\varphi(q)}{q} \left (\log(x/q)+\gamma(q)+O(q/x) \right ),
$$
where $\gamma(q)$ is a constant depending on $q\geq1$, to prove Lemma 5.1:
$$
\sum_{n\leq x} \frac{\mu(n) \mu(n+1)}{n}=\sum_{q\leq x} \frac{\mu(q)d(q)}{q} \sum_{m \leq x/q} \frac{\mu^2(qm)\mu(qm+1)}{m}
=O \left ( \frac{1}{\log^{C-1} x} \right).
$$
Hint: Compare this to the proof for Lemma 2.17 in \cite[p.\ 66]{MV2007}.}
\end{exe}

\begin{exe} {\normalfont Let $x\geq 1$ be a large number. Find an asymptotic formula for the finite sum
$$
\sum_{n \leq x} \mu^2(n)\mu(n+1)^2\stackrel{?}{=}\frac{6^2}{\pi^4} x+c_1+O(x^{1/2}),
$$
where $c_1$ is a constant.}
\end{exe}

\begin{exe} {\normalfont Let $x\geq 1$ be a large number, and let $k \ne 0$. Find an asymptotic formula for the finite sum
$$
\sum_{n \leq x} \mu^2(n)\mu(n+k)^2\stackrel{?}{=}\frac{6^2}{\pi^4} x+c_k+O(x^{1/2}),
$$
where $c_k$ is a constant. Hint: there is no cancelations, so find a way or an argument to prove that $\sum_{n \leq x} \mu^2(n)\mu(n+k)^2\stackrel{?}{\ne}o(x)$. }
\end{exe}

\chapter{Autocorrelation of the vonMangoldt Function} \label{c777}

The autocorrelation of the vonMangoldt function is a long standing problem in number theory. It is the foundation of many related conjectures concerning the distribution of primes pairs $p$ and $p+2k$ as $p \to \infty$, and the prime $m$-tuples 
\begin{equation}
p, \quad p+a_1, \quad p+a_2, \quad \ldots, \quad p+a_m,
\end{equation}
where $k \geq 1$ is fixed, and $a_1,a_2, \ldots, a_m$ is an admissible finite sequence of relatively prime integers respectively. \\

The qualitative version of the prime pairs conjecture is attributed to dePolinac, \cite{PJ2009}, and \cite{GD2009}. The heuristic for the quantitative prime pairs conjecture, based on the circle method, appears in \cite[p. 42]{HL1923}. The precise statement is the following. Let 
\begin{equation}
\pi_{2k}(x)=\#\{p \leq x: p \;\text{and} \;p+2k\; \text{are primes}\}
\end{equation}
be the prime pair measure.

\begin{conj} { \normalfont Prime Pairs Conjecture}  There are infinitely many prime pairs $p$, and $p+2k$ as the prime $p \to \infty$. Moreover, the counting function has the asymptotic formula
\begin{equation}
\pi _{2k}(x)=2 \prod_{p \mid 2k} \left( \frac{p-1}{p-2} \right ) \prod _{p \,\nmid \,2k} \left(1-\frac{1}{(p-1)^2}\right) \frac{x}{\log ^2 x}+O\left (\frac{x}{\log ^3 x}\right),
\end{equation}
where \(k \geq 1\) is a fixed integer, and $x \geq 1$ is a large number.
\end{conj}

The density constant
\begin{equation}
\mathfrak{G}(2k)=\lim_{x \to \infty} \frac{1}{x} \frac{\#\{p \leq x: p \;\text{and} \;p+2k\; \text{are primes}\}}{\#\{p \leq x: p \; \text{is prime}\}}
\end{equation}

arises from the singular series
\begin{equation}
\mathfrak{G}(2k)=\prod_{p\geq 2} \left (1-\frac{v_p(f)}{p} \right ) \left (1-\frac{1}{p} \right )^{-2}=2\prod_{p \mid  2k} \left( \frac{p-1}{p-2} \right ) \prod _{p \,\nmid \,2k} \left(1-\frac{1}{(p-1)^2}\right),
\end{equation}
where $v_p(f)$ is the number of root in the congruence $x(x+2k)\equiv 0 \text{ mod }p$.\\

The heuristic based on probability, yields the Gaussian type analytic formula
\begin{equation}
\pi _{2k}(x)=a_{2k} \int_{2}^{x} \frac{1}{\log^2 t} dt+O\left(\frac{x}{\log ^3 x}\right),
\end{equation}
where $a_{2k}=\mathfrak{G}(2k)$ is the same density constant as above. The deterministic approach is based on the weighted prime pairs counting function
\begin{equation} \label{eq600.13}
\sum_{n\leq x}\Lambda(n)\Lambda(n+t),
\end{equation}
which is basically an extended version of the Chebyshev method for counting primes.\\

For an odd integer $t\geq 1$, the average order of the finite sum (\ref{eq600.13}) is very small, and the number of prime pairs $p$, and $p + t$ is small. This is demonstrated in Theorem \ref{thm773.8}. Another way of estimating this is shown in \cite[p. 507]{GP1999}. Consequently, it is sufficient to consider even integer $t = 2k$.\\

\section{Correlation Functions For Twin Primes And Prime Pairs} \label{s6668}
The qualitative version of the prime pairs is attributed to dePolinac, see the survey in \cite{PJ2009}. The heuristic for the quantitative prime pairs conjecture, based on the circle method, appears in \cite[p. 42]{HL1923}. The precise statement is the following.\\ 

\begin{conj} \label{cj6668.90}{\normalfont (Prime Pairs Conjecture)}  There are infinitely many prime pairs $p, p+2k$ as the prime $p \to \infty$. Moreover, the counting function has the asymptotic formula
	\begin{equation} \label{eq6668.31B}
		\pi _{2k}(x)=2 \prod_{p |2k} \left( \frac{p-1}{p-2} \right ) \prod _{p \nmid 2k} \left(1-\frac{1}{(p-1)^2}\right) \frac{x}{\log ^2 x}+O\left(\frac{x}{\log ^3 x}\right),
	\end{equation}
	where \(k \geq 1\) is a fixed integer, and $x \geq 1$ is a large number.
\end{conj}

The constant arises from the singular series
\begin{equation}  \label{eq6668.33}
	\mathfrak{G}(2k)=\prod_{p\geq 2} \left (1-\frac{v_p(f)}{p} \right ) \left (1-\frac{1}{p} \right )^{-2}=2\prod_{p |2k} \left( \frac{p-1}{p-2} \right ) \prod _{p \nmid 2k} \left(1-\frac{1}{(p-1)^2}\right),
\end{equation}
where $v_p(f)$ is the number of root in the congruence $x(x+2k)\equiv 0 \text{ mod }p$ for $p \geq 2$.\\

The heuristic based on probability, yields the Gaussian type analytic formula
\begin{equation} \label{eq6668.35}
	\pi _{2k}(x)=C_{2k} \int_{2}^{x} \frac{1}{\log^2 t} dt+O\left(\frac{x}{\log ^3 x}\right),
\end{equation}
where $C_{2k}$ is the same constant as above. The deterministic approach is based on the weighted prime pairs counting function
\begin{equation}  \label{eq6668.37}
	\sum_{n\leq x}\Lambda(n)\Lambda(n+m),
\end{equation}
which is basically an extended version of the Chebyshev method for counting primes.\\

For an odd integer $m\geq 1$, the average order of the finite sum $\sum_{n\leq x}\Lambda(n)\Lambda(n+m)$ is very small, and the number of prime pairs $p, p + m$ is finite. Consequently, it is sufficient to consider even integer $m = 2k$. The best known case is for $k=1$.

\begin{conj} \label{cj6668.93}{\normalfont (Twin Prime Conjecture)}  There are infinitely many twin prime $p$ and $p+2$ as the prime $p \to \infty$. Moreover, the counting function has the asymptotic formula
	\begin{equation} \label{eq6668.31}
		\pi _{2}(x)=2  \prod _{p \geq 3} \left(1-\frac{1}{(p-1)^2}\right) \frac{x}{\log ^2 x}+O\left(\frac{x}{\log ^3 x}\right),
	\end{equation}
	as $x \to \infty$. The twin prime constant is
	\begin{equation}  \label{eq6668.35B}
		\mathfrak{G}(2)=2 \prod _{p \geq 3} \left(1-\frac{1}{(p-1)^2}\right)=2(0.66016 18158 46869 57392 78121 10014\dots).
	\end{equation}
\end{conj}

\section{Autocorrelation of the vonMangoldt Function}

The most recent progress in the theory of the correlation of the von Mangoldt function are the various results based on the series of papers authored by \cite{GP2005}. The results for double auto correlation states the following.

\begin{thm} \label{thm773.15}  {\normalfont (\cite[Theorem 5.1]{GP2003})} Let $x \geq 1$ be a sufficiently large number. For $1 \leq R \leq x$, and $0\leq  |k|\leq R$, let $C>0$ be a constant.
\begin{enumerate} [font=\normalfont, label=(\roman*)]
\item If $k=0$, then\\
$$  \sum_{n \leq x} \Lambda_R(n)^2=x \log x +O \left (x\right ) +O \left (R^2\right ).$$

\item If $k \ne 0$, then \\
$$ \sum_{n \leq x} \Lambda_R(n)\Lambda_R(n+k) =\mathfrak{G}(k)x +O \left (\frac{k}{\varphi(k)} \frac{x}{(\log 2R/k)^{C}}\right ) +O \left (R^2\right ).$$ 
\end{enumerate}

\end{thm}

The definition of $\Lambda_R(n)$ appears in (\ref{eq600.42}). Clearly, this result is not effective for $R > (x \log x)^{1/2}$. Another result in \cite[Theorem 1.3]{MR2017} has a claim for average Hardy-Littlewood conjecture, it covers almost all shift of the correlation, but no specific value. A synopsis of the statement is recorded below.\\

\begin{thm} \label{thm773.18}  For some $X\geq 2$, let $C>0$, $0 < \varepsilon<1/2$, and let $0 \leq h_0 <X^{1-\varepsilon}$. Suppose that 
$X^{8/33+\varepsilon} \leq H \leq X^{8/33+\varepsilon}$. Then
\begin{equation} \label{eq773.502}
\sum_{n \leq x} \Lambda(n)\Lambda(n+h) =\mathfrak{G}(h)x +O \left (\frac{x}{(\log x)^{C}}\right ) 
\end{equation}
for all but $O \left (H/\log x^{C} \right )$ values $h$ with $|h-h_0| \leq H$. 
\end{thm}

 \section{A Three Levels Autocorrelation} \label{s773}
A level on the definition of an autocorrelation function denotes a range of values. The graph of the function $R(t)=\sum_{n \leq x} \Lambda(n)\Lambda(n+t)$ is an unbounded and nonmonotonic 3-level step function as $t \in \Z$ varies over the integers. 
\begin{enumerate}
\item It has a global maximum at the origin $t=0$.
\item It has infinitely many local maxima at the even integers $t=2k$.
\item It has infinitely many local minima at the odd integers $t=2k+1$.
\end{enumerate}
A quantitative version is given below.
\begin{thm} \label{thm773.8} For a large number $x \geq 1$, the autocorrelation of the vonMangoldt function is a $3$-level autocorrelation
\begin{equation} \label{eq773.200}
\sum_{n \leq x} \Lambda(n)\Lambda(n+t)=
\begin{cases} \displaystyle x \log x+O \left ( (x \log x)e^{-c\sqrt{\log x}} \right ), & \mbox{if } t=0,\\
\displaystyle a_kx+O \left ( xe^{-c\sqrt{\log x}} \right ), &\mbox{if } t \equiv 0 \bmod 2,\\
\displaystyle O \left ( \log^3 x \right ), &\mbox{if }t \equiv 1 \bmod 2.
\end{cases}
\end{equation}
\end{thm}

\begin{proof} (i) The proof of the first level for $t=0$ follows from the prime number theorem, see Theorem \ref{thm8200.200}. (ii) The proof of the second level for even integers $t=2k$ is an open problem, see Theorem \ref{thm9990A.950}, and Theorem \ref{thm9990.950B}. (iii) For odd $n \geq 3$ and $t=2k+1$, there is $\Lambda(n+t)=0$. Thus, let $n =2^m 2$ and $t=2k+1$. Then, there is the estimate
\begin{eqnarray}
\sum_{n \leq x}\Lambda(n)\Lambda(n+t)&=&\sum_{2^n \leq x}\Lambda(2^n)\Lambda(2^n+2k+1) \\
&\leq & 2\sum_{2^n \leq x}\Lambda(2^n)^2 \nonumber  \\
&\leq & 2(\log 2)^2 \sum_{2^n \leq x}n^2\\
&= & O \left ( \log^3 x \right )\nonumber. 
\end{eqnarray}
Part (iii) also provides an upper bound for the number of pairs of prime powers of the form $p =2^n$ and $p+t=2^n+2k+1$ as $n \to \infty$ with $k\geq 1$ fixed.  
\end{proof}

\chapter{ Log-Correlation Function and Its Applications}\label{MF2332}
This chapter introduces the log-correlation function and its applications to study the correlation of arithmetic functions. Specific applications show that the autocorrelations of vonMangoldt functions do not vanish on a short interval. A few results are proved here, the material is self-contained.

	\section{Introduction}\label{S9990A}	
	Let $f,g:  \N\longrightarrow \C$ be a pair of arithmetic functions. The correlation function
	\begin{equation}\label{eq9990A.900}
		R(t)=\sum_{n\leq x}f(n)g(n+t),
	\end{equation}
	where $t\in \R$ is a real number, is a topic of research in all the mathematical sciences. This note introduces the closely related log-correlation function 
	\begin{equation}\label{eq9990A.910}
		V(t)=\sum_{n\leq x}\log \left (u+f(n)\right )\left (v+g(n+t)\right ),
	\end{equation}
	where $u,v\in \R^{\times}$, of the arithmetic functions $f$ and $g$. The form of equation \eqref{eq9990A.910} encapsulates a smooth link between the additive and multiplicative structures of the integers. Moreover, the elementary properties of the equation \eqref{eq9990A.910}, or other variations of it, makes it an effective tool to study the correlation functions of certain classes of unbounded functions $f\geq 0$. The supports $\mathcal{D}=\{n\geq 1:f(n)\ne0\}$ of these functions are expected to be infinite subsets of integers $\mathcal{D}\subset \N$ of zero or nonzero densities, $\delta(\mathcal{D})\in [0,1)$. \\
	
	An application to the autocorrelation of the vonMangoldt function is demonstrated here. This autocorrelation function is essential tool in the investigations of additive number problems. Specifically, it is a weighted count of the number of consecutive prime pairs $p$ and $p+2k$ up to $x$. 
	The qualitative version of the relevant conjecture is attributed to dePolignac, \cite{DA1849}, and the quantitative version of this conjecture is attributed to Hardy and Littlewood, \cite[Conjecture B]{HL1923}. The precise statement, is the followings. 
	
	\begin{conj} \label{conj9990A.200}{\normalfont (\cite{DA1849}, \cite{HL1923})} Let $x \geq 1$ be a large number. Let $\Lambda$ be the vonMangoldt function. Then
		\begin{equation} \label{eq9990A.200}
			\sum_{n\leq x} \Lambda \left( n \right )\Lambda \left( n+2k \right )
			=\mathfrak{G}(2k)x+O\left ( x^{1/2+\varepsilon }\right )\nonumber,
		\end{equation}
		where $k\geq 1$ is a fixed integer, and the density constant is given by the singular series 
		\begin{equation} \label{eq9990A.220}
			\mathfrak{G}(2k)=\prod_{2\ne p \mid k}\left( \frac{p-1}{p-2 }\right )\prod_{p \geq 3}\left( 1-\frac{1}{(p-1)^2} \right ) .
		\end{equation}
	\end{conj}
	
	The first case for $k=1$ is known as the twin prime conjecture. There is a vast literature on this topic. Some of the recent works are given \cite{GP2007}, \cite{MT2017}, et alii. A result for a constrained averaged correlation function
	\begin{equation} \label{eq9990A.230}
		\sum_{n\leq x} \Lambda \left( n \right )\Lambda \left( n+2k \right )
		=\mathfrak{G}(2k)x+O\left (\frac{ x}{(\log x)^c }\right ),
	\end{equation}
	for almost all $2k$, but a subset of integers of cardinality $O\left ( x(\log x)^{-c }\right )$, where $c>0$ is a constant, is given in \cite[Theorem 1.3]{MT2017}. A partial but effective result is proved here.	
	
	\begin{thm} \label{thm9990A.950} If $x\geq1$ is a large number, and $2k\geq2$ is a small fixed integer, then 
		\begin{equation}\label{eq9990A.950}
			\frac{x}{(\log x)^{4}}\ll	\sum_{ x\leq n\leq 2x}	\frac{\Lambda(n)\Lambda(n+2k)}{(\log n)^{4}}. 
		\end{equation} 
		In particular, \begin{equation}\label{eq9990A.960}
			\sum_{ x\leq n\leq 2x}	\Lambda(n)\Lambda(n+2k)\geq (\log x)^{2}
		\end{equation}
		is unbounded as $x\to\infty$.
	\end{thm}
	Section \ref{S9900A} and Section \ref{S9920A} cover elementary background materials, and the proof of Theorem \ref{thm9990A.950} appears in Section \ref{S9988A}.

	\section{Prime Numbers Results}\label{S9900A}
	The prime numbers have a few indicator functions, two of them are displayed in the next two Lemmas. However, every result in number theory is derived or proved via the simpler weighted indicator function specified in Definition \ref{dfn9900A.020}.
	\begin{lem}\label{lem9900A.100} Let $\omega(n)=\#\{p\mid n\}$ be the prime divisors counting function, and let $\mu(n)\in \{-1,0,1\}$ be the Mobius function. If $n\geq 1$ is an integer, then the prime numbers indicator function is given by 
		\begin{equation}\label{eq9900A.100}
			\varkappa(n)=\sum_{d\mid n}\mu(d)\omega(n/d)=
			\begin{cases}
				1& \text{ if } n \text{ is prime,}\\
				0& \text{ if } n \text{ is not prime.}\nonumber 
			\end{cases}
		\end{equation}
	\end{lem}
	\begin{proof}[\textbf{Proof}] Consider the generating series
		\begin{equation}\label{eq9900A.110}
			\sum_{n\geq 1}\frac{\omega(n)}{n^s}=\sum_{n\geq 1}\frac{1}{n^s}\sum_{n\geq 1}\frac{\varkappa(n)}{n^s},
		\end{equation}
		and solve for $\varkappa(n)$.
	\end{proof}
	\begin{lem}\label{lem9900A.120} Let $\Omega(n)$ be the prime divisors, including multiplicities, counting function, and let $\mu(n)\in \{-1,0,1\}$ be the Mobius function. If $n\geq 1$ is an integer, then the prime powers indicator function is given by 
		\begin{equation}\label{eq9900A.120}
			\varkappa_0(n)=\sum_{d\mid n}\mu(d)\Omega(n/d)=
			\begin{cases}
				1& \text{ if } n \text{ is a prime power,}\\
				0& \text{ if } n \text{ is not prime power.}\nonumber 
			\end{cases}
		\end{equation}
	\end{lem}
	
	\begin{dfn}\label{dfn9900A.020}{\normalfont The vonMangoldt function is defined by the weighted prime powers indicator function
			$$\Lambda(n)=
			\begin{cases}
				\log p & \text{if } n=p^m,\\ 
				0 & \text{if } n\ne p^m.\\ 
			\end{cases}
			$$
		}
	\end{dfn}
	The symbol $p^m\geq 2$, with $ m \in \N$, denotes a prime power.\\
	
	\begin{thm}\label{thm9900A.050} If $x\geq 1$ is a large number, then,
		\begin{enumerate} [font=\normalfont, label=(\roman*)]
			\item $\displaystyle \sum_{n\leq x} \Lambda(n)=x\left(1+O\left( \frac{1}{\log x}\right )\right) , $ \tabto{8cm}  unconditionally.
			\item $\displaystyle \sum_{n\leq x} \varkappa(n)=\frac{x}{\log x}\left(1+O\left( \frac{1}{\log x}\right )\right) ,$\tabto{8cm}  unconditionally.
		\end{enumerate}
	\end{thm}
	
	\begin{proof} These are standard results in number theory, stated in simpler notations, sharper results are proved in \cite{EE1985}, \cite[p.\ 179]{MV2007}, et alii.  
	\end{proof}
	These primes counting functions are usually denoted by $\psi(x)=\sum_{n\leq x} \Lambda(n)$ and $\pi(x)=\sum_{n\leq x} \varkappa(n)$, respectively. Moreover, there is a link by the basic identities
	\begin{equation}\label{lem9900A.140B}
		\psi(x)=\pi(x)\log x -\int_2^x\frac{\pi(t)}{t}dt,
	\end{equation}
	and 
	\begin{equation}\label{lem9900A.145B}
		\pi(x)=\frac{\psi(x)}{\log x} +\int_2^x\frac{\psi(t)}{t(\log t)^2}dt,
	\end{equation}
	for all large numbers $x\geq1$.

	\section{Preliminary Results Method I}\label{S9910A}
	\begin{lem}\label{lem9910A.130} If $x\geq 1$ is a large number, then,
		\begin{equation}\label{eq9910A.130}
			\sum_{p\leq x} \log \log p =\pi(x)\log\log x+O\left( \frac{x}{(\log x)^2}\right ).\nonumber  
		\end{equation}
	\end{lem}
	\begin{proof}[\textbf{Proof}] Let $\pi(x)=\sum_{p\leq x}1=x/\log x+O(x(\log x)^{-2})$, see Theorem \ref{thm9900A.050}. Use an integral technique to evaluate this finite sum over the primes.
		\begin{eqnarray}\label{eq9910A.140}
			\sum_{p\leq x} \log \log p
			&=&\int_2^x\log \log (t) d\pi(t)\\
			&=&\pi(x)\log\log x+c_0-\int_2^x\frac{\pi(t)}{t\log t}dt\nonumber\\
			&=&\pi(x)\log\log x+O\left( \frac{x}{(\log x)^2}\right ),\nonumber
		\end{eqnarray}
		where the $c_0$ constant is absorbed into the error term.
	\end{proof}
	
	\begin{lem}\label{lem9910A.150} If $x\geq 1$ is a large number, then,
		\begin{equation}\label{eq9910A.150}
			\sum_{x\leq p\leq 2x} \log \log p =\frac{x\log \log x}{\log x}\left(1+O\left( \frac{1}{\log x}\right )\right )\nonumber . 
		\end{equation}
	\end{lem}
	\begin{proof}[\textbf{Proof}] Employ the previous result to derive the asymptotic formula.
		\begin{eqnarray}\label{eq9910A.160}
			\sum_{x\leq p\leq 2x} \log \log p 
			&=&\left(\pi(2x)\log\log 2x+O\left( \frac{x}{(\log x)^2}\right )\right)\\
			&&\hskip 1.5 in -\left(\pi(x)\log\log x+O\left( \frac{x}{(\log x)^2}\right )\right)\nonumber\\
			&=&\frac{2x\log \log x}{\log x}\left(1+O\left( \frac{1}{\log x}\right )\right )\nonumber\\
			&&\hskip 1.85 in -\frac{x\log \log x}{\log x}\left(1+O\left( \frac{1}{\log x}\right )\right )\nonumber\\
			&=&\frac{x\log \log x}{\log x}\left(1+O\left( \frac{1}{\log x}\right )\right )\nonumber,
		\end{eqnarray}
		since
		\begin{eqnarray}\label{eq9910A.170}
			\pi(2x)\log\log 2x
			&=&\left(\frac{2x}{\log x}+O\left( \frac{x}{(\log x)^2}\right )\right)\left(\log \log x+\log \left(1+\dfrac{\log 2}{\log x}\right)\right)\nonumber\\
			&=&\frac{2x\log \log x}{\log x}\left(1+O\left( \frac{1}{\log x}\right )\right ) .
		\end{eqnarray}
	\end{proof}

	\section{The Main Result: Method I}\label{S9988A}
	The basic idea is to estimate the log-correlation function \eqref{eq9988A.825} in two different ways: The upper bound is computed inside the logarithm function, for example,
	\begin{equation}\label{eqS9988A.832}
		\log (1+3)(1+5)=	\log (1+3+5+15)\leq \log(1+2\cdot 5+15).		
	\end{equation} In contrast, the lower bound is computed by expanding the logarithm function, for example, 
	\begin{equation}\label{eqS9988A.833}
		\log (1+3)(1+5)=\log (1+3)+\log(1+5) \geq \log(3)+\log (5).	
	\end{equation}
	Therefore,
	\begin{equation}\label{eqS9988A.834}
		\log(3)+\log (5)\leq 	\log (1+3)(1+5)\leq \log(1+2\cdot 5+15).	
	\end{equation}
	This simple idea in conjunction with the assumption that the correlation function vanishes on a short interval 
	\begin{equation}\label{eq9988A.800}
		\sum_{ x\leq n\leq 2x}	\Lambda(n)\Lambda(n+2k)=0 
	\end{equation} 
	for all large real numbers $x\geq x_0$, lead to a contradiction.
	\begin{proof}[\textbf{Proof:} {\normalfont \textbf{(Theorem \ref{thm9990A.950})}}]
		Assume that there are no consecutive prime pairs $p$ and $p+2k$ such that $p\geq x\geq x_0$ for large number $x\geq 1$. Equivalently,
		\begin{equation}\label{eq9988A.810}
			\Lambda(n)\Lambda(n+2k)=0 
		\end{equation}
		for all integers $n\geq x$. To develop a reductio ad absurdum, the upper bound and the lower bound of the log-correlation function
		\begin{equation}\label{eq9988A.825}
			V(x)=\sum_{x\leq n\leq 2x} \log \left (1+	\Lambda(n)\right )\left (1+	\Lambda(n+2k)\right ) 
		\end{equation}
		are estimated in two different ways.\\
		
		\textbf{Evaluation of an upper bound.} Use the algebraic and analytic properties of the logarithm to compute an upper bound for
		\begin{eqnarray}\label{eq9988A.830}
			V(x)&=&\sum_{x\leq n\leq 2x} \log \left (1+	\Lambda(n)\right )\left (1+	\Lambda(n+2k)\right )\\
			&=&\sum_{x\leq n\leq 2x} \log \left (1+	\Lambda(n)+	\Lambda(n+2k)+\Lambda(n)\Lambda(n+2k)\right )\nonumber.
		\end{eqnarray}
		By the hypothesis \eqref{eq9988A.810}, the product  $\Lambda(n)\Lambda(n+2k)=0$, for each $n\in [x,2x]$, does not contribute to the finite sum \eqref{eq9988A.830}, and $\Lambda(n)$ and $\Lambda(n+2k)$ are increasing functions of $n\geq1$. Thus, for small fixed integer $2k\geq 2$ such that $n$ or $n+2k$ is prime, the argument of the log-correlation function has the upper bound
		\begin{eqnarray}\label{eq2288A.340}
			1+\Lambda(n)+\Lambda(n+2k)+\Lambda(n)\Lambda(n+2k)&=&1+\Lambda(n)+\Lambda(n+2k)\\
			&\leq&1+  \log p +\log (p+2k)\nonumber\\
			&\leq&  1+\log p +2\log p\nonumber\\
			&\leq&  4\log p,\nonumber
		\end{eqnarray}
		where $n,p\in[x,2x]$ are integer and prime variables, respectively. Replacing this data leads to the upper bound
		\begin{eqnarray}\label{eq9988A.845}
			V(x)&=&\sum_{x\leq n\leq 2x} 
			\log  \left (1+	\Lambda(n)+	\Lambda(n+2k)+\Lambda(n)\Lambda(n+2k)\right )\\
			&=&
			\sum_{\substack{x\leq n\leq 2x\\ \text{squarefree }n, n+2k}} \log  \left (1+	\Lambda(n)+	\Lambda(n+2k)\right ) \nonumber \\
			&& \hskip 1.5 in +\sum_{\substack{x\leq n\leq 2x\\ \text{nonsquarefree }n, n+2k}} \log  \left (1+	\Lambda(n)+	\Lambda(n+2k)\right )\nonumber\\		
			&\leq&
			\sum_{x\leq p\leq 2x} \log\left (4\log p\right ) +\sum_{2\leq m\leq \log 2x,}\sum_{x\leq p^m\leq 2x} 
			\log \left (4\log p^m\right )\nonumber\\
			&\leq&
			\sum_{x\leq p\leq 2x} \log\left( 4\log p\right) + O\left((\log x)^2x^{1/2}\right)\nonumber,
		\end{eqnarray}
		where the double sum reaps the prime powers $p^m\geq p^2$. Using the prime number theorem, (Theorem \ref{thm9900A.050}), and simplifying it, yield
		\begin{eqnarray}\label{eq9988A.850}
			V(x)		&\leq&
			\sum_{x\leq p\leq 2x} \log\left( 4\log p\right) + O\left((\log x)^2x^{1/2}\right)\\
			&\leq&
			(\log 4)\sum_{x\leq p\leq 2x} 1+\sum_{x\leq p\leq 2x} \log\log p+ O\left((\log x)^2x^{1/2}\right)\nonumber\\
			&\leq&
			(2\log 4)\frac{x}{\log x}+\sum_{x\leq p\leq 2x} \log \log p\nonumber.
		\end{eqnarray}
		Substituting Lemma \ref{lem9910A.150} into \eqref{eq9988A.850}, the log-correlation function has the upper bound 
		\begin{eqnarray}\label{eq9988A.860}
			V(x)
			&\leq& (2\log 4)\frac{x}{\log x}+\sum_{x\leq p\leq 2x} 
			\log \log p \\
			&\leq&(2\log 4)\frac{x}{\log x}+\frac{x\log \log x}{\log x}\left(1+O\left( \frac{1}{\log x}\right )\right )	\nonumber	.
		\end{eqnarray} 
		
		\textbf{Evaluation of a lower bound.} Observe that both $\Lambda(n)$ and $\Lambda(n+2k)$ contribute the same amount, up to a term, to the log-correlation function \eqref{eq9988A.870}. Now, use the algebraic and analytic properties of the logarithm to compute a lower bound for
		\begin{eqnarray}\label{eq9988A.870}
			V(x)&=&	\sum_{x\leq n\leq 2x} \log \left (1+	\Lambda(n)\right )\left (1+	\Lambda(n+2k)\right )\\
			&=&\sum_{x\leq n\leq 2x} \log \left (1+	\Lambda(n)\right )+\sum_{x\leq n\leq 2x} \log \left (1+	\Lambda(n+2k)\right )\nonumber.
		\end{eqnarray}
		The finite sum
		\begin{eqnarray}\label{eq9988A.880}
			\sum_{x\leq n\leq 2x} \log \left (1+	\Lambda(n))\right )&\geq&\sum_{x\leq p\leq 2x} \log \left (1+	\log p\right )\\
			&\geq&\sum_{x\leq p\leq 2x} 
			\log \log p\nonumber,
		\end{eqnarray}
		and the finite sum
		\begin{eqnarray}\label{eq9988A.885}
			\sum_{x\leq n\leq 2x} \log \left (1+	\Lambda(n+2k))\right )&\geq&\sum_{x\leq p\leq 2x} \log \left (1+	\log (p+2k)\right )\\
			&\geq&\sum_{x\leq p\leq 2x} 
			\log \log p\nonumber	.
		\end{eqnarray}
		Thus, summing \eqref{eq9988A.880} and \eqref{eq9988A.885}, then  
		substituting Lemma \ref{lem9910A.150}, yield the lower bound 
		\begin{eqnarray}\label{eq9988A.895}
			V(x)
			&\geq& 2\sum_{x\leq p\leq 2x} 
			\log \log p \\
			&\geq&\frac{2x\log \log x}{\log x}\left(1+O\left( \frac{1}{\log x}\right )\right )	\nonumber	.
		\end{eqnarray} 
		Comparing the upper bound in \eqref{eq9988A.860} and the lower bound in \eqref{eq9988A.895} return
		\begin{eqnarray}\label{eq9988A.900}
			\frac{2x\log \log x}{\log x}\left(1+O\left( \frac{1}{\log x} \right )\right )
			&\leq &	V(x)\\
			&\leq&(2\log 4)\frac{x}{\log dx}+\frac{x\log \log x}{\log x} \left(1+O\left( \frac{1}{\log x}\right )\right )	\nonumber	.
		\end{eqnarray} 
		Clearly, this is false. Therefore, each interval $[x,2x]$ contains at least one prime pair $p$ and $p+2k$, and the autocorrelation function
		\begin{equation}\label{eq9988A.915}
			\sum_{ x\leq n\leq 2x}	\Lambda(n)\Lambda(n+2k)\geq (\log x)^2 
		\end{equation} 
		is unbounded as $x\to\infty$. 
		Quod erat demonstrandum.
	\end{proof}

	\section{Preliminary Results For Method II}\label{S9920A}
	
	\begin{lem}\label{lem9920A.150} If $x\geq 1$ is a large number, and $c\geq1$ is a constant, then,
		\begin{equation}\label{eq9920A.150}
			\sum_{x\leq n\leq 2x} 	 \frac{\Lambda^3(n)}{(\log n)^{6c}}\gg 	 \frac{x}{(\log x)^{6c-2}}\nonumber . 
		\end{equation}
	\end{lem}
	\begin{proof}[\textbf{Proof}] Let $\pi(t)=\#\{p\leq x\}\gg t/\log t$. Use Definition \ref{dfn9900A.020}, and Theorem \ref{thm9900A.050} to estimate the lower bound.
		\begin{eqnarray}\label{eq9920A.160}
			\sum_{x\leq n\leq 2x} 	 \frac{\Lambda^3(n)}{(\log n)^{6c}}&\gg&\sum_{x\leq p\leq 2x} \frac{(\log p)^3}{(\log p)^{6c}} \\
			&\gg&\sum_{x\leq p\leq 2x} \frac{1}{(\log p)^{6c-3}}\nonumber\\
			&\gg&\int_x^{2x}\frac{1}{(\log t)^{6c-3}}d\pi(t)\nonumber\\
			&\gg&\frac{x}{(\log x)^{6c-2}}\nonumber.
		\end{eqnarray}
	\end{proof}
	
	\begin{lem}\label{lem9920A.250} If $x\geq 1$ is a large number, and $b>1$ is a constant, then,
		\begin{equation}\label{eq9920.250}
			\sum_{x\leq n\leq 2x} 	 \frac{\Lambda(n)\Lambda(n+t)}{(\log n)^{b}}\ll 	 \frac{x}{(\log x)^{b-1}}\nonumber . 
		\end{equation}
	\end{lem}
	\begin{proof}[\textbf{Proof}] Let $\pi(t)=\#\{p\leq x\}\ll t/\log t$. Use the Definition \ref{dfn9900A.020}, and Theorem \ref{thm9900A.050} to estimate the lower bound.
		\begin{eqnarray}\label{eq9920A.260}
			\sum_{x\leq n\leq 2x} 	 \frac{\Lambda(n)\Lambda(n+t)}{(\log n)^{b}}&\ll&(\log x)^2\sum_{x\leq p\leq 2x} \frac{1}{(\log p)^{b}} \\
			&\ll&(\log x)^2\sum_{x\leq p\leq 2x} \frac{1}{(\log p)^{b}}\nonumber\\
			&\ll&(\log x)^2\int_x^{2x}\frac{1}{(\log t)^{b}}d\pi(t)\nonumber\\
			&\ll&\frac{x}{(\log x)^{b-1}}\nonumber.
		\end{eqnarray}
	\end{proof}
	
	\begin{lem}\label{lem9920A.900} Let $x\geq 1$ be a large number, and let $t=2k\geq2$ be a small fixed integer. If $\Lambda(n)\Lambda(n+2k)\ne0$, then,
		\begin{equation}\label{eq9920A.900}
			\sum_{x\leq n\leq 2x} \log \left (1+	\frac{\Lambda(n)}{(\log n)^{2c}}\right )\log\left (1+	\frac{\Lambda(n+t)}{(\log n)^{2c}}\right ) \gg 	 \frac{x}{(\log x)^{6c-2}}	\nonumber,
		\end{equation}
		where $c\geq 1$ is an arbitrary constant.
	\end{lem}
	\begin{proof}[\textbf{Proof}] Let $t\geq0$, and $c\geq1$ be small fixed parameters. The real number inequality $0\leq z^2\leq z$ for $0\leq z<1$, and the fact 
		\begin{eqnarray}\label{eq9920A.910}
			0&\leq &\frac{\Lambda(n+t)}{(\log x)^{2c}}-\frac{1}{2}\left(\frac{\Lambda(n+t)}{(\log x)^{2c}}\right)^2\\
			&\leq&	\log \left (1+	\frac{\Lambda(n+t)}{(\log x)^{2c}}\right )\nonumber\\
			&<&1\nonumber,
		\end{eqnarray}	
		for large $n\geq x$, imply that
		\begin{equation}\label{eq9920A.920}
			\left (	\frac{\Lambda(n)}{(\log x)^{2c}}-\frac{1}{2}\left(\frac{\Lambda(n)}{(\log x)^{2c}}\right)^2\right )^2\leq	\frac{\Lambda(n+t)}{(\log x)^{2c}}-\frac{1}{2}\left(\frac{\Lambda(n+t)}{(\log x)^{2c}}\right)^2.
		\end{equation}
		Now, to derive the asymptotic lower bound, substitute \eqref{eq9920A.920} into \eqref{eq9920A.930}, to obtain 
		
		\begin{eqnarray}\label{eq9920A.930}
			V(x)&=&\sum_{x\leq n\leq 2x} \log \left (1+	\frac{\Lambda(n)}{(\log n)^{2c}}\right )\log\left (1+	\frac{\Lambda(n+t)}{(\log n)^{2c}}\right )\\
			&\geq&\sum_{x\leq n\leq 2x} \left (	\frac{\Lambda(n)}{(\log n)^{2c}}-\frac{1}{2}\left(\frac{\Lambda(n)}{(\log n)^{2c}}\right)^2\right ) \left (	\frac{\Lambda(n+t)}{(\log n)^{2c}}-\frac{1}{2}\left(\frac{\Lambda(n+t)}{(\log n)^{2c}}\right)^2\right )\nonumber\\
			&\geq&\sum_{x\leq n\leq 2x} \left (	 \frac{\Lambda(n)}{(\log n)^{2c}}-\frac{1}{2}\left(\frac{\Lambda(n)}{(\log n)^{2c}}\right)^2\right )^3\nonumber,
		\end{eqnarray}
		since $t=2k>0$. Expanding the argument, and applying Lemma \ref{lem9920A.150}, return
		\begin{eqnarray}\label{eq9920A.940}
			V(x)
			&\geq&\sum_{x\leq n\leq 2x} \left (	 \frac{\Lambda(n)}{(\log n)^{2c}}-\frac{1}{2}\left(\frac{\Lambda(n)}{(\log n)^{2c}}\right)^2\right )^3\\
			&\gg&\sum_{x\leq n\leq 2x} 	 \frac{\Lambda^3(n)}{(\log n)^{6c}}	\nonumber\\
			&\gg&\frac{x}{(\log x)^{6c-2}}	\nonumber,
		\end{eqnarray}
		where $c\geq 1$ is a constant.
	\end{proof}
	
	\section{The Main Result: Method II}\label{S9988}
	For the admissible pair $(a_0,a_1)=(0,2k)$, the hypothesis take the form
	\begin{equation}\label{eq9988.100}
		\sum_{ x\leq n\leq 2x}	\Lambda(n)\Lambda(n+2k)=O\left( x^{1-\varepsilon}\right) , 
	\end{equation} 
	where $\varepsilon\in (0,1)$ is a small number, for all large numbers $x\geq x_0$. \\
	
	The number of prime pairs $p_n$, and $p_n+2k$ over the short interval $[x,2x]$ predicted by Conjecture \ref{conj9990A.200} is $\pi_2(x)\asymp x(\log x)^{-2}$, and the predicted average gap between prime pairs is
	\begin{equation}\label{eq9988.050}
		\frac{x}{\pi_2(x)}\asymp(\log x)^2.
	\end{equation} 
	While, the number of prime pairs $p_n$, and $p_n+2k$ over the short interval $[x,2x]$ predicted by hypothesis \eqref{eq9988.100} is $\pi_2(x)\asymp x^{1-\varepsilon}(\log x)^{-2}$, and the predicted average gap between prime pairs is
	\begin{equation}\label{eq9988.060}
		\frac{x}{\pi_2(x)}\asymp x^{\varepsilon}(\log x)^{2}.
	\end{equation}

	This much weaker hypothesis on the number of prime pairs, and the average gap between prime pairs over the short interval $[x,2x]$, will be used to derive a contradiction.
\begin{thm} \label{thm9990.950B} If $x\geq1$ is a large number, and $2k\geq2$ is a small fixed integer, then 
	\begin{equation}\label{eq9990.950}
		\frac{x}{(\log x)^{4}}\ll	\sum_{ x\leq n\leq 2x}	\frac{\Lambda(n)\Lambda(n+2k)}{(\log n)^{4}}. 
	\end{equation} 
	In particular, \begin{equation}\label{eqS9990.960}
		\sum_{ x\leq n\leq 2x}	\Lambda(n)\Lambda(n+2k)\gg x
	\end{equation}
	as $x\to\infty$.
\end{thm}

	\begin{proof}[\textbf{Proof:}]
		To demonstrates that the hypothesis \eqref{eq9988.100} is false for all large numbers $x\geq x_0$, consider the log-correlation function
		\begin{equation}\label{eq9988.825}
			V(x)=\sum_{x\leq n\leq 2x} \log \left (1+	\frac{\Lambda(n)}{(\log n)^{2c}}\right )\log\left (1+	\frac{\Lambda(n+2k)}{(\log n)^{2c}}\right ), 
		\end{equation}
		where $c\geq1$ is an arbitrary constant.\\
		
		The result in Lemma \ref{lem9920A.900}, and the real number inequality
		\begin{equation}\label{eq9988.830}
			0<\log \left (1+u\right )\log\left (1+v\right )\leq uv	
		\end{equation}
		for $0<u,v<1$, yield
		\begin{eqnarray}\label{eq9988.840}
			\frac{x}{(\log x)^{6c-2}}&\ll&\sum_{x\leq n\leq 2x} \log \left (1+	\frac{\Lambda(n)}{(\log n)^{2c}}\right )\log\left (1+	\frac{\Lambda(n+2k)}{(\log n)^{2c}}\right )\\
			&\ll&\sum_{x\leq n\leq 2x} \frac{\Lambda(n)\Lambda(n+2k)}{(\log n)^{4c}}\nonumber.
		\end{eqnarray}
		Setting $c=1$, and a small fixed integer $t=2k$, these data imply that 
		\begin{equation}\label{eq9988.845}
			\frac{x}{(\log x)^{4}}\ll\sum_{x\leq n\leq 2x} \frac{\Lambda(n)\Lambda(n+2k)}{(\log n)^{4}}.
		\end{equation}
		Furthermore, Lemma \ref{lem9920A.250} implies that
		\begin{equation}\label{eq9988.850}
			\frac{x}{(\log x)^{4}}\ll\sum_{x\leq n\leq 2x} \frac{\Lambda(n)\Lambda(n+2k)}{(\log n)^{4}}\ll\frac{x}{(\log x)^{4}}.
		\end{equation}
		The last expression and partial summation imply that
		\begin{equation}\label{eq9988.855}
			\sum_{ x\leq n\leq 2x}	\Lambda(n)\Lambda(n+2k)\gg x, 
		\end{equation}
		as $x\to\infty$. Therefore, this contradicts the hypothesis \eqref{eq9988.100} for any $\varepsilon\in (0,1)$. In particular, it proves the expected asymptotic formula in Conjecture \ref{conj9990A.200}. Quod erat demonstrandum.
	\end{proof}

	\section{Generalization To Linear Prime Pairs}\label{S0980}
	The concept covered in the Section \ref{S9988} has a seamless generalization to \textit{admissible} pairs of linear polynomials. For example, a pair of admissible polynomials specifies a prime pairs $p=an+b$ and $q=cn+d$. The derivations of the conjecture, and the proofs for various prime pairs are the identical.

	\section{Germain Primes}\label{S0988}
	The specific case of Germain primes $p$ and $2p+1$ is explicated below. \\
	
	The conjecture number of Germain primes has the asymptotic formula
	\begin{equation}\label{eq0988.700}
		\sum_{ n\leq x}\Lambda(n)\Lambda(2n+1)
		=\mathfrak{G}(f)x+o(x),
	\end{equation}
	where $f(t)=f_1(t)f_2(t)=t\cdot (2t+1)$, $deg f=\deg f_1\cdot\deg f_2$=1, 
	\begin{equation}\label{eq0988.710}
		\mathfrak{G}(f)=\frac{1}{\deg(f)}\prod_{p\geq 2} \left(1-\frac{1}{p}\right)^{-2}\left(1-\frac{\rho(p)}{p}\right)=2\prod_{p\geq 3} \left(1-\frac{1}{(p-1)^2}\right),
	\end{equation}
	since 
	\begin{equation}\label{eq0988.720}
		\rho(p)=\#\{n:f(n)\equiv 0 \bmod p\}=
		\begin{cases}
			1&\text{ if } p=2;\\
			2&\text{ if } p\geq 3,
		\end{cases}
	\end{equation}
	confer Conjecture \ref{conj1188.800} for the general details. A partial but effective result is proved here.	\\

For Germain primes, the hypothesis take the form
\begin{equation}\label{eq0988.800}
	\sum_{ x\leq n\leq 2x}	\Lambda(n)\Lambda(2n+1)=O\left( x^{1-\varepsilon}\right) , 
\end{equation} 
where $\varepsilon\in (0,1)$ is a small number, for all large numbers $x\geq x_0$.\\

The number of Germain primes $p$ and $2p+1$ over the short interval $[x,2x]$ by this conjecture is $\pi_2(x)\asymp x(\log x)^{-2}$, and the predicted average gap between prime pairs is
\begin{equation}\label{eq0988.050}
	\frac{x}{\pi_2(x)}\asymp(\log x)^2.
\end{equation} 
While, the number of Germain primes $p$ and $2p+1$ over the short interval $[x,2x]$ predicted by hypothesis \eqref{eq0988.800} is $\pi_2(x)\asymp x^{1-\varepsilon}(\log x)^{-2}$, and the predicted average gap between prime pairs is
\begin{equation}\label{eq0988.060}
	\frac{x}{\pi_2(x)}\asymp x^{\varepsilon}(\log x)^{2}.
\end{equation} 

This much weaker hypothesis on the number of Germain primes, and the average gap between prime pairs over the short interval $[x,2x]$, will be used to derive a contradiction.

\begin{thm} \label{thm0988.900} There exists infinitely many Germain primes. Specifically, if $x\geq1$ is a large number, then 
	\begin{equation}\label{eqS0988.910}
		\sum_{x\leq n\leq 2x}\Lambda(n)\Lambda(2n+1) \gg x,
	\end{equation}
	as $x\to\infty$. 
\end{thm}
\begin{proof}[\textbf{Proof:}] To demonstrates that the hypothesis \eqref{eq0988.800} is false, consider the log-correlation function
	\begin{equation}\label{eq9988.920}
		V(x)=\sum_{x\leq n\leq 2x} \log \left (1+	\frac{\Lambda(n)}{(\log n)^{2c}}\right )\log\left (1+	\frac{\Lambda(2n+1)}{(\log n)^{2c}}\right ), 
	\end{equation}
	where $c\geq1$ is an arbitrary constant.\\
	
	Except for the required modifications, the analysis is the same as in \eqref{eq9988.825} to 
	\eqref{eq9988.855}. Therefore,
	\begin{equation}\label{eq0988.930}
		\frac{x}{(\log x)^{4}}\ll	\sum_{ x\leq n\leq 2x}	\frac{\Lambda(n)\Lambda(2n+1)}{(\log n)^{4}}. 
	\end{equation} 
	In particular, 
	\begin{equation}\label{eq0988.935}
		\sum_{ x\leq n\leq 2x}	\Lambda(n)\Lambda(2n+1)\gg x
	\end{equation}
	as $x\to\infty$. But, this contradicts the hypothesis \eqref{eq0988.800} for any $\varepsilon\in (0,1)$. This proves that \eqref{eq0988.935} is the correct asymptotic order.
\end{proof}

\section{Algebraic Primes Counting Problems}\label{S1188}
All the algebraic primes counting problems are handled under the umbrella of the Bateman-Horn conjecture. The algebraic sequences of primes are defined by algebraic equations. For examples, the sequence of twin primes, the sequence of Germain primes, et cetera. In contrast, the nonalgebraic primes counting problems do not have a general conjectural framework, but there are a few heuristic tools such as the Cramer probabilistic model, and the Gaussian counting model. The nonalgebraic sequences of primes are defined by exponential equations, transcendental equations, and other schemes. For examples, the sequence of Fermat primes, the sequence of Beatty primes, et cetera. 

\begin{dfn}\label{dfn1188.800}{\normalfont
		The \textit{fixed divisor} $\text{div}(f)=\gcd(f(\mathbb{Z}))$ of a polynomial $f(x) \in \mathbb{Z}[x]$ over the integers is the greatest common divisor of its image $f(\mathbb{Z})=\{f(n):n \in \mathbb{Z}\}$.
	}
\end{dfn}

The fixed divisor $\tdiv(f)=1$ if $f(x) \equiv 0 \mod p$ has $\nu_f(p)<p$ solutions for every prime $p<\deg(f)$, see \cite[p. 395]{FI2010}. An irreducible polynomial can represent infinitely many primes if and only if it has fixed divisor $\text{div}(f)=1$.\\

\begin{exa}\label{exa1188.800}{\normalfont
		The polynomial $f_1(x)=x^2+1$ is irreducible over the integers and has the fixed divisor $\tdiv(f_1)=1$, so it can represent infinitely many primes. On the other hand, the irreducible $f_2(x)=x(x+1)+2$ over the integers has fixed divisor $\text{div}(f_2)=2$, so it cannot represents infinitely many primes.
	}
\end{exa}

\begin{dfn}\label{dfn1188.810}{\normalfont
		A list of irreducible polynomials $f_1(t),f_2(t), \ldots, f_k(t)\in\Z[t]$ is \textit{admissible} if each polynomial has fixed divisor $\tdiv(f_i)=1$, and the greatest common divisor  $\gcd(f_i,f_j)=1$ for $i\ne j$.
	}
\end{dfn}

\begin{conj} \label{conj1188.800} {\normalfont (Bateman-Horn conjecture)} Given a list of $k$ admissible polynomials $f_1(t),f_2(t), \ldots, f_k(t)\in\Z[t]$, the following holds.
	\begin{equation}\label{eq1188.800}
		\sum_{x\leq n\leq 2x}\Lambda(f_1(n))\Lambda(f_2(n))\cdots\Lambda(f_k(n))=\mathfrak{G}(f)x+o(x),
	\end{equation}
	where $f(t)=f_1(t)f_2(t) \cdots f_k(t)$, $\mathcal{A}(f)=\deg f_1\cdot\deg f_2\cdots\deg f_k$, 
	\begin{equation}\label{eq1188.8 810}
		\mathfrak{G}(f)=\frac{1}{\mathcal{A}(f)}\prod_{p\geq 2} \left(1-\frac{1}{p}\right)^{-k} \left(1-\frac{\rho(p)}{p}\right),
	\end{equation}
	and 
	\begin{equation}\label{eq1188.8 820}
		\rho(p)=\#\{n:f(n)\equiv 0 \bmod p\}.
	\end{equation}
\end{conj}

Extensive materials is available in the literature: on the derivation, see \cite{BH1962}; on the convergence of the product, see \cite{RI2015}; on the oscillatory property, see \cite{FG1991}; and the new materials in the recent literature. All the Hardy-Littlewood conjectures about primes, derived and promoted in \cite{HL1923}, are special cases of the Bateman-Horn conjecture.\\

The number of prime $k$-tuples is expected to be
\begin{equation}\label{eq1188.830}
	\pi_f(x)=\mathfrak{G}(f)\int_2^x\frac{1}{(\log t)^k}dt+o(x)=\mathfrak{G}(f)\frac{x}{(\log t)^k}dt+o(x),
\end{equation}
where $k\geq1$ is an integer. A cursory inspection of \eqref{eq1188.830} seems to limit the parameter $k$ to small values, for example, $k\ll (\log x)/(\log \log x)$.

\section{Primes $k$-Tuples}\label{S1212}
\begin{dfn}\label{dfn1212.100}{\normalfont
		A subset of integers $\mathcal{A}=\{a_0,a_1,\ldots , a_{k-1}\}$ is an admissible $k$-tuple if it not a complete residue number system modulo any prime $p\geq2$. The associated $k$-tuple constant is defined by
		\begin{equation}\label{eq1212.100}
			C(\mathcal{A})=\prod_{p\geq2}\left(1-\frac{\rho(p)}{p} \right)\left( 1-\frac{1}{p}\right)^{-k} , 
		\end{equation}
		where $\rho(p)\geq0$ is the number of solutions of the congruence
		\begin{equation}\label{eq1212.105}
			(x+a_0)(x+a_1)\cdots(x+a_{k-1})\equiv 0 \bmod p.
		\end{equation} 
	}
\end{dfn}

An \textit{optimal admissible} $k$-tuples has the smallest diameter possible. The \textit{diameter} of an admissible $k$-tuple is the smallest distance $d=\min\{|a_i-a_j|\}$. An extensive list of the optimal admissible $k$-tuples is archived in \cite{ktuples}. A few are given in Table \ref{T1212.100}. \\

\begin{table}[h!]
	\centering
	\caption{List of Small $k$-Tuples} \label{T1212.100}
	\begin{tabular}{c|c|c}
		$k$&$(a_0,a_1,\ldots,a_{k-1})$&Diameter $d$\\
		\hline
		2&    $(0,2)$&$2$\\
		3&   $(0,2,6)$,$(0,4,6)$ & $6$\\
		4&   $(0,2,6,8)$, $(0,2,6,10)$& $8$\\
		5&   $(0,2,6,8,12)$, $(0,4,6,10,12)$& $12$\\
	\end{tabular}
\end{table}

The constant attached to a complete residue system automatically reduces to zero. For example, if $\mathcal{A}=\{0,1,2,\ldots , p-1\}$, where $p\geq2$ is a prime, then $C(\mathcal{A})=0$. 
A small sample of the associated density constant is listed in Table \ref{T1212.150}.
\begin{table}[h!]
	\centering
	\caption{List of Density Constant for Small $k$-Tuples} \label{T1212.150}
	\begin{tabular}{c|c|c}
		$f(t)$&$(a_0,a_1)=(0,a_1)$&Density  $C(\mathcal{A})$\\
		\hline
		$t(t+2)$&    $(0,2)$&$0.660162$\\
		$t(t+6)$&   $(0,2\cdot 3)$ & $ 1.32032$\\
		$t(t+30)$&   $(0,2\cdot 3\cdot 5)$& $1.76043$\\
		$t(t+210)$&   $(0,2\cdot 3\cdot 5\cdot 7)$& $4.284370	$\\
		$t(t+2310)$&   $(0,2\cdot 3\cdot 5\cdot 7\cdot 11)$& $4.760412$\\
		$t(t+30030)$&   $(0,2\cdot 3\cdot 5\cdot 7\cdot 11\cdot 13)$& $5.166451$\\		
	\end{tabular}
\end{table}


\section{Preliminary Results For Triple Correlation}\label{S9930}

\begin{lem}\label{lem9930.150} If $x\geq 1$ is a large number, and $c\geq1$ is a constant, then,
	\begin{equation}\label{eq9930.150}
		\sum_{x\leq n\leq 2x} 	 \frac{\Lambda^5(n)}{(\log n)^{10c}}\gg 	 \frac{x}{(\log x)^{10c-4}}\nonumber . 
	\end{equation}
\end{lem}
\begin{proof}[\textbf{Proof}] Let $\pi(t)=\#\{p\leq x\}\gg t/\log t$. Use Definition \ref{dfn9900A.020}, and Theorem \ref{thm9900A.050} to estimate the lower bound.
	\begin{eqnarray}\label{eq9930.160}
		\sum_{x\leq n\leq 2x} 	 \frac{\Lambda^5(n)}{(\log n)^{10c}}&\gg&\sum_{x\leq p\leq 2x} \frac{(\log p)^5}{(\log p)^{10c}} \\
		&\gg&\sum_{x\leq p\leq 2x} \frac{1}{(\log p)^{10c-5}}\nonumber\\
		&\gg&\int_x^{2x}\frac{1}{(\log t)^{10c-5}}d\pi(t)\nonumber\\
		&\gg&\frac{x}{(\log x)^{10c-4}}\nonumber.
	\end{eqnarray}
\end{proof}

\begin{lem}\label{lem9930.250} Let $(a_0=0, a_1, a_2)$ be an admissible triple. If $x\geq 1$ is a large number, and $b>1$ is a constant, then,
	\begin{equation}\label{eq9930.350}
		\sum_{x\leq n\leq 2x} 	 \frac{\Lambda(n)\Lambda(n+a_1)\Lambda(n+a_2)}{(\log n)^{b}}\ll 	 \frac{x}{(\log x)^{b-2}}\nonumber . 
	\end{equation}
\end{lem}
\begin{proof}[\textbf{Proof}] Let $\pi(t)=\#\{p\leq x\}\ll t/\log t$. Use the Definition \ref{dfn9900A.020}, and Theorem \ref{thm9900A.050} to estimate the lower bound.
	\begin{eqnarray}\label{eq9930.360}
		\sum_{x\leq n\leq 2x} 	 \frac{\Lambda(n)\Lambda(n+a_1)\Lambda(n+a_1)}{(\log n)^{b}}&\ll&(\log x)^3\sum_{x\leq p\leq 2x} \frac{1}{(\log p)^{b}} \\
		&\ll&(\log x)^3\sum_{x\leq p\leq 2x} \frac{1}{(\log p)^{b}}\nonumber\\
		&\ll&(\log x)^3\int_x^{2x}\frac{1}{(\log t)^{b}}d\pi(t)\nonumber\\
		&\ll&\frac{x}{(\log x)^{b-2}}\nonumber.
	\end{eqnarray}
\end{proof}

\begin{lem}\label{lem9930.900} Let $x\geq 1$ be a large number, and let $(a_0, a_1, a_2)$ be an admissible triple. If $\Lambda(n+a_0)\Lambda(n+a_1)\Lambda(n+a_2)\ne0$, then, 
	\begin{equation}\label{eq9930.900}
		\sum_{x\leq n\leq 2x} \log \left (1+	\frac{\Lambda(n+a_0)}{(\log n)^{2c}}\right )\log\left (1+	\frac{\Lambda(n+a_1)}{(\log n)^{2c}}\right ) \log\left (1+	\frac{\Lambda(n+a_2)}{(\log n)^{2c}}\right )\gg 	 \frac{x}{(\log x)^{10c-4}}	\nonumber,
	\end{equation}
	where $c\geq 1$ is an arbitrary constant.
\end{lem}
\begin{proof}[\textbf{Proof}] Let $a_0=0$, let $2\leq a_1<a_2$, and let $c\geq1$ be small fixed parameters. The real number inequality $0\leq z^2\leq z$ for $0\leq z<1$, and the fact 
	\begin{eqnarray}\label{eq9930.910}
		0&\leq &\frac{\Lambda(n+a_i)}{(\log x)^{2c}}-\frac{1}{2}\left(\frac{\Lambda(n+a_i)}{(\log x)^{2c}}\right)^2\\
		&\leq&	\log \left (1+	\frac{\Lambda(n+a_i)}{(\log x)^{2c}}\right )\nonumber\\
		&<&1\nonumber,
	\end{eqnarray}	
	for large $n\geq x$, imply that
	\begin{equation}\label{eq9930.920}
		\left (	\frac{\Lambda(n)}{(\log x)^{2c}}-\frac{1}{2}\left(\frac{\Lambda(n)}{(\log x)^{2c}}\right)^2\right )^2\leq	\frac{\Lambda(n+a_i)}{(\log x)^{2c}}-\frac{1}{2}\left(\frac{\Lambda(n+a_i)}{(\log x)^{2c}}\right)^2.
	\end{equation}
	Now, to derive the asymptotic lower bound, substitute \eqref{eq9930.920} into \eqref{eq9930.930}, to obtain 
	
	\begin{eqnarray}\label{eq9930.930}
		V(x)&=&\sum_{x\leq n\leq 2x} \log \left (1+	\frac{\Lambda(n)}{(\log n)^{2c}}\right )\log\left (1+	\frac{\Lambda(n+a_1)}{(\log n)^{2c}}\right )\log\left (1+	\frac{\Lambda(n+a_2)}{(\log n)^{2c}}\right )\nonumber\\
		&\geq&\sum_{x\leq n\leq 2x} \left (	\frac{\Lambda(n)}{(\log n)^{2c}}-\frac{1}{2}\left(\frac{\Lambda(n)}{(\log n)^{2c}}\right)^2\right ) \left (	\frac{\Lambda(n+a_1)}{(\log n)^{2c}}-\frac{1}{2}\left(\frac{\Lambda(n+a_1)}{(\log n)^{2c}}\right)^2\right )\nonumber\\
		&&\hskip 2 in \times \left (	\frac{\Lambda(n+a_2)}{(\log n)^{2c}}-\frac{1}{2}\left(\frac{\Lambda(n+a_2)}{2(\log n)^{2c}}\right)^2\right )\nonumber\\
		&\geq&\sum_{x\leq n\leq 2x} \left (	 \frac{\Lambda(n)}{(\log n)^{2c}}-\frac{1}{2}\left(\frac{\Lambda(n)}{(\log n)^{2c}}\right)^2\right )^5,
	\end{eqnarray}
	since $a_i\geq0$. Expanding the argument, and applying Lemma \ref{lem9930.150}, return
	\begin{eqnarray}\label{eq9930.940}
		V(x)
		&\geq&\sum_{x\leq n\leq 2x} \left (	 \frac{\Lambda(n)}{(\log n)^{2c}}-\frac{1}{2}\left(\frac{\Lambda(n)}{(\log n)^{2c}}\right)^2\right )^5\\
		&\gg&\sum_{x\leq n\leq 2x} 	 \frac{\Lambda^5(n)}{(\log n)^{10c}}	\nonumber\\
		&\gg&\frac{x}{(\log x)^{10c-4}}	\nonumber,
	\end{eqnarray}
	where $c\geq 1$ is a constant.
\end{proof}

\section{Result for Triple Autocorrelation }\label{S3090}
For the admissible triple $(a_0,a_1,a_2)=(0,a_1,a_2)$, the hypothesis take the form
\begin{equation}\label{eq3090.100}
	\sum_{ x\leq n\leq 2x}	\Lambda(n)\Lambda(n+a_1)\Lambda(n+a_2)=O\left( x^{1-\varepsilon}\right) , 
\end{equation} 
where $\varepsilon\in (0,1)$ is a small number, for all large numbers $x\geq x_0$. \\

The number of prime triples $p_n$, $p_n+a_1$, and $p_n+a_2$ over the short interval $[x,2x]$ predicted by the standard conjecture, confer Conjecture \ref{conj1188.800}, is $\pi_3(x)\asymp x(\log x)^{-3}$, and the predicted average gap between prime triples is
\begin{equation}\label{eq3090.050}
	\frac{x}{\pi_3(x)}\asymp(\log x)^3.
\end{equation} 
While, the number of prime triples $p_n$, $p_n+a_1$, and $p_n+a_2$ over the short interval $[x,2x]$ predicted by hypothesis \eqref{eq3090.100} is $\pi_3(x)\asymp x^{1-\varepsilon}(\log x)^{-3}$, and the predicted average gap between prime triples is
\begin{equation}\label{eq3090.060}
	\frac{x}{\pi_2(x)}\asymp x^{\varepsilon}(\log x)^{3}.
\end{equation}

This much weaker hypothesis on the number of prime triples, and the average gap between primes triples will be used to derive a contradiction.

\begin{thm} \label{thm3090.300} If $x\geq1$ is a large number, and $(0,a_1,a_2)$ is an \textit{admissible} triple, then 
	\begin{equation}\label{eq3090.300}
		\frac{x}{(\log x)^{6}}\ll	\sum_{ x\leq n\leq 2x}	\frac{\Lambda(n)\Lambda(n+a_1)\Lambda(n+a_2)}{(\log n)^{6}}. 
	\end{equation} 
	In particular, \begin{equation}\label{eq3090.305}
		\sum_{ x\leq n\leq 2x}	\Lambda(n)\Lambda(n+a_1)\Lambda(n+a_2)\gg x
	\end{equation}
	as $x\to\infty$.
\end{thm}

\begin{proof}[\textbf{Proof:} ]
	To demonstrates that the hypothesis \eqref{eq3090.100} is false, consider the log-correlation function
	\begin{equation}\label{eq3090.310}
		V(x)=\sum_{x\leq n\leq 2x} \log \left (1+	\frac{\Lambda(n)}{(\log n)^{2c}}\right )\log\left (1+	\frac{\Lambda(n+a_1)}{(\log n)^{2c}}\right )\log\left (1+	\frac{\Lambda(n+a_2)}{(\log n)^{2c}}\right ), 
	\end{equation}
	where $c\geq1$ is an arbitrary constant.\\
	
	By hypothesis \eqref{eq3090.100}, there are sufficiently many $n\in [x,2x]$ such that
	\begin{equation}\label{eq3090.320}
		\Lambda(n)\Lambda(n+a_1)\Lambda(n+a_2)\ne0.	
	\end{equation}
	Thus, the result in Lemma \ref{lem9930.900}, and the real number inequality
	\begin{equation}\label{eq3090.325}
		0<\log \left (1+t\right )\log\left (1+u\right )\log\left (1+v\right )\leq tuv	
	\end{equation}
	for $0<t,u,v<1$, yield
	\begin{eqnarray}\label{eq3090.330}
		\frac{x}{(\log x)^{10c-4}}&\ll&\sum_{x\leq n\leq 2x} \log \left (1+	\frac{\Lambda(n)}{(\log n)^{2c}}\right )\log\left (1+	\frac{\Lambda(n+a_1)}{(\log n)^{2c}}\right )\\
		&&\hskip 2.5 in \times \log\left (1+	\frac{\Lambda(n+a_2)}{(\log n)^{2c}}\right )\nonumber\\
		&\ll&\sum_{x\leq n\leq 2x} \frac{\Lambda(n)\Lambda(n+a_1)\Lambda(n+a_2)}{(\log n)^{6c}}\nonumber.
	\end{eqnarray}
	Setting $c=1$, these data imply that 
	\begin{equation}\label{eq3090.340}
		\frac{x}{(\log x)^{6}}\ll\sum_{x\leq n\leq 2x} \frac{\Lambda(n)\Lambda(n+a_1)\Lambda(n+a_2)}{(\log n)^{6}}.
	\end{equation}
	Furthermore, Lemma \ref{lem9930.250} implies that
	\begin{equation}\label{eq3090.350}
		\frac{x}{(\log x)^{6}}\ll\sum_{x\leq n\leq 2x} \frac{\Lambda(n)\Lambda(n+a_1)\Lambda(n+a_2)}{(\log n)^{6}}\ll\frac{x}{(\log x)^{4}}.
	\end{equation}
	The last expression and partial summation imply 
	\begin{equation}\label{eq3090.360}
		\sum_{ x\leq n\leq 2x}	\Lambda(n)\Lambda(n+a_1)\Lambda(n+a_2)\gg x, 
	\end{equation} 
	as $x\to\infty$. Therefore, the relation \eqref{eq3090.360} contradicts the hypothesis \eqref{eq3090.100} for any $\varepsilon\in (0,1)$. Therefore, 
	\begin{equation}\label{eq3090.370}
		\sum_{ x\leq n\leq 2x}	\Lambda(n+a_0)\Lambda(n+a_1)\Lambda(n+a_2)\gg x, 
	\end{equation} 
	as $x\to\infty$. 
\end{proof}

\section{Preliminary Results For $k$-Tuple Correlation}\label{S9930k}

\begin{lem}\label{lem9930k.150} If $x\geq 1$ is a large number, and $c\geq1$ is a constant, then,
	\begin{equation}\label{eq9930k.150}
		\sum_{x\leq n\leq 2x} 	 \frac{\Lambda^{2k-1}(n)}{(\log n)^{4kc-2c}}\gg \frac{x}{(\log x)^{4kc-2c-2k+2}}\nonumber . 
	\end{equation}
\end{lem}
\begin{proof}[\textbf{Proof}] Let $\pi(t)=\#\{p\leq x\}\gg t/\log t$. Use Definition \ref{dfn9900A.020}, and Theorem \ref{thm9900A.050} to estimate the lower bound.
	\begin{eqnarray}\label{eq9930.160k}
		\sum_{x\leq n\leq 2x} 	 \frac{\Lambda^{2k-1}(n)}{(\log n)^{4kc-2c}}&\gg&\sum_{x\leq p\leq 2x} \frac{(\log p)^{2k-1}}{(\log p)^{4kc-2c}} \\
		&\gg&\sum_{x\leq p\leq 2x} \frac{1}{(\log p)^{4kc-2c-2k+1}}\nonumber\\
		&\gg&\int_x^{2x}\frac{1}{(\log t)^{4kc-2c-2k+1}}d\pi(t)\nonumber\\
		&\gg&\frac{x}{(\log x)^{4kc-2c-2k+2}}\nonumber.
	\end{eqnarray}
\end{proof}

\begin{lem}\label{lem9930k.250} Let $(a_0, a_1, \ldots, a_{k-1})$ be an admissible $k$-tuple. If $x\geq 1$ is a large number, and $b>1$ is a constant, then,
	\begin{equation}\label{eq9930.250}
		\sum_{x\leq n\leq 2x} 	 \frac{\Lambda(n+a_0)\cdots \Lambda(n+a_{k-1})}{(\log n)^{b}}\ll 	 \frac{x}{(\log x)^{b-k+1}}\nonumber . 
	\end{equation}
\end{lem}
\begin{proof}[\textbf{Proof}] Let $\pi(t)=\#\{p\leq x\}\ll t/\log t$, and let $a_igeq0$. Use the Definition \ref{dfn9900A.020}, and Theorem \ref{thm9900A.050} to estimate the lower bound.
	\begin{eqnarray}\label{eq9930.260}
		\sum_{x\leq n\leq 2x} 	 \frac{\Lambda(n+a_0)\cdots \Lambda(n+a_{k-1})}{(\log n)^{b}}&\ll&(\log x)^k\sum_{x\leq p\leq 2x} \frac{1}{(\log p)^{b}} \\
		&\ll&(\log x)^k\sum_{x\leq p\leq 2x} \frac{1}{(\log p)^{b}}\nonumber\\
		&\ll&(\log x)^k\int_x^{2x}\frac{1}{(\log t)^{b}}d\pi(t)\nonumber\\
		&\ll&\frac{x}{(\log x)^{b-k+1}}\nonumber.
	\end{eqnarray}
\end{proof}

\begin{lem}\label{lem9930k.900} Let $x\geq 1$ be a large number, and let $(a_0,a_1,\ldots,a_{k-1})$ be an admissible $k$-tuple. If $\Lambda(n+a_0)\cdots\Lambda(n+a_{k-1})\ne0$, then,
	\begin{equation}\label{eq9930k.900}
		\sum_{x\leq n\leq 2x} \log\left (1+	\frac{\Lambda(n+a_0)}{(\log n)^{2c}}\right ) \cdots \log\left (1+	\frac{\Lambda(n+a_{k-1})}{(\log n)^{2c}}\right )\gg 	\frac{x}{(\log x)^{4kc-2c-2k+2}}	\nonumber,
	\end{equation}
	where $c\geq 1$ is an arbitrary constant.
\end{lem}
\begin{proof}[\textbf{Proof}] Let $a_i\geq0$, and $c\geq1$ be small fixed parameters. The real number inequality $0\leq z^2\leq z$ for $0\leq z<1$, and the fact 
	\begin{eqnarray}\label{eq9930k.910}
		0&\leq &\frac{\Lambda(n+a_i)}{(\log x)^{2c}}-\frac{1}{2}\left(\frac{\Lambda(n+a_i)}{(\log x)^{2c}}\right)^2\\
		&\leq&	\log \left (1+	\frac{\Lambda(n+a_i)}{(\log x)^{2c}}\right )\nonumber\\
		&<&1\nonumber,
	\end{eqnarray}	
	for large $n\geq x$, imply that
	\begin{equation}\label{eq9930k.920}
		\left (	\frac{\Lambda(n)}{(\log x)^{2c}}-\frac{1}{2}\left(\frac{\Lambda(n)}{(\log x)^{2c}}\right)^2\right )^2\leq	\frac{\Lambda(n+a_i)}{(\log x)^{2c}}-\frac{1}{2}\left(\frac{\Lambda(n+a_i)}{(\log x)^{2c}}\right)^2.
	\end{equation}
	Now, to derive the asymptotic lower bound, substitute \eqref{eq9930k.920} into \eqref{eq9930k.930}, to obtain 
	
	\begin{eqnarray}\label{eq9930k.930}
		V(x)&=&\sum_{x\leq n\leq 2x} \log\left (1+	\frac{\Lambda(n+a_0)}{(\log n)^{2c}}\right )\cdots\log\left (1+	\frac{\Lambda(n+a_{k-1})}{(\log n)^{2c}}\right )\nonumber\\
		&\geq&\sum_{x\leq n\leq 2x} \left (	\frac{\Lambda(n+a_0)}{(\log n)^{2c}}-\frac{1}{2}\left(\frac{\Lambda(n+a_0)}{(\log n)^{2c}}\right)^2\right ) \left (	\frac{\Lambda(n+a_1)}{(\log n)^{2c}}-\frac{1}{2}\left(\frac{\Lambda(n+a_1)}{(\log n)^{2c}}\right)^2\right )\nonumber\\
		&&\hskip 2 in \times \left (	\frac{\Lambda(n+a_{k-1})}{(\log n)^{2c}}-\frac{1}{2}\left(\frac{\Lambda(n+a_{k-1)}}{(\log n)^{2c}}\right)^2\right )\nonumber\\
		&\geq&\sum_{x\leq n\leq 2x} \left (	 \frac{\Lambda(n)}{(\log n)^{2c}}-\frac{1}{2}\left(\frac{\Lambda(n)}{(\log n)^{2c}}\right)^2\right )^{2k-1},
	\end{eqnarray}
	since $a_i\geq0$. Expanding the argument, and applying Lemma \ref{lem9930k.150}, return
	\begin{eqnarray}\label{eq9930k.940}
		V(x)
		&\geq&\sum_{x\leq n\leq 2x} \left (	 \frac{\Lambda(n)}{(\log n)^{2c}}-\frac{1}{2}\left(\frac{\Lambda(n)}{(\log n)^{2c}}\right)^2\right )^{2k-1}\\
		&\gg&\sum_{x\leq n\leq 2x} 	 \frac{\Lambda^{2k-1}(n)}{(\log n)^{4kc-2c}}	\nonumber\\
		&\gg&\frac{x}{(\log x)^{4kc-2c-2k+2}}	\nonumber,
	\end{eqnarray}
	where $c\geq 1$ is a constant.
\end{proof}

\section{Autocorrelation of $k$-Tuple}\label{S3090k}
The parameter $k\geq1$ is restricted to small integer, this follows from \eqref{eq1188.830}.
For an admissible $k$-tuple $(a_0,a_1,\ldots,a_{k-1})$, the hypothesis take the form
\begin{equation}\label{eq3090k.100}
	\sum_{ x\leq n\leq 2x,}\prod_{0\leq i<k}	\Lambda(n+a_i)=O\left( x^{1-\varepsilon}\right) , 
\end{equation} 
where $\varepsilon\in (0,1)$ is a small number. \\

The number of prime $k$-tuples $n+a_0$, $n+a_1$, \ldots, $n+a_{k-1}$ over the short interval $[x,2x]$ predicted by the standard conjecture, confer Conjecture \ref{conj1188.800}, is $\pi_k(x)\asymp x(\log x)^{-k}$, and the predicted average gap between prime pairs is
\begin{equation}\label{eq3090k.050}
	\frac{x}{\pi_3(x)}\asymp(\log x)^k.
\end{equation} 
While, the number of prime $k$-tuples $n+a_0$, $n+a_1$, \ldots, $n+a_{k-1}$ over the short interval $[x,2x]$ predicted by hypothesis \eqref{eq3090k.100} is $\pi_k(x)\asymp x^{1-\varepsilon}(\log x)^{-k}$, and the predicted average gap between prime triples is
\begin{equation}\label{eq3090k.060}
	\frac{x}{\pi_k(x)}\asymp x^{\varepsilon}(\log x)^{k}.
\end{equation} 

This much weaker hypothesis on the number of prime $k$-tuples, and the average gap between primes $k$-tuples will be used to derive a contradiction.
\begin{thm} \label{thm3090k.300} If $x\geq1$ is a large number, and $(a_0,a_1,\ldots,a_{k-1})$ is an \textit{admissible} $k$-tuple, then 
	\begin{equation}\label{eq3090k.300}
		\frac{x}{(\log x)^{2k}}\ll	\sum_{ x\leq n\leq 2x}	\frac{1}{(\log n)^{2k}}\prod_{0\leq i<k}	\Lambda(n+a_i), 
	\end{equation} 
	where $k\ll (\log \log x)^d$, $d>0$. In particular, \begin{equation}\label{eq3090k.305}
		\sum_{ x\leq n\leq 2x,}	\prod_{0\leq i<k}	\Lambda(n+a_i)\gg x
	\end{equation}
	as $x\to\infty$.
\end{thm}

\begin{proof}[\textbf{Proof:} ]
	To verify that the hypothesis \eqref{eq3090k.100} is false for all large numbers $x\geq x_0$, consider the log-correlation function
	\begin{equation}\label{eq309k.310}
		V(x)=\sum_{x\leq n\leq 2x,}\prod_{0\leq i<k}	\log\left (1+	\frac{\Lambda(n+a_i)}{(\log n)^{2c}}\right ), 
	\end{equation}
	where $c\geq1$ is an arbitrary constant.\\
	
	By hypothesis \eqref{eq3090k.100}, there are sufficiently many $n\in [x,2x]$ such that
	\begin{equation}\label{eq3090K.320}
		\Lambda(n+a_0)\Lambda(n+a_1)\cdots \Lambda(n+a_{k-1})\ne0.	
	\end{equation}
	Thus, the result in Lemma \ref{lem9930k.900}, and the real number inequality
	\begin{equation}\label{eq3090k.325}
		0<\log \left (1+u_0\right )\log\left (1+u_1\right )\cdots \log\left (1+u_{k-1}\right )\leq u_0u_1\cdots u_{k-1}	
	\end{equation}
	for $0<u_i<1$, yield
	\begin{eqnarray}\label{eq3090k.330}
		\frac{x}{(\log x)^{4kc-2c-2k+2}}&\ll& \sum_{x\leq n\leq 2x,}\prod_{0\leq i<k}	\log\left (1+	\frac{\Lambda(n+a_i)}{(\log n)^{2c}}\right )\\
		&\ll&\sum_{x\leq n\leq 2x} \frac{\Lambda(n+a_0)\Lambda(n+a_1)\cdots\Lambda(n+a_{k-1})}{(\log n)^{2kc}}\nonumber.
	\end{eqnarray}
	Setting $c=1$, these data imply that 
	\begin{equation}\label{eq3090k.340}
		\frac{x}{(\log x)^{2k}}\ll\sum_{x\leq n\leq 2x} \frac{\Lambda(n+a_0)\Lambda(n+a_1)\cdots\Lambda(n+a_{k-1})}{(\log n)^{2k}}.
	\end{equation}
	Furthermore, Lemma \ref{lem9930k.250} implies that
	\begin{equation}\label{eq3090k.350}
		\frac{x}{(\log x)^{2k}}\ll\sum_{x\leq n\leq 2x} \frac{\Lambda(n+a_0)\Lambda(n+a_1)\cdots\Lambda(n+a_{k-1})}{(\log n)^{2k}}\ll\frac{x}{(\log x)^{k+1}}.
	\end{equation}
	For all large $x\geq1$, the relation \eqref{eq3090k.340} contradicts the hypothesis \eqref{eq3090k.100} for any $\varepsilon\in (0,1)$. Therefore, 
	\begin{equation}\label{eq3090k.360}
		\sum_{ x\leq n\leq 2x}	\Lambda(n+a_0)\Lambda(n+a_1)\cdots\Lambda(n+a_{k-1})\gg x, 
	\end{equation} 
	as $x\to\infty$. 
	
\end{proof}

\section{Problems}\label{P1212}

\begin{exe} \label{exe1212.200} { \normalfont Let $f(t)=(t+a_0)(t+a_1)\in\Z[t]$. An admissible $k$-tuple $\mathcal{A}=\{a_0, a_1\}$ generates a sequence of prime pairs $n+a_0$ and $n+a_1$, as $n\to \infty$. Explain how to select the integers $a_0$, and $a_1$ to maximize the density constant
		$$C(\mathcal{A})=\prod_{p\geq 2} \left(1-\frac{1}{p}\right)^{-2}\left(1-\frac{\rho(p)}{p}\right),$$		
		where $	\rho(p)=\#\{n:f(n)\equiv 0 \bmod p\}$,
		see Table \ref{T1212.150}. 
	}
\end{exe}
\begin{exe} \label{exe1212.205} { \normalfont Let $f(t)=t(t+a), g(t)=t(t+b)\in\Z[t]$, where $a=2\cdot 3\cdot 5$ and $b=2\cdot 3\cdot 5\cdot 7$. Explain why the sequence of prime pairs $n$ and $n+a$ has a smaller density than the sequence of prime pairs $n$ and $n+b$, as $n\to \infty$. Consult Table \ref{T1212.150} for some detail. 
	}
\end{exe}

\begin{exe} \label{exe1212.220} { \normalfont Let $f(t)=(t+a_0)(t+a_1)(t+a_2)\in\Z[t]$. An admissible $k$-tuple $\mathcal{A}=\{a_0, a_1,a_2\}$ generates a sequence of prime triples $n+a_0$, $n+a_1$, and $n+a_2$, as $n\to \infty$. Explain how to select the integers $a_0$, $a_1$, and $a_2$ to maximize the density constant
		$$C(\mathcal{A})=\prod_{p\geq 2} \left(1-\frac{1}{p}\right)^{-3}\left(1-\frac{\rho(p)}{p}\right),$$		
		where $	\rho(p)=\#\{n:f(n)\equiv 0 \bmod p\}$. 
	}
\end{exe}

\chapter{Summatory Mobius Function over the Shifted Primes} \label{C1212T}
This Chapter provides a self-contained asymptotic result for the summatory Mobius function $\sum_{p \leq x} \mu(p+a) =O \left (x\log^{-c}x \right )$ over the shifted primes, where $a\ne0$ is a fixed parameter, and $c>1$ is a constant.

\section{Introduction} \label{S1212T}
The Mobius function $\mu:\mathbb{N} \longrightarrow \{-1,0,1\}$ is defined by
\begin{equation}\label{eq1212T.050}
	\mu(n) =
	\left \{
	\begin{array}{ll}
		(-1)^{v}     &n=p_1 p_2 \cdots p_v\\
		0           &n \ne p_1 p_2 \cdots p_v,\\
	\end{array}
	\right .
\end{equation}
where the $p_i\geq 2$ are primes. The autocorrelation of the Mobius function 
\begin{equation} \label{eq1212T.060}
	\sum_{n \leq x} \mu(n)\mu(n+a) 
\end{equation}
is a topic of current research in several area of Mathematics, \cite{WD2016}, \cite{TJ2018},  \cite{MR2021}, et alii. The same autocorrelation functions of multiplicative functions over the shifted primes reduce to standard arithmetic averages over the shifted primes, for example, \eqref{eq1212T.060} reduce to the followings: 

\begin{equation} \label{eq1212T.070}
	\sum_{p \leq x} \mu(p)\mu(p+a)=-\sum_{p \leq x} \mu(p+a). 
\end{equation}		
The best result is $\sum_{p \leq x} \mu(p+a)=(1-\delta)\pi(x)$, where $\delta>0$ is a constant, and it is expected that $\sum_{p \leq x} \mu(p+a)=o(\pi(x))$, see \cite[Theorem 1]{HA1989}, and \cite{LJ2021} for extensive details on recent developments. This note proposes the first nontrivial upper bound.

\begin{thm} \label{thm1212T.200} Let $c>1$ be a constant, and let $x>1$ be a large number. If $a \ne0$ is a fixed integer, then
	\begin{equation} \label{eq197.36}
		\sum_{p \leq x} \mu(p+a) =O \left (\frac{x}{(\log x)^{c}} \right )\nonumber. 
	\end{equation}	
\end{thm}

The essential foundational topics are covered in Section \ref{S3970M} to Section \ref{S3970A}, and the proof of the main result is assembled in Section  \ref{S1616T}. Last but not least observe that an autocorrelation function of degree 3,

\begin{equation} \label{eq5980.150}
	\sum_{n \leq x} \mu(n)\mu(n+a)\mu(n+b), 
\end{equation}
where $a,b\ne0$ are fixed integers, reduces to an autocorrelation function of degree 2,
\begin{equation} \label{eq5980.160}
	-\sum_{p \leq x} \mu(p+a)\mu(p+b)
\end{equation}		
over the shifted primes. Accordingly, these two problems are equivalent.

\section{Standard Results for the Mobius Function}\label{S3970M}
A variety of results for the average orders of the Mobius function are stated in this section. A complete proof for a new result in Theorem \ref{thm3970M.200} for the average order over arithmetic progression is included. \\

There are many sharp bounds of the summatory function of the Mobius function, say, $O(xe^{-\sqrt{\log x}})$, and the conditional estimate $O(x^{1/2+\varepsilon})$ presupposes that the nontrivial zeros of the zeta function $ \zeta(\rho)=0$ in the critical strip $\{0<\Re e(s)<1 \}$ are of the form $\rho=1/2+it, t \in \mathbb{R}$. However, the simpler notation will be used here. 
\begin{thm} \label{thmS3970M.050} If $C>0$ is a constant, and $\mu$ is the Mobius function, then, for any large number $x>1$,
	\begin{equation}
		\sum_{n \leq x} \mu(n)=O \left (\frac{x}{\log^{C}x}\right )\nonumber. 
	\end{equation}
\end{thm}
\begin{proof}  See \cite[p.\ 6]{DL2012}, \cite[p.\ 182]{MV2007}.   
\end{proof}

A few results for the average orders over short intervals are proved in the literature. For the new developments, confer \cite[Theorem 1.1]{MT2019}. 
\begin{thm} \label{thmS3970M.080} Let $C>0$ be a constant, and let $\theta>7/12$ be a small real number. If $x>1$ is a large number, and $H\geq x^{\theta}$, then
	\begin{equation}
		\sum_{x\leq n \leq x+H} \mu(n)=O \left (\frac{H}{\log^{C}x}\right )\nonumber. 
	\end{equation}
\end{thm}

The standard proof for the summatory Mobius function over an arithmetic progression is linked to the Siegel-Walfisz Theorem for primes in arithmetic progressions, the upper bounds are proved or discussed in 
\cite[p.\ 424]{IK2004}, \cite[p.\ 385]{MV2007}, et alii. 

\begin{thm} \label{thmS3970M.070} Let $x\geq1$ be a large number, and let $q\ll(\log x)^B$, where $B\geq0$ is an arbitrary constant. If $1\leq a<q$ are relatively prime integers, then, 
	\begin{equation}
		\sum_{\substack{n \leq x\\n\equiv a \bmod q}} \mu(n)=O \left (\frac{x}{\log^{C}x}\right )\nonumber, 
	\end{equation}
	where $C=C(B)>0$ is a constant. 
\end{thm}
\begin{proof}A sketch of the proof appears in \cite[p.\ 385]{MV2007}.
\end{proof}

\section{Equivalent Twisted Exponential Sums}
One of the earliest result for twisted sums is stated below, and the recent version  over short intervals appears in \cite[Theorem 1.5]{MT2019}.
\begin{thm} \label{thm3970M.300} {\normalfont (\cite{DH1937})} If $\alpha$ is a real number, and $D>0$ is an arbitrary constant, then
	\begin{equation}\label{eq3970M.310} 
		\sup_{\alpha\in\R}\sum_{n \leq x} \mu(n)e^{i 2 \pi  \alpha n}<\frac{c_1x}{(\log x)^{D}} \nonumber,
	\end{equation}
	where $c_1=c_1(D)>0$ is a constant depending on $D$, as the number $x \to \infty$.
\end{thm}

Let $f: \C \longrightarrow \C$ be a function, and let $q \in \N$ be a large integer. The finite Fourier transform 
\begin{equation} \label{eq3970M.400}
	\hat{f}(s)=\frac{1}{p} \sum_{0 \leq t\leq p-1}f(t) e^{i \pi st/p}
\end{equation}
and its inverse are used here to derive a summation kernel function.

\begin{dfn} \label{dfn3970M.400} {\normalfont Let $ p$ be a prime, and let $\omega=e^{i 2 \pi/p}$ be a root of unity. The \textit{finite summation kernel} is defined by the finite Fourier transform identity
		\begin{equation} \label{eq3970M.405}
			\mathcal{K}(f(n))=\frac{1}{p} \sum_{0 \leq t\leq p-1,}  \sum_{0 \leq s\leq p-1} \omega^{t(n-s)}f(s)=f(n).\end{equation}
	} 
\end{dfn}
This simple identity is used to derive some equivalent exponential sums as illustrated below.

\begin{lem} \label{lem3970M.410} Let $q>1$ and $1\leq u<q$ be integers. If $a\ne0$, and $D>0$ is an arbitrary constant, then
	\begin{equation}\label{eq3970M.410} 
		\sum_{n \leq x} \mu(n+a)e^{i 2 \pi  u n/q}= \sum_{n \leq x} \mu(n+a)e^{i 2 \pi  n/q}+O \left (\frac{x}{(\log x)^{D}} \right )\nonumber,
	\end{equation}
	where the implied constant depends only on $D$, as the number $x \to \infty$.
\end{lem}

\begin{proof} Let $1<q\leq x$ be an integer, let $p\geq x$ be a large prime, and let $u\ne 0$. \\
	
	Applying the finite summation kernel to $f(n)=\mu(n+a)e^{i 2 \pi u n /q}$, see Definition \ref{dfn3970M.400}, the twisted exponential sum has the form 
	\begin{equation} \label{eq3970M.420}
		\sum_{ n \leq x}  \mu(n+a)e^{i 2 \pi u n /q}=\frac{ 1}{p} \sum_{ n \leq x,} \sum_{0 \leq t\leq p-1,}   \sum_{1 \leq s\leq p-1}\mu(s+a) \omega^{t(n-s)}e^{i2\pi u s/q} .
	\end{equation}
	The term $t=0$ contributes 
	\begin{equation} \label{eq3970M.425}
		E_0(x)=\frac{1}{p}\sum_{ n \leq x,}  \sum_{1 \leq s\leq p-1}\mu(s+a)e^{i2\pi u s/q}\ll\frac{x}{p}\frac{x}{(\log x)^{D_0}}\ll\frac{x}{(\log x)^{D_0}},
	\end{equation}
	where $D_0>0$, and the implied constant depends only on $D_0$, this follows from $x\leq p$, and Theorem \ref{thm3970M.300}. Rearranging it yields
	\begin{eqnarray} \label{eq3970M.430}
		\sum_{ n \leq x} \mu(n+a)  e^{i2\pi u n/q}
		&=&\frac{1}{p} \sum_{ n \leq x,}  \sum_{1 \leq t\leq p-1,}  \sum_{1 \leq s\leq p-1} \mu(s+a)\omega^{t(n-s)}e^{i2\pi u s/q}+E_0(x) \nonumber\\
		&=&\frac{1}{p}  \sum_{1 \leq t\leq p-1}  \left (\sum_{1 \leq s\leq p-1} \mu(s+a)\omega^{-ts}e^{i2\pi u s/q} \right ) \\
		&&\hskip 2.75 in \times \sum_{ n \leq x} \omega^{tn} +E_0(x)\nonumber.
	\end{eqnarray}
	
	Similarly, for $u=1$, the term $t=0$ contributes
	\begin{equation} \label{eq3970M.445}
		E_1(x)=\frac{1}{p}\sum_{ n \leq x,}  \sum_{1 \leq s\leq p-1}\mu(s+a)e^{i2\pi s/q}\ll\frac{x}{p}\frac{x}{(\log x)^{D_1}}\ll\frac{x}{(\log x)^{D_1}},
	\end{equation}
	where $D_1>0$, and the implied constant depends only on $D_1$, this follows from $x\leq p$, and Theorem \ref{thm3970M.300}. Accordingly, the twisted exponential sum has the form 
	\begin{eqnarray} \label{eq3970M.440}
		\sum_{ n \leq x}  \mu(n+a) e^{i2\pi  n/q}
		&=&\frac{ 1}{p} \sum_{ n \leq x,}  \sum_{1 \leq t\leq p-1,}  \sum_{1 \leq s\leq p-1}\mu(s+a) \omega^{t(n-s)}e^{i2\pi  s/q}+E_1(x) \nonumber\\
		&=&\frac{1}{p}  \sum_{1 \leq t\leq p-1}  \left (\sum_{1 \leq s\leq p-1} \mu(s+a)\omega^{-ts}e^{i2\pi  s/q} \right )\\
		&&\hskip 2.75 in \times \sum_{ n \leq x} \omega^{tn} +E_1(x)\nonumber.
	\end{eqnarray}
	Taking the difference of \eqref{eq3970M.430} and \eqref{eq3970M.440} returns 
	
	\begin{eqnarray} \label{eq3970M.450}
		R(x)&=&\sum_{ n \leq x}   \mu(n+a)e^{i2\pi u n/q}-\sum_{ n \leq x}   \mu(n+a)e^{i2\pi  n/q}\\
		&=&\frac{1}{p}  \sum_{1 \leq t\leq p-1}  \left (\sum_{1 \leq s\leq p-1} \mu(s+a)\omega^{-ts}e^{i2\pi u s/q} -\sum_{1 \leq s\leq p-1}\mu(s+a) \omega^{-ts}e^{i2\pi  s/q} \right )\nonumber\\
		&&\hskip 2.5 in \times \sum_{ n \leq x}\omega^{tn}  +E_0(x)+E_1(x)\nonumber.
	\end{eqnarray}
	By Lemma \ref{lem3970M.500}, the inner sum satisfies the upper bound
	\begin{equation}\label{eq3970M.460}
		\left |\sum_{ n \leq x} \omega^{tn}\right |\leq \frac{2p}{\pi t} .
	\end{equation}
	And by Lemma \ref{lem3970M.600}, the two middle sums have the upper bound
	\begin{eqnarray} \label{eq3970M.470}
		&& \left |  \sum_{1 \leq s\leq p-1}\mu(s+a) \omega^{-ts}e^{i2\pi u s/q} -\sum_{1 \leq s\leq p-1} \mu(s+a)\omega^{-ts}e^{i2\pi  s/q} \right | \\
		&\leq &\left |  \sum_{1 \leq s\leq p-1} \mu(s+a)\omega^{-ts}e^{i2\pi u s/q} \right |+ \left | \sum_{1 \leq s\leq p-1} \mu(s+a)\omega^{-ts}e^{i2\pi  s/q} \right |\nonumber\\
		&\ll &\frac{x}{(\log x)^{D_2}}\nonumber,
	\end{eqnarray}
	where $D_2>0$, and the implied constant depends only on $D_2$,
	Together, these estimates lead to
	\begin{eqnarray} \label{eq3970M.480}
		\left | R(x) \right |
		&\leq&\frac{1}{p}  \sum_{1 \leq t\leq p-1}  \left |\sum_{1 \leq s\leq p-1} \mu(s+a)\omega^{-ts}e^{i2\pi u s/q} -\sum_{1 \leq s\leq p-1} \mu(s+a)\omega^{-ts}e^{i2\pi  s/q} \right | \nonumber\\
		&&\hskip 2.5 in \times \left |\sum_{ n \leq x}\omega^{tn} \right | +E_0(x)+E_1(x)\nonumber\\
		&\ll&\frac{1}{p} \sum_{1 \leq t\leq p-1} \left ( \frac{x}{(\log x)^{D_2}} \right ) \cdot \left ( \frac{2p}{\pi t}  \right ) +\frac{x}{(\log x)^{D_0}}+\frac{x}{(\log x)^{D_1}}\nonumber\\
		&\ll&\left ( \frac{x}{(\log x)^{D_2-1}} \right ) +\frac{x}{(\log x)^{D_0}}+\frac{x}{(\log x)^{D_1}}\nonumber\\
		&\ll& \frac{x}{(\log x)^{D}},
	\end{eqnarray}
	where $D=\min \{D_0,D_1,D_2-1\}$, and the implied constant depends only on $D$. 
\end{proof}


\begin{lem} \label{lem3970M.500} Let $p$ be a large prime, and let $1<q<p$ be an integer. If $\omega^{nt}=e^{i2\pi nt/p}$, then
	\begin{equation}\label{eq3970M.500} 
		\left |\sum_{1 \leq n\leq x} \omega^{nt}\right |\leq \frac{2p}{\pi t}  \nonumber,
	\end{equation}
	as the number $p \to \infty$.
\end{lem}

\begin{proof} Summing over the variable $n\leq x$ returns
	\begin{equation} \label{eq3970M.520}
		\sum_{1 \leq n\leq x} \omega^{nt} 
		= \sum_{1 \leq n\leq x} e^{i2\pi nt/p} = \frac{1-e^{\frac{i2\pi  t(x+1)}{p}}}{1-e^{\frac{i2\pi  t}{p}}} .
	\end{equation}
	Routine calculations lead to
	\begin{eqnarray} \label{eq3970M.530}
		\left |\sum_{1 \leq n\leq x} e^{i2\pi nt/p}\right | 
		&\leq & \left |  \frac{1-e^{\frac{i2\pi  t(x+1}{p}}}{1-e^{\frac{i2\pi t}{p}}} \right | \\
		&\leq & \frac{2p}{\pi t} \nonumber.
	\end{eqnarray}
\end{proof}

\begin{lem} \label{lem3970M.600} Let $p$ be a large prime, and let $1<q<p$ be an integer. If $\omega^{-ts}=e^{-i2\pi t s/p}$, and $u/q\ne0$ is a real number, then
	\begin{equation}\label{eq3970M.600} 
		\sum_{1 \leq s\leq p-1}\mu(s+a) \omega^{-ts}e^{i2\pi u s/q}\ll \frac{x}{(\log x)^{D_2}} \nonumber,
	\end{equation}
	where $D_2>0$, and the implied constant depends only on $D_2$, as the number $p \to \infty$.
\end{lem}

\begin{proof} Merge the primitive roots
	\begin{equation}\label{eq3970M.610}
		\omega^{-ts}e^{\frac{i2\pi u s}{q}}=e^{\frac{-i2\pi t s}{p}}\cdot e^{\frac{i2\pi u s}{q}}=e^{\frac{i2\pi  s(up-tq)}{pq}}=e^{i2\pi \alpha s},
	\end{equation}
	where $\alpha=(up-tq)/pq\ne0$ since $q<p$ and $p$ is prime. Summing over the variable $s\geq1$ returns
	\begin{eqnarray} \label{eq3970M.620}
		\sum_{1 \leq s\leq p-1}\mu(s+a) \omega^{-ts}e^{i2\pi u s/q} 
		&=& \sum_{1 \leq s\leq p-1} \mu(s+a)e^{i2\pi  \alpha s} \\
		&=& e^{-i2\pi  \alpha a}\sum_{a \leq r\leq p-1+a} \mu(r)e^{i2\pi  \alpha r} \nonumber\\
		&\ll& \frac{x}{(\log x)^{D_2}} \nonumber,
	\end{eqnarray}
	where $D_2>0$, this follows from Theorem \ref{thm3970M.300}.
\end{proof}

\section{Mobius Sums Over Equivalent Classes}
Observe that for any pair of integers $q\geq 4$, and $a\geq 0$, the sequence of consecutive values
\begin{equation}\label{eq3970M.900}
	\mu(1+a)=\mu(2+a)=\mu(3+a)=\mu(4+a)=\cdots=\mu(q+a)=1,
\end{equation}
is impossible since $\mu$ is not periodic. Likewise, the sequence of consecutive values in arithmetic progression
\begin{equation}\label{eq3970M.905}
	\mu(q+a)=\mu(2q+a)=\mu(3q+a)=\mu(4q+a)=\cdots=\mu([x/q]q+a)=1
\end{equation}
is impossible. Therefore, the cardinalities of the two subsets of integers 
\begin{equation}\label{eq3970M.910}
	\{qm+a\leq x:m\leq x/q\} \quad \text{ and }\quad \{m\leq x/q\},
\end{equation}
with even number of prime factors, satisfy the relation 
\begin{equation}\label{eq3970M.915}
	\sum_{m\leq x/q}\mu^{+}(qm+a)\leq \sum_{m\leq x/q}\mu^{+}(m)=\frac{3}{\pi^2}\frac{x}{q}+O\left(\frac{x}{q(\log x/q)^D}\right).
\end{equation}
A closely related conjecture claims that
\begin{equation}\label{eq3970M.917}
	\sum_{\substack{n\leq x\\ n\equiv a \bmod q}}\mu^2(n)=\frac{6}{\pi^2}\prod_{p\mid q}\left(1-\frac{1}{p^2}\right)\frac{x}{q}+O\left((x/q)^{1/4+\varepsilon}\right),
\end{equation}
for $1\leq q<x^{1-\varepsilon}$, see \cite[p.\ 2]{GR2021}.

\begin{lem} \label{lem3970M.920} Let $x$ be a large number, and let $1<q<x$ be an integer. If $a\geq0$, then,
	\begin{equation}\label{eq3970M.920} 
		\sum_{\substack{ n\leq x\\q\mid n}} \mu(n+a) \ll\frac{x}{q(\log x/q)^{D}},  \nonumber
	\end{equation}
	where $D>0$, and the implied constant depends only on $D$, as the number $x \to \infty$.
\end{lem}
\begin{proof} The asymptotic number of squarefree integers with even number of prime factors in a subset of integers \eqref{eq3970M.910} of cardinality $x/q$ has the formula
	\begin{eqnarray}\label{eq3970M.930}
		\sum_{m\leq x/q}\mu^{+}(qm+a)&=&\sum_{m\leq x/q}\frac{\mu^2(qm+a)+\mu(qm+a)}{2}\\
		&=&c(a,q)\frac{x}{q}+O\left(\frac{x}{q(\log x/q)^D}\right)\nonumber,
	\end{eqnarray} 
	where $0<c(a,q)<1$ is a constant, see \eqref{eq3970M.917}. Likewise, the asymptotic number of squarefree integers with odd number of prime factors in a subset of integers \eqref{eq3970M.910} of cardinality $x/q$, has the formula
	\begin{eqnarray}\label{eq3970M.940}
		\sum_{m\leq x/q}\mu^{-}(qm+a)&=&\sum_{m\leq x/q}\frac{\mu^2(qm+a)-\mu(qm+a)}{2}\\
		&=&c(a,q)\frac{x}{q}+O\left(\frac{x}{q(\log x/q)^D}\right)\nonumber.
	\end{eqnarray}
	The difference yields, 
	\begin{eqnarray}\label{eq3970M.950}
		\sum_{m\leq x/q}\mu(qm+a)&=&\sum_{m\leq x/q}\mu^{+}(qm+a)-\sum_{m\leq x/q}\mu^{-}(qm+a)\\
		&=&O\left(\frac{x}{q(\log x/q)^D}\right)\nonumber.
	\end{eqnarray}
\end{proof}

The upper bound in Lemma \ref{lem3970M.920} is not sharp since
\begin{eqnarray}\label{eq3970M.970}
	\sum_{n\leq x}\mu(n)&=&\sum_{0\leq a<q,}\sum_{m\leq x/q}\mu(qm+a)\\
	&=&O\left(q\cdot \frac{x}{q(\log x/q)^D}\right)\nonumber\\
	&=&O\left(\frac{x}{(\log x/q)^D}\right)\nonumber,
\end{eqnarray}
which is weaker than $\sum_{n\leq x}\mu(n)=O\left(x(\log x)^{-D}\right)$ for $q>1$, see Theorem \ref{thmS3970M.050}.

\begin{thm} \label{thm3970M.980} If $x$ is a large real number, and $a<q<x$ is a pair of integers, then, for any $u\ne0$ and any arbitrary constant $D>0$,
	\begin{equation}\label{eq3970M.980} 
		\sum_{n \leq x} \mu(n+a)e^{i 2 \pi  un/q}=O \left (\frac{x}{(\log x/q)^{D}} \right )\nonumber,
	\end{equation}
	where the implied constant depends only on $D$, as the number $x \to \infty$.
\end{thm}

\begin{proof} Substituting the change of variable $n=qm+v$ returns
	\begin{eqnarray}\label{eq3970M.982}
		\sum_{n \leq x} \mu(n+a)e^{i 2 \pi u n/q}&=&\sum_{0\leq v<q,}\sum_{0\leq m\leq (x-q+)/q}\mu(qm+v+a)e^{i 2 \pi  u(qm+v)/q}\\
		&=&e^{i 2 \pi  uv/q}\sum_{0\leq v<q,}\sum_{0\leq m\leq (x-q+a)/q}\mu(qm+v+a)\nonumber.
	\end{eqnarray}
	Apply Lemma \ref{lem3970M.920} to the inner sum, and take a sum over all the equivalent classes $a\geq 0$, to obtain the upper bound.
	\begin{eqnarray}\label{eq3970M.960}
		e^{i 2 \pi  uv/q}\sum_{0\leq v<q,}\sum_{0\leq m\leq (x-q+)/q}&=&\sum_{0\leq a<q,}\sum_{m\leq x/q}\mu(qm+a)\\
		&=&O\left(q\cdot \frac{x}{q(\log x/q)^D}\right)\nonumber\\
		&=&O\left(\frac{x}{(\log x/q)^D}\right)\nonumber,
	\end{eqnarray} 
	as claimed.
\end{proof}

\section{Mertens Function Over Arithmetic Progressions}\label{S3970A}
A different result for the average order of the Mobius function over arithmetic progressions, sharper than Theorem \ref{thmS3970M.070} is given below. Unlike the proof based on the Siegel-Walfisz Theorem methodology, this technique is based on completely different ideas, and has no limitation on the parameter $q\geq1$.

\begin{thm} \label{thm3970M.200} For any fixed integer $a$, and an arbitrary constant $D>0$. If $q< x$ is a parameter, then
	\begin{equation}\label{eq3970M.200} 
		\sum_{\substack{n \leq x\\q  \mid n}} \mu(n+a)\ll \frac{1}{q}\frac{x}{(\log x/q)^D}\nonumber,
	\end{equation}
	where the implied constant depends only on $D$, as the number $x \to \infty$.
\end{thm}

\begin{proof} Fix an integer $a\ne0$, and let $q< x$. Now, use an indicator function
	\begin{equation}\label{eq3970M.210}
		\frac{1}{q}\sum_{0\leq u<q }e^{i 2 \pi u (n-a)/q}
		=
		\begin{cases}
			1&\text{ if } n=mq+a,\\
			0&\text{ if } n\ne mq+a,
		\end{cases}
	\end{equation}
	to remove the congruence $q\mid n$ in the Mertens sum:
	
	\begin{eqnarray}  \label{eq3970M.220}
		\sum_{\substack{n \leq x\\q  \mid n}} \mu(n+a) 
		&=&\sum_{n \leq x} \mu(n+a) \times \frac{1}{q}\sum_{0\leq u<q }e^{i 2 \pi u (n-a)/q}\\
		&=&\frac{1}{q}\sum_{n \leq x} \mu(n+a) +\frac{1}{q}\sum_{n \leq x} \mu(n+a)\sum_{1\leq  u<q }e^{i 2 \pi u (n-a)/q}\nonumber\\
		&=&V_0(x)\;+\;V_1(x)\nonumber.
	\end{eqnarray}
	
	The first term is bounded by
	\begin{equation}  \label{eq3970M.230}
		V_0(x)
		=\frac{1}{q}\sum_{n \leq x} \mu(n+a) \ll\frac{1}{q}\frac{x}{(\log x)^{D_0}},
	\end{equation}
	where $D_0>0$, and the implied constant depends only on $D_0$, see Theorem \ref{thmS3970M.050}. Now, use Lemma \ref{lem3970M.410} to rewrite the second term in the following way, (it removes the dependence on the variable $u\ne0$).
	\begin{eqnarray}  \label{eq3970M.240}
		V_1(x)&=&\frac{1}{q}\sum_{n \leq x} \mu(n+a)\sum_{1\leq u<q }e^{i 2 \pi u (n-a)/q}\\
		&=&\frac{1}{q}\sum_{1\leq u<q }e^{-i 2 \pi a u /q} \sum_{n \leq x} \mu(n+a)e^{i 2 \pi u n/q}  \nonumber\\
		&\leq&\frac{1}{q}\sum_{1\leq u<q }e^{-i 2 \pi a u /q} \left(\sum_{n \leq x} \mu(n+a)e^{i 2 \pi  n/q} +\frac{c_1x}{(\log x)^{D_1}}  \right)\nonumber,
	\end{eqnarray}
	where $D_1>0$, and $c_1>0$ is a constant depending on $D_1$. By Theorem \ref{thm3970M.980}, the inner sum
	\begin{equation}\label{eq3970M.245} 
		\sum_{n \leq x} \mu(n+a)e^{i 2 \pi  n/q}\leq\frac{c_2x}{(\log x/q)^{D_2}},
	\end{equation}
	where $D_2>0$, and $c_2$ is a constant depending on $D_2$. Hence, the second term has the following upper bound.
	\begin{eqnarray}  \label{eq3970M.250}
		\left |V_1(x)\right | &\ll&\frac{1}{q}\left |\sum_{1\leq u<q }e^{-i 2 \pi a u /q}\right |\left | \sum_{n \leq x} \mu(n+a)e^{i 2 \pi  n/q}+\frac{c_1x}{(\log x)^{D_1}} \right|  \\
		&\ll&\frac{1}{q}\left |-1\right |\left (\frac{c_2x}{(\log x/q)^{D_2}}+ \frac{c_1x}{(\log x)^{D_1}}\right)  \nonumber\\
		&\ll&\frac{1}{q} \frac{x}{(\log x/q)^{D_2}}  \nonumber,
	\end{eqnarray}
	where the factor $\sum_{1\leq u<q }e^{-i 2 \pi a u /q}=-1$. Summing \eqref{eq3970M.230} and \eqref{eq3970M.250}, and setting $D=\min\{D_1,D_2\}$ completes the proof.
\end{proof}

The evidence generated by random numerical experiments are within the expected conditional estimate
\begin{equation}\label{eq3970M.260} 
	\sum_{\substack{n \leq x\\q  \mid n}} \mu(n+a)\ll \frac{1}{q}x^{1/2+\varepsilon}.
\end{equation}

\section{Proof of the Main Result } \label{S1616T}
The analysis of the plain average order over the shifted primes

\begin{equation}\label{eq1616T.050}
	\sum_{p \leq x} \mu(p+a)
\end{equation}
seems to be unmanageable. But, the introduction of the weighted prime indicator function, (vonMangoldt function),  
\begin{equation}\label{eq1616T.050}
	\Lambda(n)=
	\begin{cases}
		\log n &\text{ if } n=p^k,\\
		0&\text{ if } n\ne p^k,\\
	\end{cases}
\end{equation}
where $p^k$ is a prime power, the identity
\begin{equation}\label{eq1616T.070}
	\Lambda(n)=-\sum_{d\mid n}\mu(d)\log d,
\end{equation}
see \cite[Theorem 2.11]{AT1976}, and the result in Theorem \ref{thm3970M.200} change everything. 

\begin{proof} ({\bfseries Theorem \ref{thm1212T.200}}) To remove the reference to primes, insert the vonMangoldt function to obtain the equivalent form 
	\begin{equation}   \label{eq1616T.100}
		\sum_{n \leq x} \Lambda(n)\mu(n+a) . 
	\end{equation}		
	Applying the identity \eqref{eq1616T.070} and reversing the order of summation produce the followings. 
	\begin{eqnarray}  \label{eq1616T.110}
		\sum_{n \leq x} \Lambda(n)\mu(n+a) &=&-\sum_{n \leq x} \mu(n+a)\sum_{d\mid n} \mu(d) \log d\\ 
		&=&-\sum_{d \leq x} \mu(d)\log(d) \sum_{\substack{n \leq x\\ d\mid n}} \mu(n+a) \nonumber. 
	\end{eqnarray} 
	Letting $q=d$, and applying Theorem \ref{thm3970M.200} yield: 
	\begin{eqnarray} \label{eq1616T.120}
		\left |\sum_{d \leq x} \mu(d)\log(d) \sum_{\substack{n \leq x\\ d\mid n}} \mu(n+a)\right |
		&\leq& \sum_{d \leq x}\left | \mu(d)\log(d)\right | \left |\sum_{\substack{n \leq x\\ d\mid n}} \mu(n+a)\right |\\
		&\ll&\sum_{d \leq x} \log(d) \left (\frac{1}{d} \frac{x}{(\log x/d)^{D}}\right ) \nonumber\\
		&\ll&\frac{x}{(\log x)^{D}}\sum_{d \leq x} \frac{\log d}{d} \nonumber,  
	\end{eqnarray} 
	where $D>0$ is an arbitrary constant. The estimate of the finite sum 
	\begin{equation}\label{eq1616T.130}
		\sum_{n\leq x}\frac{\log n}{n}\ll (\log x)^2,
	\end{equation}
	is a routine calculation. Thus, replacing this estimate yields
	
	\begin{eqnarray} \label{eq1616T.140}
		\left |\sum_{d \leq x} \mu(d)\log(d) \sum_{\substack{n \leq x\\ d\mid n}} \mu(n+a)\right |
		&\ll&\frac{x}{(\log x)^{D}}\sum_{d \leq x} \frac{\log d}{d} \\
		&\ll&\frac{x}{(\log x)^{D}}(\log x)^2 \nonumber\\
		&\ll&\frac{x}{(\log x)^{c} }\nonumber,  
	\end{eqnarray} 
	where $D-2\geq c>0$ is a constant.
\end{proof}


\chapter{Some Arithmetic Correlation Functions} \label{s8}
The autocorrelations for a few arithmetic functions $f:\mathbb{N} \longrightarrow \mathbb{C}$ are sketched in this section. 

\section{Divisors Correlation Functions}
The shifted divisor problem $\sum_{n \leq x} d(n) d(n+1)$ was estimated in 1928. 

\begin{lem} \label{lem600.1} {\normalfont (\cite{IA1928}) } Let $k \ne 0 $ be a fixed integer. Then
\begin{enumerate}
\item For a large number $x \geq 1$,
$$\label{eq600.00}
\sum_{n \leq x} d(n)d(n+k)=\frac{6}{\pi^2} \frac{\sigma(k)}{k}x \log^2(x)+o(x \log^2 x)	.
$$
\item For a large integer $n\geq 1$,
$$\label{eq600.01}
\sum_{k \leq n} d(n)d(n-k)=\frac{6}{\pi^2} \sigma(n) \log^2(n)+o( \sigma(n) \log^2(n)).	
$$
\end{enumerate}	
\end{lem}

The above data and the asymptotic $\sum_{n \leq x} d(n)^2=c_0x \log^3 x +O\left (x \log^2 x \right )$  suggest that the autocorrelation of the divisor function has a two-value autocorrelation
\begin{equation} \label{eq132.10}
\sum_{n \leq x} d(n) d(n+t)=
\begin{cases} \displaystyle c_0x \log^3 x +O\left (x \log^2 x \right ) & t=0,\\
\displaystyle c_0x \log^2 x +O\left (x \log x \right ) &k\ne 0.
\end{cases}
\end{equation}

The next level of complexity $\sum_{n \leq x} d_k(n) d_m(n+1)$ for various integers parameters $k,m\geq2$ has a vast literature. In contrast, the analysis for the triple correlation $\sum_{n \leq x} d(n) d(n+1)d(n+2)$ is relatively new. Some rudimentary analysis was established in 2015, see \cite{BV2015}, and the functions fields version was proved in \cite{AR2014}.\\

\section{Autocorrelation of the Sum of Divisors Function }	\label{s4135.00}
\begin{lem} \label{lem600.2} {\normalfont (\cite{IA1928}) }  Let $k \ne 0 $ be a fixed integer. Then 
\begin{enumerate} [font=\normalfont, label=(\roman*)]
\item For a large number $n\geq 1$,
$$\sum_{n \leq x} \sigma_a(n) \sigma_b(n+k)=\frac{1}{a+b}\frac{\zeta(a+1)\zeta(b+1)}{\zeta(a+b+2)} \sigma_{-a-b-1}(k) x^{a+b+1} +o(x^{a+b+1})	
$$
\item For a large integer $n\geq 1$,
$$\label{eq135.45}
\sum_{k \leq n} \sigma_a(n) \sigma_b(n-k)=\frac{\Gamma(a+1)\Gamma(b+1)}{\Gamma(a+b+2)}\frac{\zeta(a+1)\zeta(b+1)}{\zeta(a+b+2)} \sigma_{a+b+2}(n) +o(\sigma_{a+b+2}(n))	
$$
\end{enumerate}
\end{lem}	

More details appear in \cite[p.\ 208]{IA1928}, and \cite{AR2007}. A recent proof is given in \cite[Corollary 1]{CS2015}.
\section{Autocorrelation of the Euler Totient Function }	\label{s1537.00}	

\begin{lem} \label{lem600.3} {\normalfont (\cite{IA1928}) } Let $k \ne 0 $ be a fixed integer. Then 
\begin{enumerate} [font=\normalfont, label=(\roman*)]
\item For a large number $n\geq 1$,
$$ \label{eq137.40}
\sum_{n \leq x} \varphi_a(n) \varphi_b(n+k)=\frac{c_k}{3} x^{3} +o(x^{3}),	
$$
\item For a large integer $n\geq 1$,
$$
\label{eq137.48}
\sum_{k \leq n} \varphi_a(n) \varphi_b(n-k)=\frac{c_k}{6} n^{3} +o(n^{3}),
$$	
where the constant is defined by
$$\label{eq137.43}
c_k=\prod_{ p \mid k} \left ( \frac{p^3-2p+1}{p(p^2-2) } \right ) \prod_{ p \geq 2} \left (1- \frac{2}{p^2} \right ).	
$$
\end{enumerate}
\end{lem}
The earliest proof seems to be that appearing in \cite[p.\ 208]{IA1928}, a recent proof is given in \cite[Corollary 2]{CS2015}.

\section{Divisor And vonMangoldt Correlation Functions}
The shifted prime divisor problem, better known as the Titchmarsh divisor problem, that is  
\begin{equation}
\sum_{p \leq x} d(p-a)=a_0x+a_1\li(x)+O\left( \frac{x}{\log^Cx} \right )
\end{equation} 
where $a \ne0$ is a fixed integer, and $a_0,a_1>0$ are constants, was conditionally estimated in 1931, see \cite{TE1931}, and unconditionally in \cite{LJ1963}. Later the analysis was simplified in \cite{RG1965}. Utilizing summation by part, this result is equivalent to the correlation of the vonMangoldt function and the divisor function. Specifically, the analysis for
\begin{equation}
\sum_{n \leq x} \Lambda(n)d(n-a)
\end{equation} 
and $\sum_{p \leq x}d(p+a)$ are equivalent.

\section{Characters Correlation Functions}
The characters $ \chi : \mathbb{Z} \longrightarrow \mathbb{C}$ are periodic and completely multiplicative functions modulo some integers $q \geq 1$. Often, these properties allow simpler analysis of the character correlation functions. Several classes of these correlation functions have been settled. One of these is the binary sequence of (Legendre) quadratic symbol $f(n)= \chi(n)$, where $\chi(n) \equiv n^{(p-1)/2} \mod p$. Specifically,
\begin{equation}
\sum_{n \leq x} \chi(n+\tau_{1}) \chi(n+\tau_{2}) \cdots \chi(n+\tau_{k})=O(k x^{1/2}\log x).
\end{equation}

This is estimate, which is derived from the Weil bound, is dependent on the degree $k \geq 1$, and the sequence of integers $\tau_{1},\tau_{2}, \ldots,\tau_{k}$, see \cite[p.\ 183]{BE2009}, \cite[p.\ 112]{CS2002}, and the literature.\\


\chapter{Twisted Arithmetic Sums} \label{s12}
A sample of finite sums of arithmetic functions twisted by the Mobius function are given here.\\

\begin{lem} \label{lem3.1} Let $C> 1$ be a constant, and let $d(n)=\sum_{d|n}1$ be the divisors function. Then, for any sufficiently large number $x>1$,
	\begin{equation} 
		\sum_{n \leq x} \mu(n)d(n) =O \left (\frac{x}{(\log x)^{C}} \right ). 
	\end{equation} 
\end{lem}

\begin{proof} Rewrite the finite sum as 
	\begin{equation}
		\sum_{n\leq x} \mu(n)d(n)=\sum_{n\leq x} \mu(n) \sum_{d|n} 1
		=\sum_{d\leq x,} \sum_{n\leq x/d} \mu(n).  
	\end{equation}
	
	Next, applying Theorem \ref{thmMF222.050}  to the inner finite sum yields: 
	\begin{eqnarray}
		\sum_{d\leq x,} \sum_{n\leq x/d} \mu(n) 
		&=&O \left (\frac{x}{\log^{C}x} \sum_{d\leq x} \frac{1}{d}\right ) \\
		&=& O \left (\frac{x}{\log^{C-1}x}\right ),   \nonumber
	\end{eqnarray}
	with $C>1$. 
\end{proof}

Exactly the same result is obtained via the hyperbola method, see \cite[p. 322]{RM2008}. For $C>1$, this estimate of the twisted summatory divisor function is nontrivial. The summatory divisor function  has the asymptotic formula $\sum_{n \leq x}d(n) =x(\log x+2\gamma-1)+O (x^{1/2} )$. 
\\

\begin{lem} \label{lem3.2} Let $C>1$ be a constant, and let $d(n)=\sum_{d|n}1$ be the divisor function. Then, for any sufficiently large number $x>1$,
	\begin{equation}
		\sum_{n \leq x} \frac{\mu(n)d(n)}{n} =O \left (\frac{1}{(\log x)^{C}} \right ).
	\end{equation} 
\end{lem}

\begin{proof} Let $U(x)= \sum_{n \leq x} \mu(n)d(n)$. A summation by parts leads to the integral representation 
	\begin{equation}
		\sum_{n\leq x} \frac{\mu(n)d(n)}{n}= \int_{1}^{x} \frac{1}{t} d U(t).  
	\end{equation}
	For information on the Abel summation formula, see \cite[p.\ 4]{CR2006}, \cite[p.\ 4]{MV2007},  \cite[p.\ 4]{TG2015}. Evaluate the integral: 
	\begin{equation}
		\int_{1}^{x} \frac{1}{t} d U(t)=\frac{1}{x} \cdot  O \left (\frac{x}{\log^{C}x} \right )+ \int_{1}^{x}\frac{1}{t^2} U(t)dt=O \left (\frac{x}{\log^{C}x} \right ) ,  
	\end{equation}
	where the constant is $C-1>0$. 
\end{proof}

\begin{thm} \label{thm3.3} If $C>0$ is a constant, and $\lambda$ is the Liouville function, then, for any large number $x>1$,\\
	\begin{enumerate} [font=\normalfont, label=(\roman*)]
		\item $ \displaystyle \sum_{n \leq x} \lambda(n)=O \left (\frac{x}{\log^{C}x}\right )$. 
		\item $ \displaystyle \sum_{n \leq x} \frac{\lambda(n)}{n}=O \left (\frac{1}{\log^{C}x}\right ). 
		$
	\end{enumerate}
\end{thm}

\begin{proof} These follow from Theorems \ref{thmMF222.050}, and \ref{thm3.2} via Lemma \ref{lem297.93}.  
\end{proof}

\begin{lem} \label{lem3.3} Let $C> 1$ be a constant, and let $d(n)=\sum_{d|n}1$ be the divisors function. Then, for any sufficiently large number $x>1$,
	\begin{equation} 
		\sum_{n \leq x} \lambda(n)d(n) =O \left (\frac{x}{(\log x)^{C}} \right ). 
	\end{equation}
\end{lem}

\begin{proof} Use Lemma \ref{lem297.93}, and Lemma \ref{lem3.2}. \end{proof}

\begin{lem} \label{lem3.4} Let $C>1$ be a constant, and let $d(n)=\sum_{d|n}1$ be the divisor function. Then, for any sufficiently large number $x>1$,
	\begin{equation}
		\sum_{n \leq x} \frac{\lambda(n)d(n)}{n} =O \left (\frac{1}{(\log x)^{C}} \right ).
	\end{equation} 
\end{lem}

\begin{proof} Let $R(x)= \sum_{n \leq x} \lambda(n)d(n)$. Now use summation by parts as illustrated in the proof of Lemma \ref{lem3.3}.  
\end{proof}

\section{Conditional Estimates}
The conditional estimates assume the optimal zerofree region $\{s \in \mathbb{C}: \Re e(s)>1/2 \}$ of the zeta function. The conditional results offer sharper bounds.\\

\begin{lem} \label{lem3.5} Suppose that $ \zeta(\rho)=0 \Longleftrightarrow \rho=1/2+it, t \in \mathbb{R}$. Let $d(n)=\sum_{d|n}1$ be the divisor function. Then, for any sufficiently large number $x>1$,
	\begin{equation}
		\sum_{n \leq x} \mu(n)d(n) =O \left (x^{1/2} \log x \right ).
	\end{equation} 
\end{lem}
\begin{proof} The generating series has the expression
	\begin{eqnarray}
		\sum_{n \geq 1} \frac{\mu(n)d(n)}{n^s} &=& \prod_{p \geq 2} \left (1-\frac{2}{p^s} \right ) \nonumber \\
		&=& \frac{1}{\zeta(s)}\prod_{p \geq 2} \left (1-\frac{2}{p^s} \right)\left (1-\frac{1}{p^s} \right)^{-1}  \\
		&=&\frac{1}{\zeta(s)}\prod_{p \geq 2} \left (1-\frac{1}{p^s-1} \right)=\frac{g(s)}{\zeta(s)}\nonumber,
	\end{eqnarray}
	where $g(s)$ is an absolutely convergent holomorphic function on the complex half plane $\Re e(s)>1/2$. Let $x \in \mathbb{R}-\mathbb{Q}$ be a large real number. Applying the Perron formula returns: 
	\begin{equation}
		\sum_{n\leq x} \mu(n)d(n)=\frac{1}{i2 \pi}\int_{c-i\infty}^{c+i \infty}\frac{g(s)}{\zeta(s)}\frac{x^s}{s}ds=\sum_{s\in \mathcal{P}} \text{Res}(s,f(s)),  
	\end{equation}
	
	where $c$ is a constant, and $\mathcal{P}=\{0,\rho :\zeta(\rho)=0\}$ is the set of poles of the meromorphic function $f(s)=\frac{g(s)}{\zeta(s)} \frac{x^s}{s}$. Using standard analytic methods, \cite[p. 139]{MV2007}, \cite[p.\ 219]{TG2015}, et cetera, this reduces to the sum of residues
	\begin{equation}
		\sum_{s\in \mathcal{P}} \text{Res}(s,f(s))
		= O \left (x^{1/2} \log^2 x \right ).  
	\end{equation}
	Let $x \in \mathbb{N}$ be a large integer, and $\varepsilon>0$ be an arbitrarily small number. Since the average
	\begin{equation}
		\frac{1}{2} \left (\sum_{n\leq x-\varepsilon} \mu(n)d(n)+\sum_{n\leq x+\varepsilon} \mu(n)d(n)\right)=O \left (x^{1/2} \log^2 x \right ),  
	\end{equation}
	holds for all integers $x \geq 1$, the upper bound holds for all large real numbers $x \geq 1$. 
\end{proof}

\begin{lem} \label{lem3.6} Suppose that $ \zeta(\rho)=0 \Longleftrightarrow \rho=1/2+it, t \in \mathbb{R}$. Let $d(n)=\sum_{d|n}1$ be the divisor function. Then, for any sufficiently large number $x>1$, 
	\begin{equation}
		\sum_{n \leq x} \frac{\mu(n)d(n)}{n} =O \left ( \frac{\log^2 x}{x^{1/2}} \right ).
	\end{equation}
\end{lem}
The proofs are similar to those in Lemmas \ref{lem3.3} and \ref{lem3.4}, but use the conditional results in Lemma \ref{lem3.5}.

\begin{thm} \label{thm12.1} Suppose that $ \zeta(\rho)=0 \longleftrightarrow \rho=1/2+it, t \in \mathbb{R}$. Let $\varphi$ be the totient function, and let $x \geq 1$ be a large number. Then, 
\begin{enumerate} [font=\normalfont, label=(\roman*)]
 \item $ \displaystyle \sum_{n \leq x} \mu(n) \varphi(n) =O \left (x^{3/2}\log x \right) . $    
\item $ \displaystyle\sum_{n \leq x} \lambda(n) \varphi(n) =O \left (x^{3/2}\log x \right).   $ 
 \end{enumerate}
\end{thm}
\begin{proof} Let $\varphi(n)=\sum_{d|n} \mu(d)d$. Rewrite the sum as 
\begin{equation}
\sum_{n\leq x} \mu(n) \varphi(n) = \sum_{n\leq x} \mu(n)\sum_{d|n} \mu(d)d =\sum_{d \leq x} \mu(d)d \sum_{n\leq x,d|n} \mu(n).
\end{equation}
Now let $n=dm \leq x $. Substituting this returns 
\begin{equation}
\sum_{d \leq x} \mu(d)d \sum_{n\leq x,d|n} \mu(n)=\sum_{d \leq x} \mu(d)^2d \sum_{m\leq x/d} \mu(m).
\end{equation} 
Using Lemma 3.1 yields
\begin{equation}
\sum_{d \leq x} \mu(d)^2d \sum_{m\leq x/d} \mu(m) \ll \sum_{d \leq x} \mu(d)^2d \left ( \dfrac{x^{1/2}\log (x/d)}{d^{1/2}} \right ) \ll(x^{1/2}\log x)\sum_{d \leq x} \mu(d)^2d^{1/2} .
\end{equation}
Let $Q(x)= \sum_{n \leq x}\mu(n)^2 \ll x$. A summation by parts leads to the integral representation \\
\begin{equation}
(x^{1/2}\log x)\sum_{d \leq x} \mu(d)^2d^{1/2}
=\int_{1}^{x} t^{1/2} \cdot d Q(t).  
\end{equation}
Evaluate the integral to complete the proof. 
\end{proof}



\chapter{Results for the Fractional vonMangoldt Function}\label{C3131}

\section{Fractional vonMangoldt Autocorrelation Function}\label{S3131}
As usual $[x]=x-\{x\}$ denotes the largest integer function. 
\begin{thm} {\normalfont(\cite{BS2018} )}  \label{thm437.01} Let $f$ be a complex-valued arithmetic function and assume that there exists $0<\alpha<2$ such that 
	\begin{equation}
		\sum_{n\leq x}\left | f(n) \right |^2 \ll x^{\alpha}.
	\end{equation}
	Then
	\begin{equation}
		\sum_{n\leq x}f \left( [x/n] \right ) =x\sum_{n\geq 1}\frac{f(n)}{n(n+1)}+O\left ( x^{(\alpha+1)/3}(\log x)^{(1+\alpha)(2+\varepsilon_2(x))/6} \right ),
	\end{equation}
	where $\varepsilon_2(x)\leq 1$.
\end{thm}
This result provides a different and simpler method for proving the existence of primes in fractional sequences of real numbers. For example, Beatty sequences, and Piatetski-Shapiro primes. The next few results are applications to the fractional summatory function of arithmetic functions associated with primes in sequences of real numbers.
\begin{thm} {\normalfont  }  \label{thm437.05} Let $x \geq 1$ be a large number and let $\Lambda$ be the von Mangoldt function. Then 
	\begin{equation}
		\sum_{n\leq x} \Lambda \left( [x/n] \right )
		=a_kx+O\left ( x^{(2+\varepsilon)/3}(\log^2 x) \right ),
	\end{equation}
	where the constant is
	\begin{equation}
		a_k=\sum_{n\geq 1}\frac{ \Lambda \left( n \right ) }{n(n+1) }>0.
	\end{equation}
\end{thm}

\begin{proof} Let $f(n)=\Lambda \left( n \right ).$ The condition
	\begin{equation}
		\sum_{n\leq x}\left | f(n) \right |^2=  \sum_{n\leq x}\Lambda^2 \left( n \right ) \leq x \log^2 x  \ll x^{\alpha}
	\end{equation}
	is satisfied with $\alpha=1+\varepsilon$ for any small number $\varepsilon>0$. Applying Theorem \ref{thm437.01} yield
	\begin{equation}
		\sum_{n\leq x} \Lambda \left( [x/n] \right )
		=x\sum_{n\geq 1}\frac{ \Lambda \left( n \right ) }{n(n+1) }+O\left ( x^{(2+\varepsilon)/3}(\log^2 x) \right ).
	\end{equation}
	The constant is
	\begin{eqnarray}
		a_k&=&\sum_{n\geq 1}\frac{ \Lambda \left( n \right )}{n(n+1) } \nonumber\\
		&\geq &\frac{ \Lambda \left( n_0 \right ) }{n_0 (n_0+1)} \\
		&>&0 \nonumber,
	\end{eqnarray}
	where $n_0=2$ is the first prime in the sequence primes $\{2,3,5,7,11, \ldots \}$.
\end{proof}

\section{Fractional Infinite Series} \label{s13}
The alternating fractional function has the Fourier series
\begin{equation}
	D(x)=-\frac{1}{\pi}\sum_{n\geq1} \frac{\sin(2 \pi nx)}{n}
	=\left \{\begin{array}{ll}
		\{x\}-1/2     &x \not \in \mathbb{Z},\\
		0           &x \in \mathbb{Z}.\\
	\end{array}
	\right.
\end{equation}

\begin{thm} \label{thm1300} {\normalfont (\cite{DH1937})} Let $a_n,n\geq 1,$ be a sequence of numbers, and let $A_n=\sum_{d|n}a_d$. If the series $\sum_{n\geq1} \frac{a_n}{n^s}$ is absolutely convergent for $\Re e(s)>1$, then, for any $x \in \mathbb{R}$, 
	\begin{equation}
		\sum_{n\geq1} \frac{a_n}{n}D(nx) =-\frac{1}{\pi}\sum_{n\geq1} \frac{A_n}{n}\sin(2\pi nx).
	\end{equation}
\end{thm} 
This result was proved in \cite{DH1937}, and improved in \cite{SL1976}.\\

\begin{lem}  \label{lem1300} Let $D(x)$ be the Fourier series of the function $\{x\}-1/2$. Then, 
	\begin{enumerate} [font=\normalfont, label=(\roman*)]
		\item The alternating fractional function $D(nx)$ is not orthogonal (it is correlated) to the Liouville function $\lambda(n)$ for all $x\in \mathbb{R-Z}$.
		
		\item The alternating fractional function $D(nx)$ is not orthogonal (it is correlated) to the Mobius function $\mu(n)$ for all $2x\ne m$  with $m \in \mathbb{N}$. 
	\end{enumerate}
\end{lem}
\begin{proof} (i) Let $a_n=\lambda(n)$, and let 
	\begin{equation}
		A_n=\sum_{d|n}\lambda(d)=
		\left \{
		\begin{array}{ll}
			1     &n=m^2,\\
			0           &n\ne m^2.\\
		\end{array}
		\right.
	\end{equation}
	Then, the series $\sum_{n\geq1} \frac{a_n}{n^s}=\sum_{n\geq1} \frac{\lambda(n)}{n^s}$ is absolutely convergent for $\Re e(s)>1$. And the right side of the series
	\begin{equation}
		\sum_{n\geq1} \frac{\lambda(n)}{n}D(nx) =-\frac{1}{\pi}\sum_{n\geq1} \frac{A_n}{n}\sin(2\pi nx)=-\frac{1}{\pi}\sum_{n\geq1} \frac{1}{n^2}\sin(2\pi n^2x)
	\end{equation}
	converges to a nonzero value if and only if $x\ne 0$. Therefore, by Theorem \ref{thm1300}, left side converges to the same nonzero number. Specifically,
	\begin{equation}
		\sum_{n\geq1} \frac{\lambda(n)}{n}D(nx) =c_0+O\left ( \frac{1}{\log x} \right ).
	\end{equation}
	This implies that the functions $\lambda(n)$ 
	and $D(nx)$ are not orthogonal (but are correlated). 
	(ii) Let $a_n=\mu(n)$, and let 
	\begin{equation}
		A_n=\sum_{d|n}\mu(d)=
		\left \{
		\begin{array}{ll}
			1     &n=1,\\
			0           &n\ne 1.\\
		\end{array}
		\right.
	\end{equation}
	Then, the series $\sum_{n\geq1} \frac{a_n}{n^s}=\sum_{n\geq1} \frac{\mu(n)}{n^s}$ is absolutely convergent for $\Re e(s)>1$. And the right side of the series
	\begin{equation}
		\sum_{n\geq1} \frac{\mu(n)}{n}D(nx)=-\frac{1}{\pi}\sum_{n\geq1} \frac{A_n}{n}\sin(2\pi nx) =-\frac{1}{\pi} \sin(2\pi x)
	\end{equation}
	is a nonzero value if and only if $2x\ne m$, where $m \in \mathbb{N}$. Therefore, by Theorem \ref{thm1300}, left side converges to the same nonzero number. Specifically,
	\begin{equation}
		\sum_{n\geq1} \frac{\mu(n)}{n}D(nx) =c_1+O\left ( \frac{1}{\log x} \right ).
	\end{equation}
	This implies that the functions $\mu(n)$ 
	and $D(nx)$ are orthogonal (not correlated).  
\end{proof}

\begin{exa} \label{exa1300} {\normalfont The sequence $\{(n-2)/4 \text{ mod }4:n \geq 1\}$ and $\{\mu(n):n \geq 1\}$ are not orthogonal, (are correlated).

		\begin{tabular}{ l l}
			$ \displaystyle \sum_{n\geq1} \frac{\mu(n)}{n}D(nx) =c_1+O\left ( \frac{1}{\log x} \right )$
			
		\end{tabular}
	}
\end{exa}

The next series considered has an intrinsic link to the sequence of primes $p=n^2+1$. Some extra work is required to prove that the partial sum $\sum_{n\leq x} \Lambda(n^2+1)$ is unbounded as $x \to \infty$.

\begin{lem}  \label{lem1305} Let $\Lambda(n)$ be the vonMangoldt function, and let $D(x)$ be the saw tooth function. For any real number $x\in \R-\Z$, 
	\begin{equation}
		\sum_{n\geq1} \frac{\Lambda(n^2+1)}{n^2}D(n^2x)=-\frac{1}{\pi}\sum_{n\geq1} \frac{\sum_{d^2 \mid n} \Lambda(d^2+1)}{n}\sin(2\pi nx)
	\end{equation}	
	
\end{lem}
\begin{proof} Let 
	\begin{equation}
		a_n=\Lambda(n+1)\sum_{d|n}\lambda(d),
	\end{equation} and let
	\begin{equation}
		A_n=\sum_{d \mid n}\Lambda(d+1)\sum_{e|d}\lambda(e)=
		\sum_{d^2 \mid n} \Lambda(d^2+1).
	\end{equation}
	
	Then, the series 
	\begin{equation}
		\sum_{n\geq1} \frac{a_n}{n^s}D(nx)=\sum_{n\geq1} \frac{\Lambda(n+1)\sum_{d|n}\lambda(d)}{n^s}D(nx)=\sum_{n\geq1} \frac{\Lambda(n^2+1)}{n^{2s}}D(n^2x)
	\end{equation}
	is absolutely convergent for $\Re e(s)>1/2$. And the right side of the series
	\begin{equation}
		-\frac{1}{\pi}\sum_{n\geq1} \frac{A_n}{n^s}\sin(2\pi nx)=-\frac{1}{\pi}\sum_{n\geq1} \frac{\sum_{d^2 \mid n} \Lambda(d^2+1)}{n^s}\sin(2\pi nx)
	\end{equation}
	converges to a nonzero value if and only if $x \in \R-\Z$. Therefore, by Theorem \ref{thm1300}, left side converges to the same nonzero number.
\end{proof}

\section{Problems} \label{s775}
\begin{exe} {\normalfont Let $ t \in \Z$ be a fixed integer. Show that the Dirichlet series
		$$\sum_{n \geq 1}  \frac{\Lambda(n)\Lambda(n+t)}{n^s}$$
		is analytic on the half plane $\Re e(s)>1$, and has a pole at $s=1$ if and only if $t=2k$ is even.}
\end{exe}

\begin{exe} {\normalfont Improve the numerical value for the average density constant
		$$-\sum_{n \geq 1}  \frac{\mu(n) \log n}{n \varphi(n)}=0.42783980 \ldots.$$}
\end{exe}

\begin{exe} {\normalfont Determine an analytic relationship between the series 
		$$f(s)=\sum_{n \geq 1}  \frac{\mu(n) }{n^s \varphi(n)}\quad \quad \text{and}  \quad \quad f^{'}(s)=\sum_{n \geq 1}  \frac{\mu(n) \log n}{n^s \varphi(n)}$$
		suitable for numerical calculations of $f(1)$ and $f^{'}(1)$.}
\end{exe}

\begin{exe} {\normalfont Let $\chi $ be a character modulo $q \geq 3$, and let $ t \in \Z$ be a fixed integer. Compute a lower bound for the twisted correlation
		$$\sum_{n \leq x} \chi(n) \Lambda(n)\Lambda(n+t).$$}
\end{exe}

\begin{exe} {\normalfont Let $\chi $ be a character modulo $q \geq 3$, and let $ t \in \Z$ be a fixed integer. Determine the region of convergence of the Dirichlet series
		$$\sum_{n \geq 1}  \frac{\chi(n)\Lambda(n)\Lambda(n+t)}{n^s}.$$
		It is analytic on the half plane $\Re e(s)>1- \varepsilon$, with $\varepsilon>0$, and does it has a pole at $s=1$ if and only if $t=2k$ is even?}
\end{exe}

\begin{exe} {\normalfont Let $f $ be an arithmetic function, let $ x \geq 1$ be a large number, and let $[x]=x-\{x\}$ denotes the largest integer function. Show that
		$$\sum_{n\leq x}f \left( [x/n] \right ) =\sum_{n\leq x}f(n) f(n+1)\left ( \left [ \frac{x}{n} \right ]-\left [\frac{x}{n(n+1)} \right ] \right ).$$
	}
\end{exe}

\begin{exe} {\normalfont Prove or disprove the following: Let $f(n)=ax^2+bn+c\in \Z[x] $ be an irreducible polynomial of divisor $\tdiv(f)=1$. If there is at least one prime $p=f(n)$ for some $n \geq 1$, then the sequence $\{f(n): n \geq 1\}$ contains infinitely many primes.}
\end{exe}

\section{Arithmetic Functions Summation Formulas} \label{S2009} 

\begin{lem} \label{lem2009.707}{\normalfont \cite[Lemma 3.1]{VR1973}} Suppose that $f: \N\longrightarrow \C$ is a multiplicative function and nonnegative, and there is a number $\tau>0$ such that 
\begin{equation} \label{eq2009.702}
 \sum_{p \leq x}f(p)=\left( \tau+ o(1) \right)=\frac{x}{\log x}
\end{equation}
as $x \to \infty$. Then, 
\begin{equation}\label{eq2009.707}
\sum_{n \leq x}\mu^2(n)f(n)=\left( \frac{e^{-\gamma \tau}}{\Gamma(\tau)}+o(1) \right )\frac{x}{\log x}\prod_{p \leq x}\left( 1+ \frac{f(p)}{p} \right),
\end{equation}
where $\gamma>0$ is a Euler constant. 
\end{lem}

\section{Problems}\label{s2009}
\begin{exe} {\normalfont  Show that Mertens constant is given by the series$$B_1=\gamma+\sum_{n \geq 2} \frac{\mu(n)}{n} \log \zeta(n).$$}
\end{exe} 
\begin{exe} {\normalfont  Show that if $a_1>0$ and $a_n \to 0$ as $ n \to \infty$, then $$\sum_{n \geq 2} \frac{\mu(n)}{n} \log \left ( 1+a_n \right)>0.$$}
\end{exe}
\begin{exe} {\normalfont  Use the rapidly convergent property of the series linking the Euler and Mertens constants: $$B_1=\gamma-\sum_{p \geq 2} \sum_{n \geq 2} \frac{1}{np^{n}}$$ to prove or disprove that $B_1$ and $\gamma$ are linearly independent over the rational numbers $\Q$.}
\end{exe} 
\begin{exe} {\normalfont  Let $\alpha  \geq 0 $ be a real number. Evaluate the finite sum $$\sum_{n \geq 2} \frac{1}{n (\log n)^{1+\alpha}} .$$ }
\end{exe} 

\begin{exe} \label{exe2003.101}{\normalfont Let $x\geq 1$ be a large number. Show that
$$\prod_{p \leq x}\left( 1+ \frac{p\pm 1}{p^2} \right)=\prod_{p \leq x}\left( 1+ \frac{1}{p} \right)\left (\prod_{p \geq  2}\left( \pm 1\frac{1}{p(p+1)} \right) +O\left ( \frac{1}{ x\log x}\right ) \right ) .$$
}
\end{exe}

\subsection{Elementary proof of the prime number theorem}
\begin{exe} \label{exe8150.101}{\normalfont Let $x\geq 1$ be a large number. Use Cauchy theorem to evaluate the integral
$$
\psi_2(x)=\frac{1}{i2\pi} \int_C\frac{\zeta^{\prime \prime}(s)}{\zeta(s)}\frac{x^{s}}{s}ds,
$$
where $C$ is a curve on the complex plane to verify the explicit formula
$$\sum_{n \leq x}\Lambda_2(n)=2x\log x +\sum_{\rho}\frac{\zeta^{\prime \prime}(\rho)}{\zeta^{\prime}(\rho)}\frac{x^{\rho}}{\rho}+ \frac{\zeta^{ \prime}(0)}{\zeta(0)}+ 
\sum_{n\geq 1}\frac{\zeta^{ \prime}(-2n)}{\zeta(-2n)}\frac{x^{-2n-1}}{-2n-1}.$$
}
\end{exe}


\currfilename.\\

\end{document}